\documentclass[centertags,oneside,
12pt]{amsart}

\usepackage{amstext    }
\usepackage{amsthm    }
\usepackage{a4}
\usepackage[mathscr]{eucal}
\usepackage{mathrsfs}
\usepackage{hyperref}
\usepackage{charter}
\usepackage{amsmath}
\usepackage{amssymb}
\usepackage{amscd}

\numberwithin{equation}{section}

\newtheorem{theorem}{Theorem}[section]
\newtheorem{definition}[theorem]{Definition}
\newtheorem{proposition}[theorem]{Proposition}
\newtheorem{corollary}[theorem]{Corollary}
\newtheorem{lemma}[theorem]{Lemma}
\newtheorem{remark}[theorem]{Remark}

\newcommand{\cali}[1]{\mathscr{#1}}

\newcommand{\Leb}{{\rm Leb}}

\newcommand{\dist}{{\rm dist}}

\newcommand{\ddbar}{{\partial\overline\partial}}

\newcommand{\Cell}{{\rm Cell}}

\newcommand{\id}{{\rm id}}

\newcommand{\hol}{{\rm hol}}

\newcommand{\lof}{\mathop{\mathrm{{log^\star}}}\nolimits}

\renewcommand{\Re}{{\rm Re}}

\newcommand{\Ac}{\cali{A}}

\newcommand{\Cc}{\cali{C}}

\newcommand{\Fc}{\cali{F}}

\newcommand{\Lc}{\cali{L}}

\newcommand{\Uc}{\cali{U}}

\newcommand{\Sc}{\cali{S}}

\newcommand{\Nc}{\cali{N}}
\newcommand{\Ic}{\cali{I}}

\newcommand{\A}{\mathbb{A}}
\newcommand{\B}{\mathbb{B}}
\newcommand{\C}{\mathbb{C}}
\newcommand{\D}{\mathbb{D}}
\newcommand{\X}{\mathbb{X}}

\newcommand{\N}{\mathbb{N}}
\newcommand{\Z}{\mathbb{Z}}
\newcommand{\R}{\mathbb{R}}
\newcommand{\T}{\mathbb{T}}
\newcommand{\U}{\mathbb{U}}
\newcommand{\V}{\mathbb{V}}

\renewcommand{\P}{\mathbb{P}}

\newcommand{\Cf}{\mathfrak{C}}
\newcommand{\Df}{\mathfrak{D}}

\title[Integrability of holonomy cocycle]{Singular holomorphic  foliations by curves  I: Integrability of holonomy cocycle in dimension $2$}

\author{ Vi{\^e}t-Anh Nguy{\^e}n}

\dedicatory{Dedicated to  Professor  Nessim Sibony for his  seventieth birthday}

\begin{document}

\maketitle

\begin{abstract}
We study  the  holonomy  cocycle $\mathcal H$  of  a  holomorphic  foliation $\Fc$ by Riemann surfaces defined on a   compact complex projective surface $X$ satisfying the following  two conditions:

$\bullet$   its   singularities $E$ are all hyperbolic;

$\bullet$   there is  no holomorphic non-constant map $\C\to X$ such that out of $E$
the image of $\C$ is locally contained in a leaf.

Let $T$ be a harmonic current tangent to $\Fc$ which does not give mass to any  invariant  analytic curve. 
Using the   leafwise  Poincar\'e metric,
we show that   $\mathcal H$  is  integrable with respect to  $T.$  
Consequently,  we infer the  existence of the Lyapunov  exponent function of $T.$
  \end{abstract}

\bigskip

\bigskip

\noindent
{\bf Classification AMS 2010:} Primary: 37F75, 37A30;  Secondary:  57R30, 58J35, 58J65, 60J65.


\tableofcontents


 \section{Introduction} \label{intro}

\subsection{General settings and  main results}

The  dynamical and  geometric theory  of holomorphic  foliations by curves has
received much attention in the past few years.  The  holonomy cocycle (or equivalently, the normal derivative cocycle)  of a foliation  
is  a very important object  which  reflects
 dynamical  as well as  geometric and analytic aspects    of  the  foliation.
 Exploring  this  object  allows  us  to understand  more  about the  foliation itself.
Let $\Fc=(X,\Lc,E)$ be  a     holomorphic  foliation by hyperbolic Riemann surfaces    which is  immersed   onto an ambient complex manifold $X$  and  which  possesses    the set of singularities $E.$ 
On the  geometric  side,   we have  {\it harmonic   currents} $T$ which are  generalizations of the  {\it foliations cycles}
introduced by  Sullivan  \cite{Sullivan}. On the  dynamical  side, the  sample-path space $\Omega$ associated  to the foliation  describes the {\it leafwise Brownian motion} with respect to the Poincar\'e  metric  
on   leaves.  This  motion  generates a Markov process on  $X.$

Assume  for the  moment that $\Fc$   does not possess any singularities  (i.e. $E=\varnothing$). Let $T$ be  
a  harmonic  current  tangent to  $\Fc.$
When  $X$ is  a surface, i.e.  $\dim X=2,$    we can define  the unique 
{\it Lyapunov  exponent function}  of  $T,$ which is  leafwise constant and which   measures  heuristically   the exponential  rate of  convergence  of leaves toward each other  along  leafwise Brownian trajectories
 (see Candel  \cite{Candel},  Deroin \cite{Deroin05}). When $\dim X\geq 2,$    our recent work  in \cite{NguyenVietAnh1}
provides $(\dim X -1)$ Lyapunov exponent functions whose  geometric  characterizations in terms of geodesic
rays have been investigated in  \cite{NguyenVietAnh2}.

Since  the main examples of holomorphic foliations by curves 
are  those  in the complex projective space $\P^k$ of arbitrary
dimension (in which case there are always singularities) or in algebraic manifolds,
  the following fundamental  question  arises  naturally: 

\smallskip

\noindent {\bf Question. }{\it Can one  define  the  Lyapunov exponent functions of 
a harmonic  current $T$ tangent to a  singular holomorphic hyperbolic  foliation $\Fc=(X,\Lc,E)$ ?}

\smallskip

The main purpose of  this  paper is to give an affirmative  answer  to this  question for generic  foliations, that is,   when 
the   ambient manifold $X$  is  a compact complex projective  surface, the foliation  enjoys  Brody hyperbolicity  (see Definition \ref{D:Brody} below),
and  $E$  is the set of singularities  which  are of hyperbolic type.

 Here is  our   main result. The   new  terminology and  notation appearing in this  theorem   will be  explained  in Section
  \ref{S:background} below.
 
\begin{theorem}\label{thm_main}  
Let $\Fc=(X,\Lc,E)$    be  
a   holomorphic Brody hyperbolic foliation   with     hyperbolic  singularities $E$ in a   Hermitian compact complex projective  surface $X.$
Let $\mathcal H$ be the holonomy cocycle of the foliation. Let $T$ be  a harmonic  current  tangent  to $\Fc$  which does not give mass to any  invariant  analytic curve.
Consider   the  corresponding harmonic  measure
 $\mu:= T\wedge  g_P$  where $g_P$ is the leafwise  Poincar\'e metric. 
Let $\Omega$  be the sample-path space associated with $\Fc.$ 
 Let $\bar\mu$ be the   invariant   measure on  $\Omega$  associated  with $\mu.$  Consider the function $\Ic:\ \Omega\to \R^+$ defined by
$$ \Ic(\omega):=\sup_{t\in[0,1]}
     |\log \|\mathcal{H}(\omega,t)\| |  ,\qquad \omega\in \Omega.
$$
 Then
$\Ic$ is $\bar\mu$-integrable.
\end{theorem}
Here   is an immediate  consequence of this theorem.

\begin{corollary}\label{C:extremal}
 Under the hypotheses and notation of Theorem \ref{thm_main}, assume in addition that
the measure $\mu$ is  ergodic.    Then  $T$ admits
the (unique) Lyapunov exponent $\lambda(T)$  given by the  formula
$$
\lambda(T):= \int_\Omega \log \|\mathcal{H}(\omega,1)\| d\bar\mu(\omega).  
$$
Moreover,
  for  $\mu$-almost  every $x\in X,$  we have
 $$
 \lim\limits_{t\to \infty} {1\over  t} \log  \| \mathcal H(\omega,t) \|=\lambda(T)  
 $$
  for    almost every path  $\omega\in\Omega$ with respect to the Wiener measure  at $x$
  which lives on the leaf passing through $x.$
\end{corollary}
For comprehensive  expositions on characterization of Lyapunov exponents using the Wiener measures,
see \cite{Candel, Deroin13,NguyenVietAnh1,NguyenVietAnh2}.  
 In Theorem  \ref{thm_harmonic_currents_vs_measures} below,  we will see that the  measure $\mu$ is  ergodic
when, for  example, the current $T$ is an  extremal point in the convex cone of all harmonic currents tangent to $\Fc.$

%

 Consider     
a   singular  foliation by curves  $\Fc=(\P^2,\Lc,E)$ on the complex projective plane $\P^2$
 such that 
all the singularities of $\Fc$   are hyperbolic and  that  $\Fc$ has no invariant algebraic  curve.
 Combining some results by   Glutsyuk \cite{Glutsyuk} and by Lins Neto  \cite{Neto},
and   by Brunella  \cite{Brunella}, we will see in Remark  \ref{R:Brody_sufficiency}  and in the  discussion after Theorem \ref{T:NetoSoares} below
that 
 $\Fc$ is   Brody  hyperbolic.  Moreover,
 the  unique  ergodicity theorem of Forn\ae ss-Sibony  \cite{FornaessSibony3} says that the harmonic current $T$ is unique
up to a multiplicative constant. In particular, the convex cone of all harmonic  currents of $\Fc$ is just  a real half-line,
and hence all these  currents  are extremal
(see the  discussion preceding Theorem \ref{thm_harmonic_currents_vs_measures} below). Therefore,   the measure $T\wedge g_P$ is ergodic  by Part 2) of  this theorem.
Consequently, Corollary  \ref{C:extremal} applies and   gives us the following result.
It can be   applied  to  every generic  foliation in $\P^2$ with a given degree $d>1.$
  
\begin{corollary}\label{cor_th_main}
 Let $\Fc=(\P^2,\Lc,E)$    be  
a   singular  foliation by curves on the complex projective plane $\P^2.$ Assume that
all the singularities  are hyperbolic and  that  $\Fc$ has no invariant algebraic  curve. 
Let $T$ be  the unique  harmonic   current  tangent to $\Fc$  such that  $\mu:=T\wedge g_P$ is a probability measure.
  Let $\mathcal H,$  
$\bar\mu$ and $\Ic$  be as in the  statement of  Theorem \ref{thm_main}. Then the conclusion of this theorem
as  well as  that  of   Corollary \ref{C:extremal} hold.
In particular, $\Fc$ admits  a unique  Lyapunov  exponent.
\end{corollary}  
  
The novelty of  the last corollary is  that the   (unique)  Lyapunov  exponent of such a foliation $\Fc$ is  intrinsic and  canonical.

 In fact,  we will prove the following more  complete version of Theorem \ref{thm_main} where we  introduce
the  so-called {\it integrability  condition}.
  \begin{theorem}\label{thm_main_2}
    Let $\Fc=(X,\Lc,E)$    be  
a   holomorphic Brody hyperbolic foliation   with      hyperbolic singularities $E$ in a   compact complex projective  surface $X.$
  Let $T$ be  a harmonic  current  tangent  to $\Fc$  which does not give mass to any  invariant  analytic curve.
Then  we have  {\rm
\begin{equation}\label{e:necessary_integrability}
 \text{(the integrability  condition):}\qquad \int_X| \log \dist(x,E)| \cdot (T\wedge g_P)(x)<\infty.
 \end{equation}
 }
\end{theorem}

Using the  Poincar\'e metric of the  punctured disc as  a  local model and Lemma  \ref{lem_poincare}  below, we  can prove that if the harmonic  current $T$ has 
a positive mass on an invariant analytic  curve, then the  integral in \eqref{e:necessary_integrability} is  infinite, i.e., the integrability  condition
breaks down.

The  condition  of Brody hyperbolicity  seems to be indispensable  for the   integrability of the  holonomy cocycle. Indeed,
in a very recent work \cite[Theorem A]{Hussenot} Hussenot discovers  the  following  remarkable property
 for  a class of  Ricatti foliations $\Fc$ on $\P^2.$
For every $x\in\P^2$ outside invariant curves of every foliation in this class, it holds that
   $$
 \limsup\limits_{t\to \infty} {1\over  t} \log  \| \mathcal H(\omega,t) \|= \infty 
 $$ for almost every path $\omega\in\Omega_x$ with respect to the Wiener measure at $x$ which lives
 on the leaf passing through $x.$ By     Glutsyuk \cite{Glutsyuk} and   Lins Neto   \cite{Neto94}, these foliations  are hyperbolic since all their  singular points    have nondegenerate linear part. Nevertheless, neither of them is Brody hyperbolic because  they all contain integral curves which are some images of $\P^1$  (see  Remark  \ref{R:Brody_sufficiency} below). 
\subsection{Outline of the proofs}\label{SS:outline}
Now   we  discuss the  method  of the proof of Theorems  \ref{thm_main} and  \ref{thm_main_2}.
Our approach consists  of  two main steps. 

In the  first main step  we show that    Theorem \ref{thm_main}  follows from  Theorem  \ref{thm_main_2}, i.e., from
the integrability condition  \eqref{e:necessary_integrability}. To this end  we  study  the behavior of the holonomy cocycle
near the  singularities  with respect to  the leafwise Poincar\'e metric.  
Let $g_X$ be  a Hermitian  metric on $X$
 and let $\dist$ denote the  distance  on $X$ induced by $g_X.$ Roughly speaking, this step
 quantifies  the  expansion speed of the hololomy  cocycle in terms
 of the ambient metric $g_X$ when  one travels  along   unit-speed geodesic rays.
 The main ingredients  are in our joint-works with  Dinh and Sibony in  \cite{DinhNguyenSibony1,DinhNguyenSibony2,DinhNguyenSibony3}.

  The  second main step is  then devoted to  the proof of Theorem \ref{thm_main_2}, i.e., inequality \eqref{e:necessary_integrability}.
  The main difficulty is that  known  estimates (see, for example,  \cite{DinhNguyenSibony1})
 on the behavior of $T$
near  linearizable  singularities,    only  give a  weaker  inequality
  \begin{equation} \label{e:known_estimate}
\int_X | \log \dist(x,E)|^{1-\delta}\cdot( T\wedge g_P)(x)<\infty, \qquad\forall \delta>0.
  \end{equation}
  So \eqref{e:necessary_integrability} is  the limiting case of  \eqref{e:known_estimate}.
  The proof of \eqref{e:known_estimate} in \cite[Proposition 4.2]{DinhNguyenSibony1} relies on  the finiteness of the Lelong number of $T$  at
  every   point. Recall from Skoda \cite{Skoda} that the Lelong number of a harmonic  current at a given point  is an important indicator measuring the  mass-density
  of the current at that point.
 Moreover, our result in \cite{NguyenVietAnh3} (see also a recent result of Dinh-Sibony \cite{DinhSibony15}) sharpens the last estimate
by showing that the Lelong number  of $T$ vanishes at every   hyperbolic singular point $x\in E.$  
Nevertheless, even this better estimate does not suffice to prove \eqref{e:necessary_integrability}.

  To overcome this  obstacle, we use  a cohomological idea
  which  exploits  fully  the assumption that $X$ is  projective.
  This  assumption  imposes a stronger mass-clustering  condition on harmonic  currents.
  
\smallskip

Now  we  explain  briefly  our proof of  the integrability condition  \eqref{e:necessary_integrability}.
Our approach is  based on a cohomological  invariance (see Proposition \ref{P:coho_mass}) which  says roughly that
if two algebraic  curves $\Cf $ and $\Df$ on $X$ are cohomologous (for example,  if they have the same algebraic  degree  when $X=\P^2$), then  under  suitable  assumptions, we can define the wedge-product
$T\wedge[\Cf],$ $T\wedge [\Df]$   which are finite positive Borel measures and their masses are equal, i.e,
\begin{equation}\label{e:coho_inv_intro}
\int_X T\wedge[\Cf] =\int_X T\wedge [\Df].
\end{equation}
Before  going further,  let us   explain     why  equality \eqref{e:coho_inv_intro} could be  true.
Since $\Cf $ and $\Df$ on $X$ are cohomologous on $X$, the  $\partial\overline\partial$-lemma for compact K\"ahler manifolds
provides us an integrable  function  $u$ on $X$  such that
$$
[\Cf]-[\Df]=i \partial\overline\partial u\qquad \text{in the sense of currents.} 
$$
So  we can write
$$
\int_X T\wedge[\Cf] -\int_X T\wedge [\Df]=\int_X T\wedge i \partial\overline\partial u.
$$
The function $u$  is, in general, not smooth near $\Cf$ and $\Df.$ However, if we  could consider it like a smooth function,
Stokes' theorem would turn the  right hand side of the last line into the following integral
$$
\int_X u (i \partial\overline\partial T ) =0,
$$
where the  last equality holds since  the harmonicity of $T$ implies that $i \partial\overline\partial T=0.$
Therefore, we may expect equality \eqref{e:coho_inv_intro} to hold.

 Resuming  the sketchy proof of  the integrability condition  \eqref{e:necessary_integrability}, let $x_0\in E$ and  fix a  coordinate system $(z,w)$ around $x_0$ such that
the two separatrices of the hyperbolic singular point $x_0$  are $\{z=0\}$ and $\{w=0\}.$
Then we can show that the vanishing of the Lelong number of $T$ at  $0$  established in \cite{NguyenVietAnh3} is  equivalent to 
the following convergence 
\begin{equation}\label{e:reduction_intro}
\int_{\B(0,r)}  T\wedge[z=r] \to  0 \qquad\text{as}\qquad r\to 0,
\end{equation}
where   $\B(0,r)$ is the ball in $X$ with center $x_0=0$ and radius $r.$
And more importantly, 
  the integrability  condition \eqref{e:necessary_integrability}  is  somehow  equivalent to the statement that
the  convergence \eqref{e:reduction_intro}
has,  in a certain very  weak sense, a speed   of order $|\log{ r}|^{-\delta}$  as  $r\to 0$ for some $\delta>0.$
 Note, however, that this speed   does  not at all mean that $\int_{\B(0,r)}  T\wedge[z=r]=O(|\log{ r}|^{-\delta}).$ For a precise meaning of this  speed, see  Remark \ref{R:speed} below.

Now suppose for  the sake of simplicity that $X=\P^2$ and $N\in\N$ is  large  enough. We  choose an algebraic curve $\Cf$ of degree $N$ which looks like the analytic curve $\{z=w^N\}$ near $0.$
 We also   choose an algebraic curve $\Df$ of degree $N$ which looks like the analytic curve $\{r=z-w^N\}$ near $0.$  The following  seven
  observations play  a key role in our approach, where  $0<\delta<1$ is   an exponent independent of $r$
and $N,$ $0<r<r_0$ with $r_0>0$ a fixed small number.

$(i)$   Outside  a  small ball  $\B(0,r_0),$
  the analytic curve $\{z=w^N\}$  (and hence  the algebraic curve $\Cf$) falls into a tubular neighborhood
with size $O(r^\rho)$  of
  the analytic curve $\{r=z-w^N\}$ (and hence the algebraic curve $\Df$), where $\rho$ is  a real number depending on $N$
 with $0<\rho\leq 1.$ So we may expect
  $$
  \int_{X\setminus\B(0,r_0)}  T\wedge[\Cf] =\int_{X\setminus\B(0,r_0)} T\wedge [\Df]+O(r^\rho).
  $$
 
$(ii)$ Outside 
  the ball   $\B(0,r^{1/N}|\log{r}|^{3/N})$ and inside the small ball  $\B(0,r_0),$  
 the   analytic  curve $\{r=z-w^N\}$ (and hence the  algebraic curve $\Df$)  behaves like  the analytic curve $\{z=w^N\}$  (and hence the  algebraic curve $\Cf$) while intersecting the two curves
 with a general leaf.  Indeed, when $|w|\geq r^{1/N}|\log{r}|^{3/N},$  
 we have $r\ll |w|^N.$ So we may expect
  $$
  \int_{ \B(0,r_0)\setminus \B(0,r^{1/N}|\log{r}|^{3/N})} T\wedge [\Cf]=
\int_{ \B(0,r_0)\setminus \B(0,r^{1/N}|\log{r}|^{3/N})} T\wedge [\Df]
+O(|\log r|^{-\delta}).
  $$

 $(iii)$ The corona
     $\A_{r,N}:=\B(0,r^{1/N}|\log{r}|^{3/N})\setminus \B(0,r^{1/N}|\log{r}|^{-3/N}) $
     is,  in some  sense,   small and it may be considered as  negligible.
 So we may expect
$$
  \int_{ \A_{r,N}} T\wedge [\Cf]=O((\log r)^{-\delta})\quad\text{and}\quad 
\int_{  \A_{r,N} } T\wedge [\Df]
=O(|\log r|^{-\delta}).
  $$
  
 $(iv)$ Our next  observation is   the following  partition of $X$  for $0<r\ll 1:$
\begin{equation*}
X=\big(X\setminus\B(0,r_0)\big)\coprod \big(\B(0,r_0)\setminus   \B(0,r^{1/N}|\log r|^{3/N} )\big)\coprod \A_{r,N}\coprod  
\B(0,r^{1/N}|\log r|^{-3/N}).
\end{equation*}
This  allows us to decompose both integrals of \eqref{e:coho_inv_intro}
 into corresponding  pieces. 
 
Consequently, when the degree $N$ is  sufficiently high, by taking into account the   observations $(i)$-$(ii)$-$(iii)$-$(iv)$,
 and using  \eqref{e:coho_inv_intro},  we  see that
 $$
  \int_{\B(0,r^{1/N}|\log{r}|^{-3/N}) } T\wedge [\Cf]-
\int_{ \B(0,r^{1/N}|\log{r}|^{-3/N})  } T\wedge [\Df]
=O(|\log r|^{-\delta}).
  $$

$(v)$ Inside 
  the ball   $\B(0,r^{1/N}|\log{r}|^{-3/N}),$ the analytic  curve  $\{z=w^N\}$ (and hence the algebraic curve  $\Cf$)    clusters around $0,$ in a  certain sense,  much more often 
than the analytic  curve  $\{z=r\}$ (and hence the algebraic curve  $\Df$). Indeed,     we see   in the equation $z=w^N$ that both $z$ and $w$ can tend to $0,$  whereas in the  equation $z=r,$
only $w$ could tend to $0.$ So we may expect that  in a certain sense,  
$$
     \int_{\B(0,r^{1/N}|\log{r}|^{-3/N}) } T\wedge [\Df] \ll
\int_{ \B(0,r^{1/N}|\log{r}|^{-3/N})  } T\wedge [\Cf]
 .
  $$
 This, combined  with the  estimate  obtained just before $(v)$,  implies that both integrals
 $$
 \int_{\B(0,r^{1/N}|\log{r}|^{-3/N}) } T\wedge [\Cf] \quad\text{and}\quad
\int_{ \B(0,r^{1/N}|\log{r}|^{-3/N})  } T\wedge [\Df]
 $$
 admit, in a certain sense, a speed   of order $|\log{ r}|^{-\delta}.$
   
$(vi)$ Inside 
  the ball   $\B(0,r^{1/N}|\log{r}|^{-3/N}),$ 
 the   analytic  curve $\{r=z-w^N\}$ (and hence the  algebraic curve $\Df$)  behaves like  the analytic curve $\{z=r\}$ while intersecting the two curves
 with a general leaf. Indeed, when $|w|\leq r^{1/N}|\log{r}|^{-3/N},$  
 we have $|w|^N\ll r.$
 So we may expect
$$
  \int_{\B(0,r^{1/N}|\log{r}|^{-3/N}) } T\wedge [\Df]-
\int_{ \B(0,r^{1/N}|\log{r}|^{-3/N})  } T\wedge [z=r]
=O(|\log r|^{-\delta}).
  $$
 This, together with the  estimate just obtained before $(vi),$ 
yields  that
 $$
  \int_{ \B(0,r^{1/N}|\log{r}|^{-3/N})  } T\wedge [z=r]
  $$
has, in a certain sense, a speed   of order $|\log{ r}|^{-\delta}.$ 
 
$(vii)$ Our last observation is that one can show  that there is  a constant $c_N>1$ independent of $r$  such that
$$ c^{-1}_N \int_{\B(0,r^{1/N})}  T\wedge[z=r]\leq \int_{\B(0,r)}  T\wedge[z=r]\leq  c_N \int_{\B(0,r^{1/N})}  T\wedge[z=r].$$
 This, together with the  estimate just obtained before $(vii),$ 
implies  that
 $$
  \int_{ \B(0,r )  } T\wedge [z=r]
  $$
 admits, in a certain sense, a speed   of order $|\log{ r}|^{-\delta}.$
 Hence, we get the convergence with speed of  \eqref{e:reduction_intro}. This is  what  we  are looking for.

In fact, the factor  $|\log{r}|^{3/N}$ appearing in the above  observations  comes  from the degeneration of the 
Poincar\'e metric $g_P$ relative to the ambient metric $g_X$  (see formula \eqref{eq_relation_Poincare_Hermitian_metrics}).
Moreover, the larger  the degree  $N$ is,
the more  evident  the  mass-clustering phenomenon in the previous observation  becomes.  

Our approach underlines several  tasks.  On the one hand, we need to define a geometric intersection
of a harmonic  current with a  singular analytic curve defined on a neighborhood of a singular point of the foliation.
On the other hand, we need to approximate  some (local) analytic curves by global algebraic ones.   
The assumption of projectivity  of $X$ is needed in order to ensure a good supply of algebraic  curves.
 
\subsection{Organization of the  article  and  acknowledgments}
The article  is  organized as  follows.

\smallskip

In   Section \ref{S:background}   below  we  set up the  background 
and 
prepare the auxiliary results. Some  basic facts extracted from  \cite{DinhNguyenSibony1,DinhNguyenSibony2,DinhNguyenSibony3}
 about  the behavior of the leafwise Poincar\'e metric
near   the singularities are recalled here.
    A quick  discussion  on  the heat diffusions as well as the
   measure  theory on sample-path spaces and  the holonomy cocycles  will also be given in  Section \ref{S:background}.
     On the   other hand, Section \ref{S:holonomy_cocycle}  is  devoted to an  analytic  study  on the  holonomy cocycles.
     The  proofs 
 of    Theorem \ref{thm_main}  and  Corollary \ref{C:extremal}  (modulo the integrability  condition \eqref{e:necessary_integrability}, i.e., Theorem  \ref{thm_main_2})  will be  provided   in Section \ref{section_Main_Theorem_1}.
     
  The remainder of the  article  is then  devoted to the proof of  inequality \eqref{e:necessary_integrability}.
  This can be  done  in three reduction steps.  

Section \ref{S:parametrization} 
 collects several  recent results  about the mass-clustering  of harmonic currents and a  special parametrization of leaves
near  hyperbolic  singularities.  

The first reduction is carried out in  Section  \ref{S:Main_Theorem_2}. Namely,
the proof of   the integrability  condition  \eqref{e:necessary_integrability} is reduced to that of  Theorem \ref{T:main_estimate}.



Section \ref{S:intersection_interpretation}
lays the background for the geometric intersection
of a harmonic current with an analytic curve defined on an open subset of $X.$ 
We are inspired by  Forn\ae ss-Sibony's recent works in 
  \cite{FornaessSibony1,FornaessSibony2,FornaessSibony3}.
Special attention is  focused on the case where the analytic curve is defined on  a neighborhood
of a singular point of the foliation. We also  introduce  the notion of interpretations:
a way which permits us to estimate the mass of a geometric intersection  efficiently.
 
 In Section \ref{S:test_curves} we introduce  test curves which consist of algebraic  curves and analytic ones.
 The former curves are defined globally on $X,$ whereas the latter ones are only defined on 
 a neighborhood of a singular point of the foliation.
 Next, we state  the first collection of basic estimates 
(see Propositions \ref{P:comparison}, \ref{P:comparison_bis}, 
 \ref{P:separatrices})
regarding  the mass estimates of the geometric intersection
of  a harmonic current with test curves.
This allows us to reduce  the proof of    Theorem \ref{T:main_estimate} to those  of 
Propositions
 \ref{P:separatrices} and  \ref{P:key_interpretations}   modulo Propositions \ref{P:comparison}, \ref{P:comparison_bis}. This is  the  second reduction.

 Section \ref{S:coho}
  states the second     collection of   basic estimates (see Propositions \ref{P:mass_1}
and \ref{P:mass_2}). Next, using  these estimate we establish a cohomological
invariance result (see Proposition \ref{P:coho_mass}) which permits us to prove Proposition
 \ref{P:separatrices}. So  modulo Propositions
 \ref{P:comparison}, \ref{P:comparison_bis},  \ref{P:mass_1}, \ref{P:mass_2}, the proof of    Theorem \ref{T:main_estimate} is finally reduced to that  of 
Proposition \ref{P:key_interpretations}.
 This is  the  last reduction.
 
   In Section \ref{S:intersection_with_leaves} we study how the  intersection points  of test curves with a general  leaf near singularities
 distribute. This  analysis will be helpful when we want to estimate
the mass of some geometric intersections in terms of interpretations.
Based on this  analysis,
 the remaining sections  are then devoted to the proof of the above  basic estimates (Propositions
 \ref{P:comparison}, \ref{P:comparison_bis},  \ref{P:mass_1}, \ref{P:mass_2}
 and  \ref{P:key_interpretations}).

  Section \ref{S:balls} establishes Proposition \ref{P:comparison} and  the first half of Proposition  \ref{P:key_interpretations}.
  
 Section \ref{S:outside_corona}
  is devoted to the proof of  Proposition \ref{P:mass_1}.

   Proposition \ref{P:mass_2} which consists of 3 basic estimates is  proved in Section  \ref{S:on_corona}.
The proof of each estimate  occupies a  whole subsection.

Finally, Section \ref{S:completion}  completes the proof of the last half part of  Proposition \ref{P:key_interpretations} as well as 
 the proof of   Proposition \ref{P:comparison_bis}.

 
\smallskip

\noindent
{\bf Acknowledgments. } I would like to thank    Nessim  Sibony  for suggesting me to work on
the holonomy cocycle. 
 Sincere thanks also go to  Tien-Cuong Dinh   
  for  interesting  discussions.  I am  also  grateful to  Mihai Paun  and Jaigyoung Choe for very kind  help, and 
   to  the referee for carefully reading the paper and for suggestions leading to the improvement of
the exposition. 
The paper was partially prepared 
during my visit  at  Vietnam  Institute for Advanced Study in Mathematics (VIASM) and 
at the Center for Mathematical Challenges (CMC) of the Korea Institute for Advanced Study (KIAS).
  I would like to express my gratitude to these organizations for hospitality and  for  financial support.  

 
\section{Background}\label{S:background}


  Although  the main theorems  only deal  with   complex  surfaces as   the ambient manifold $X,$  we consider, in this  section, 
the  general case where  $\dim X\geq 2.$ Indeed, the  section   
  may  serve  as the background for the  ongoing   parts of the article.
For a recent account  on the     theory of foliations,  the  reader is  invited  to consult  the  survey articles  by   Forn\ae ss-Sibony \cite{FornaessSibony2},
 Ghys \cite{Ghys}, Hurder \cite{Hurder} and  textbooks by Candel-Conlon, Walczak \cite{CandelConlon1, CandelConlon2, Walczak}.  
 
\noindent {\bf Notation.} Throughout the  article, we denote by $\D$   the unit disc in $\C.$
For $r>0$  we denote by $\D_r$ and $r\D$ interchangeably the disc in $\C$ with center  $0$  and
with radius $r.$ 
 We use  several notions of distances:

$\bullet$  $\dist$ denotes the distance on $X$ induced by the Hermitian metric $g_X.$

$\bullet$
$\dist_P$ denotes the Poincar\'e metric, it will be defined in Subsection \ref{SS:background}, whereas a more elaborate variant of this distance will be considered in Section \ref{S:holonomy_cocycle}.

$\bullet$
$\dist_C$ denotes the compatible  pseudo-distance, it will be  defined in  
Section \ref{S:intersection_with_leaves}.
   
   The current of integration on an analytic curve $\Cf$ 
  is denoted by $[\Cf].$  
 
  In this  work the letters $c,$ $c',$ $c_0,$  $c_1,$ $c_2$ etc. denote  positive constants, not necessarily the same at each  occurrence.  The notation $\gtrsim$ and $\lesssim$ means inequalities  up to  a  multiplicative constant, whereas  we  write  $\approx$ when  both inequalities  are satisfied.
 Let $O$ and $o$ denote   the usual Landau asymptotic  notations. 
Let $\lof(\cdot):=1+|\log(\cdot)|$ be a log-type function. 
 

\subsection{Foliations, singularities,  Poincar\'e metric and Brody hyperbolicity} \label{SS:background} 

Let $X$ be a complex  manifold of  dimension $k.$  A {\it   holomorphic foliation by curves}   $\Fc=(X,\Lc)$ on $X$ 
 is  the  data of  a {\it foliation  atlas} $\Lc$ with charts 
$$\Phi_p:\U_p\rightarrow \B_p\times \T_p.$$
Here, $\T_p$ is a domain in $\C^{k-1},$ $\B_p$ is a domain in $\C$,  $\U_p$ is  a  domain in 
$X,$ and  
$\Phi_p$ is  biholomorphic,  and  all the changes of coordinates $\Phi_p\circ\Phi_q^{-1}$ are of the form
$$x=(y,t)\mapsto x'=(y',t'), \quad y'=\Psi(y,t),\quad t'=\Lambda(t).$$

The open set $\U_p$ is called a {\it flow
  box} and the Riemann surface $\Phi_p^{-1}\{t=c\}$ in $\U_p$ with $c\in\T_p$ is a {\it
  plaque}. The property of the above coordinate changes insures that
the plaques in different flow boxes are compatible in the intersection of
the boxes. Two plaques are {\it adjacent} if they have non-empty intersection.
 
A {\it leaf} $L$ is a minimal connected subset of $X$ such
that if $L$ intersects a plaque, it contains that plaque. So a leaf $L$
is a  Riemann surface  immersed in $X$ which is a
union of plaques. A leaf through a point $x$ of this foliation is often denoted by $L_x.$
  A {\it transversal} is  a complex submanifold of codimension  1 in $X$  which is  transverse to  the leaves  of $\Fc.$

A {\it    holomorphic foliation by  curves with singularities},    or equivalently a {\it   singular  holomorphic foliation by  curves,} is  the data
$(X,\Lc,E)$, where $X$ is a complex  manifold, $E$ a closed
subset of $X$ and $(X\setminus E,\Lc)$ is a  holomorphic  foliation by curves. Each point in $E$ is  said to be  a {\it  singular point,} and   $E$ is said to be {\it the  set of singularities} of the foliation. We always  assume that $\overline{X\setminus E}=X$, see  e.g. \cite{DinhNguyenSibony1, FornaessSibony1}  for more details.
If  $X$ is  compact, then we say that the foliation $(X,\Lc,E)$ is  {\it compact.}

We say that a vector field $F$ on $\C^k$ is {\it generic linear} if
it can be written as 
$$F(z)=\sum_{j=1}^k \lambda_j z_j {\partial\over \partial z_j}$$
where $\lambda_j$ are non-zero complex numbers.
The integral curves of $F$ define a  foliation on
$\C^k$. 
The condition $\lambda_j\not=0$ implies that the foliation   
has an isolated singularity at 0. Consider a   holomorphic foliation by  curves   $\Fc=(X,\Lc,E)$ with  a discrete set of 
singularities $E.$ We say that a
singular point $x\in E$ is {\it linearizable} if 
there is a (local) holomorphic coordinate system of $M$ on an open neighborhood $\U_x$ of $x$ on which
$x$ is identified with  $0\in \C^k$ and 
 the
leaves of $\Fc$ are integral curves of a generic linear vector field. 
Such   neighborhood  $\U_x$ is  called 
a {\it singular flow box} of  $x.$ When  $\dim X=k=2,$  we say that a
linearizable singular point $x\in E$ is {\it hyperbolic} if 
 the  associated  generic linear vector field $F(z) =\lambda_1 z_1 {\partial\over \partial z_1}
+ \lambda_2 z_2{\partial\over \partial z_2}$ satisfies $\lambda_1/\lambda_2\not\in\R.$ This property  is independent of the choice of  coordinates.  

 For the  sake of  simplicity, we  adopt the  following terminology 
throughout the  article:  Unless  otherwise specified, 
a  {\it foliation } means exactly    a {\it  holomorphic   foliation  by curves $\Fc=(X,\Lc,E)$   in a Hermitian  complex manifold
$(X,g_X)$   with a   (eventually  empty)  set $E$ of  singularities.}

Let $\Fc=(X,\Lc,E)$ be  a foliation.
We denote by  $\Cc_\Fc$ the  sheaf of functions  $f$ defined and  compactly supported  on $X\setminus E$ which are  leafwise  smooth
 and transversally continuous, that is,  
 for each  foliation   chart
$\Phi_p:\U_p\rightarrow \B_p\times \T_p$ and all $m,n\in\N,$
 the derivatives  ${\partial^{m+n}(f\circ\Phi_p^{-1})\over  \partial y^m\partial\bar{y}^n}$ exist and are continuous in $(y,t).$ 
 
 Let $g_P$ be  the Poincar\'e metric 
on the unit disc  $\D,$  defined  by
$$ g_P(\zeta):={2\over (1-|\zeta|^2)^2} i d\zeta\wedge d\overline\zeta,\qquad\zeta\in\D,\quad \text{where}\ i:=\sqrt{-1}.  $$
  
A leaf $L$  of the  foliation is  said to be  {\it hyperbolic} if
it  is a   hyperbolic  Riemann  surface, i.e., it is  uniformized   by 
$\D.$   
 For a hyperbolic  leaf $L_x,$ 
 let $\phi_x:\D\to L_x$ be a universal covering map  with $\phi_x(0)=x$. Note that $\phi_x$ is  unique  up to  a  rotation  around $0\in \D.$ Then, by pushing   forward  the Poincar\'e metric $g_P$
on $\D$  
  via $\phi_x,$ we obtain the  so-called {\it Poincar\'e metric} on $L_x$ which depends only on the leaf.  The latter metric is given by a positive $(1,1)$-form on $L_x$  that we also denote by $g_P$ for the sake of simplicity.
The   foliation   is  said to be {\it hyperbolic} if  
  its leaves   are all  hyperbolic. 

 For simplicity  we still denote by   $g_X$  the  Hermitian  metric on leaves of the foliation $(X\setminus E,\Lc)$  induced by  the ambient Hermitian metric  $g_X.$
  Consider the  function $\eta:\ X\setminus E\to [0,\infty]$ defined by
  $$
  \eta(x):=\sup\left\lbrace \|D\phi(0)\| :\ \phi:\ \D\to L_x\ \text{holomorphic such that}\ \phi(0)=x  \right\rbrace.
  $$
  Here, for the  norm of the  differential $D\phi$ we use the Poincar\'e metric on $\D$ and the Hermitian metric
  $g_X$ on $L_x.$
  We  record  the following relation  between  $g_X$   and  the Poincar\'e metric $g_P$ on leaves
\begin{equation}
\label{eq_relation_Poincare_Hermitian_metrics}
g_X=\eta^2 g_P.
\end{equation}

Recall from  a recent  joint-work with Dinh and Sibony   \cite{DinhNguyenSibony3} the following notion. 

\begin{definition}\label{D:Brody}\rm 
A   foliation $\Fc=(X,\Lc,E)$   is  said  to be {\it Brody hyperbolic}  if  there is  a constant  $c>0$  such that
$\eta(x) \leq  c$ 
for all $x\in X\setminus E.$ 
\end{definition}

 \begin{remark}\rm  \label{R:Brody_sufficiency}
It is clear that if the foliation is Brody hyperbolic then it is hyperbolic.
Moreover, when  $X$ is  compact, the  Brody hyperbolicity is  equivalent to 
the non-existence  of  holomorphic non-constant maps
$\C\rightarrow X$  such that out of  $E$ the image of $\C$ is  locally contained in a leaf, 
see \cite[Theorem 15]{FornaessSibony2}.  
\end{remark}

The following  result   is  due to  Lins Neto  and Soares  \cite{NetoSoares} (we only give  the two-dimensional version although their result is
also valid in $\P^k$):
\begin{theorem}\label{T:NetoSoares}
There exists  a real Zariski dense open subset $\mathcal S(d)$ of the  set of foliations with a given degree $d>1$ in $\P^2$
such that  any $\Fc\in\mathcal S(d)$ satisfies
\begin{enumerate}
\item[1)] $\Fc$ has only  hyperbolic singularities and  no other singular points;

\item[2)] $\Fc$ has no invariant algebraic  curve. 
\end{enumerate}
\end{theorem}

On the  other  hand, Brunella  \cite{Brunella} has  shown that  each $\Fc\in\mathcal S(d)$   does  not admit 
any holomorphic non-constant map
$\C\rightarrow \P^2$  such that out of   the singularities of $\Fc$  the image of $\C$ is  locally contained in a leaf.  Consequently, by Definition \ref{D:Brody} and Remark  \ref{R:Brody_sufficiency},
a generic   holomorphic foliation  in $\P^2$ with a given degree $d>1$  satisfies the  hypotheses of  Theorem  \ref{thm_main}, Corollaries  \ref{C:extremal}  and \ref{cor_th_main}
  and Theorem \ref{thm_main_2}.


 
 \subsection{A local  model}
 \label{SS:local_model}

 For  $\U:=\D^k$ and  $t>0,$  let
 $t\U:= (t\D)^k,$  see Notation at the  beginning of the  section for the definition of  $t\D .$
 
First  we give  a  description of the local model for linearizable singularities.
Consider the foliation $(\D^k,\Lc,\{0\})$ which is the restriction to $\D^k$ of the foliation associated to the vector field
$$F(z)=\sum_{j=1}^k \lambda_jz_j{\partial\over \partial z_j}$$
with $\lambda_j\in\C\setminus\{0\}$.  The foliation is singular at the origin. We use here the Euclidean metric on $\D^k$.
Write $\lambda_j=s_j+it_j$ with $s_j,t_j\in\R$.  
For $x=(x_1,\ldots,x_k)\in \D^k\setminus\{0\}$, define the holomorphic map $\psi_x:\C\rightarrow\C^k\setminus\{0\}$ by
\begin{equation}\label{eq_leaf_equation}
\psi_x(\zeta):=\Big(x_1e^{\lambda_1\zeta},\ldots,x_ke^{\lambda_k\zeta}\Big)\quad \mbox{for}\quad \zeta\in\C.
\end{equation}
It is easy to see that $\psi_x(\C)$ is the integral curve of $F$ which contains
$\psi_x(0)=x$.
Write $\zeta=u+iv$ with $u,v\in\R$. The domain $\Pi_x:=\psi_x^{-1}(\D^k)$ in $\C$ is defined by the inequalities
$$s_ju-t_jv< -\log|x_j| \quad \mbox{for}\quad j=1,\ldots,k.$$
So, $\Pi_x$ is a convex polygon which is not necessarily
bounded. It contains $0$ since $\psi_x(0)=x$.  
The leaf of $\Fc$ through $x$ contains  the  Riemann surface 
\begin{equation}\label{e:Riemann_surface_L_x}
\widehat L_x:=\psi_x(\Pi_x)\subset L_x.
\end{equation}
In particular, the leaves in a singular flow box are parametrized   using holomorphic maps $\psi_x:\Pi_x\to L_x.$

 Now  let  $\Fc=(X,\Lc,E)$ be a  Brody hyperbolic foliation on a   Hermitian compact complex manifold $(X,g_X).$
Assume as usual that $E$  is  finite  and   all  points of $E$  are  linearizable. 
Let $\dist$ be the  distance on $X$ induced  by the ambient  metric $g_X.$
We only consider flow boxes which are biholomorphic to $\D^k$. A {\it regular flow box} is a flow boxes outside the singularities. {\it Singular flow boxes} are identified to their models $(\D^k,\Lc,\{0\})$ as 
described above.
For each singular point $x\in E$, we fix a singular flow box $\U_x$ such that $2\U_x\cap 2\U_{x'}=\varnothing$ if $x,x'\in E$ with $x\not=x'$. We also cover $X\setminus \cup_{x\in E} \U_x$ by a finite   number of regular flow boxes $(\U_p)_{p\in P}$ which are  fine enough. In particular, each $\overline \U_p$ is contained in a larger regular flow box $2\U_p$ with $2\U_p\cap  E=\varnothing.$ 
Thus  we obtain a  finite  cover $\Uc$ of $X$  consisting  of  regular  flow boxes $\U_p$ and  singular ones $(\U_x)_{x\in E}.$  In this  section  we
 suppose that  the ambient metric $g_X$ coincides with the standard Euclidean metric on each singular flow box $2\U_x\simeq 2\D^k,$ $x\in E.$ For $x=(x_1,\ldots,x_k)\in\C^k,$ let  $\|x\|$ be the  standard Euclidean norm of $x.$ Recall  that
 $\lof(\cdot):=1+|\log(\cdot)|.$

We record here   the following crucial result which gives  a precise estimate on the  function $\eta$  introduced  in  (\ref{eq_relation_Poincare_Hermitian_metrics}).   
\begin{lemma} \label{lem_poincare}
We keep  the above  hypotheses
and notation. Then there exists  a  constant $c>1$  with the following properties.\\
1)   $\eta \leq c$ on $X$, $\eta\geq c^{-1}$ outside the singular flow boxes $\cup_{x\in E}{1\over 4}\U_x$ and 
$$c^{-1} \cdot s \lof s \leq\eta(x)  \leq c \cdot s \lof s$$
for $x\in X\setminus E$  and $s:=\dist(x,E).$  
\\
2)  For every  $x$ in  a singular box  which is identified with $\D^k,$
 for every $\zeta\in\Pi_x,$
 $$
    c^{-1}\cdot {i d\zeta\wedge d\bar\zeta \over (\lof (\psi_x(\zeta)))^2}\leq (\psi_x^* g_P) (\zeta)\leq  c\cdot {i d\zeta\wedge d\bar\zeta \over (\lof (\psi_x(\zeta)))^2}.
 $$
\end{lemma}
\begin{proof}
Part 1) has been  proved  in   \cite[Proposition 3.3]{DinhNguyenSibony3}.

To prove Part 2),
write  $y=\psi_x(\zeta)$ for $\zeta\in\Pi_x,$  and  observe that 
$$\min\{|\lambda_1|,\ldots,|\lambda_k|\} \cdot\|y\|\leq   \|\psi'_x(\zeta)\|\leq \max\{|\lambda_1|,\ldots,|\lambda_k|\} \cdot \|y\|.$$
On the other hand,
 recall from 
(\ref{eq_relation_Poincare_Hermitian_metrics}) that
$$
i\ddbar \|y\|^2=\eta^2(y) g_P(y),
$$
Moreover, we know  from Part 1) that $\eta(y)\approx \|y\| \lof\|y\|.$
Pulling  back   both members of the last equality  by $\psi_x$ and  
using the previous  estimates for  $\|\psi'_x(\zeta)\|$ and for $\eta(y),$ we  obtain  the  desired  estimate for  $(\psi_x^* g_P) (\zeta).$
\end{proof}
 \subsection{Heat diffusions and harmonic currents versus harmonic  measures}
     
Let  $\Fc=(X,\Lc,E)$ be  a hyperbolic  foliation.    
     The leafwise Poincar\'e metric
$g_P$  induces  the corresponding 
Laplacian $\Delta$  on leaves such that
\begin{equation}\label{e:Laplacian}
i\ddbar u=\Delta u \cdot g_P,\qquad \text{on $X\setminus E$ for all}\ u\in \Cc_\Fc.
\end{equation}
A positive finite  Borel measure  $\mu$ on $X$ is said  to be
{\it  harmonic}  if
$$
\int_X  \Delta u d\mu=0
$$ 
 for all  functions  $u\in \Cc_\Fc.$

 For  every point  $x\in X\setminus E,$
 consider  the   {\it heat  equation} on $L_x$
 $$
 {\partial p(x,y,t)\over \partial t}=\Delta_y p(x,y,t),\qquad  \lim_{t\to 0} p(x,y,t)=\delta_x(y),\qquad   y\in L_x,\ t\in \R_+.
 $$
Here   $\delta_x$  denotes  the  Dirac mass at $x,$ $\Delta_y$ denotes the  Laplacian  $\Delta$ with respect to the  variable $y,$
 and  the  limit  is  taken  in the  sense of distribution, that is,
$$
 \lim_{t\to 0+
}\int_{L_x} p(x,y,t) f(y) g_P( y)=f(x)
$$
for  every  smooth function  $f$   compactly supported in $L_x.$   

The smallest positive solution of the  above  equation, denoted  by $p(x,y,t),$ is  called  {\it the heat kernel}. Such    a  solution   exists   because  $(L_x,g_P)$ is
complete and   of bounded  geometry  (see, for example,  \cite{CandelConlon2,Chavel}). 
 The  heat kernel   gives  rise to   a one  parameter  family $\{D_t:\ t\geq 0\}$ of  diffusion  operators    defined on bounded Borel measurable functions  on $M\setminus E:$
 \begin{equation}\label{eq_diffusions}
 D_tf(x):=\int_{L_x} p(x,y,t) f(y) g_P (y),\qquad x\in X\setminus E.
 \end{equation}
 We record here  the  semi-group property  of this  family: $D_0=\id$ and $D_{t+s}=D_t\circ D_s$ for $t,s\geq 0.$

  Let $\Cc^1_\Fc$ denote the  space   of  forms $h$ of bidegree $(1,1)$ defined  on
leaves  of the foliations such that $h$    is  compactly supported  on $X\setminus E$  and  that  $h$ is  leafwise  smooth
 and transversally continuous.  A form $h \in \Cc^1_\Fc$ is  said to be {\it positive} if its restriction to every plaque
 is  a  positive $(1,1)$-form in the  usual  sense of Lelong.
\begin{definition}\rm\label{D:harmonic_current}
A {\it   harmonic current} $T$  on the foliation $\Fc$ (or  equivalently, {\it directed  by $\Fc$}) is  a  linear continuous  form
    on $\Cc^1_\Fc$ which  verifies $\ddbar T=0$ in the  weak sense (namely  $T(\ddbar f)=0$ for all $f\in  \Cc_\Fc$), and
which is  positive (namely, $T(h)\geq 0$ for all positive forms $h\in \Cc^1_\Fc$). 
\end{definition}
Suppose  now that $E$ is  a finite  set. Then 
   the  existence  of    nonzero    harmonic currents      has been  established   by Berndtsson-Sibony in \cite[Theorem 1.4]{BerndtssonSibony}, and   Forn\ae ss-Sibony   in \cite[Corollary 3]{FornaessSibony2}.
The extension of $T$ by zero through $E,$ still denoted by $T,$ is a positive  $\partial\overline\partial$-closed  current on $X.$ The  total mass of  the positive  measure $  T\wedge g_X$ is always finite.    
    
 We have the  following  decomposition  (see  \cite[Proposition  2.3] {DinhNguyenSibony1}).
\begin{proposition} \label{prop_current_local}
Let $\Fc=(X,\Lc,E)$ be  a hyperbolic  foliation with  linearizable singularities  $E.$
Let $T$ be a   harmonic  current on  $X$. Let $\U\simeq \B\times\T$ be a flow
box which is relatively compact in $X\setminus E.$ Then, there is a positive Radon
measure $\nu$ on $\T$ and for $\nu$-almost every $\alpha\in \T$ there is a
positive harmonic function $h_\alpha$ on $\B$ 
such that if $K$ is compact in $\B,$ 
the integral $\int_\T \|h_\alpha\|_{L^1(K)}d\nu(\alpha)$ is finite and
$$\langle T,\chi\rangle=\int_\T \Big(\int_\B h_\alpha(y) \chi(y,\alpha) \Big) d\nu(t)
$$
for every   form  $\chi\in\Cc^1_\Fc$ compactly  supported in $\U.$
\end{proposition}

 A  subset  $M\subset  X\setminus E$ is  said to be  {\it leafwise saturated} if $x\in M$ implies the  whole leaf $L_x$ is  contained in $M.$
  A positive finite  measure  $\mu$ on the $\sigma$-algebra of Borel sets in $X$ is  said  to be  {\it ergodic} if for every  leafwise  saturated  Borel measurable set $M\subset X,$
   $\mu(M)$ is  equal to either $\mu(X)$ or $0.$   A   harmonic  current $T$ is said to be {\it extremal}
   if  it  is an extremal point  in  the  convex  cone of  all  harmonic  currents, i.e., if there  are
harmonic  currents $T_1, T_2$  such that  $ T={T_1+T_2\over 2},$ then $T_1$ and $T_2$ are colinear.    
\begin{theorem}\label{thm_harmonic_currents_vs_measures}
Let $\Fc=(X,\Lc,E)$ be  a hyperbolic  foliation with  linearizable singularities  $E.$ \\
1)   
  The  relation  $\mu=T\wedge g_P$  is  a one-to-one  correspondence between the  convex  cone of   harmonic  currents $T$ and  
the convex  cone of  harmonic  measures $\mu$. 
\\
2) If $T$ is  extremal, then    $\mu=T\wedge g_P$ is  ergodic.
\\
3) Each  harmonic    measure $\mu$  is $D_t$-invariant, i.e, 
$$ \int_X  D_tf d\mu=\int_X fd\mu,  \qquad  f\in L^1(X,\mu).   $$
\end{theorem}
\begin{proof}  
We start with   Part 1). First observe that, for each harmonic current $T,$ the positive measure  
 $ \mu:=T\wedge g_P$ is finite  by  \cite[Proposition 4.2]{DinhNguyenSibony1}. Moreover,  it is  easy to  see that $\mu$ is harmonic.
 Consequently, the map $T\mapsto T\wedge g_P$ is one-to-one. Therefore, to complete Part 1) it suffices to show that
each harmonic measure  $\mu$ may be  written as   $\mu=T\wedge g_P$ for some harmonic current $T.$ 
To do this  we proceed as in  the proof of  \cite[Proposition 5.1]{DinhNguyenSibony1}.

To prove  Part 2),
 suppose in order to get a contradiction that $\mu$ is  not ergodic. So there is a leafwise saturated Borel set $A\subset X\setminus E$ such that 
$0<\mu(A)<1.$ Let $\mu_1:=2\mu|_A$ and  $\mu_2:=2\mu|_{X\setminus  A}.$ So  $\mu={\mu_1+\mu_2\over 2},$ and $\mu_1,$ $\mu_2$ are not  co-linear.  Using the local description of $T$ on  each flow box (see
 \cite[Proposition 2.3]{DinhNguyenSibony1}), we can show that both $\mu_1$ and $\mu_2$ are  harmonic  measures.
By Part 1), let $T_1,T_2$  be   harmonic currents such that $\mu_1:= T_1\wedge g_P$ and $\mu_2:= T_2\wedge g_P.$  This, combined  with   
  $\mu= {\mu_1+\mu_2\over 2},$ implies that  $T={T_1+T_2\over 2}$ and $T_1,T_2$  are not co-linear. 
This contradicts the extremality of $T.$


 Part 3) follows from   \cite[Theorem 6.4]{DinhNguyenSibony1}  applied to the positive $(1,1)$-form $\beta:=g_P.$
 \end{proof}

 \subsection{Measure theory on sample-path spaces}
 \label{SS:Measure_theory}
  
In this  subsection we follow the   expositions  given in Sections 2.2, 2.4 and 2.5 in \cite{NguyenVietAnh1} (see also \cite{CandelConlon2}).
The $\sigma$-algebra  generated  by a family $\Sc$ of subsets  of $\Omega$ is, by definition, the  smallest  $\sigma$-algebra  
containing  $\Sc.$ 

Let $\Fc=(X,\Lc,E)$ be  a   hyperbolic foliation endowed  with the leafwise Poincar\'e metric $g_P.$
Let  $\Omega:=\Omega(\Fc) $  be  the space consisting of  all continuous  paths  $\omega:\ [0,\infty)\to  X$ with image fully contained  in a  single   leaf. This  space  is  called {\it the sample-path space} associated to  $\Fc.$
  Observe that
$\Omega$  can be  thought of  as the  set of all possible paths that a 
Brownian particle, located  at $\omega(0)$  at time $t=0,$ might  follow as time  progresses.
For each $x \in X\setminus E,$
let $\Omega_x=\Omega_x(\Fc)$ be the  space  of all continuous
leafwise paths  starting at $x$ in $X\setminus E ,$ that is,
\begin{equation}\label{eq_Omega_x}
\Omega_x:=\left\lbrace   \omega\in \Omega:\  \omega(0)=x\right\rbrace.
\end{equation}
Garnett developed in \cite{Garnett} a  theory of leafwise  Brownian  motion in this  context  by constructing
a $\sigma$-algebra  $(\Omega,\widetilde \Ac)$ together  with a family of Wiener measures (see also \cite{Candel2, CandelConlon2}).
Now  recall briefly her construction. 
A {\it  cylinder  set (in $\Omega$)} is a 
 set of the form
$$
C=C(\{t_i,B_i\}:1\leq i\leq m):=\left\lbrace \omega \in \Omega:\ \omega(t_i)\in B_i, \qquad 1\leq i\leq m  \right\rbrace,
$$
where   $m$ is a positive integer  and the $B_i$ are Borel subsets of $X\setminus E,$ 
and $0\leq t_1<t_2<\cdots<t_m$ is a  set of increasing times.
In other words, $C$ consists of all paths $\omega\in  \Omega$ which can be found within $B_i$ at time $t_i.$
For  each point $x\in X\setminus E,$ let
\begin{equation}\label{eq_formula_W_x_without_holonomy}
W_x(C ) :=\Big (D_{t_1}(\chi_{B_1}D_{t_2-t_1}(\chi_{B_2}\cdots\chi_{B_{m-1}} D_{t_m-t_{m-1}}(\chi_{B_m})\cdots))\Big) (x),
\end{equation}
where $C:=C(\{t_i,B_i\}:1\leq i\leq m)$ as  above,  $\chi_{B_i}$
is the characteristic function of $B_i$ and $D_t$ is the diffusion operator
given  by  (\ref{eq_diffusions}).
Let  $\widetilde\Ac=\widetilde\Ac (\Fc)$ be  the $\sigma$-algebra   generated by all cylinder sets.
It can be proved that $W_x$ extends  to a probability measure on  $(\Omega,\widetilde\Ac).$
  
In the recent work \cite{NguyenVietAnh1} we  introduce  another   $\sigma$-algebra  $\Ac$ on $\Omega,$ which is bigger  than $\widetilde\Ac.$ In fact, 
$\Ac$  takes into account the  holonomy phenomenon,  whereas $\widetilde\Ac$ does  not so.  Here is our construction in the present context.
The  {\it covering foliation} $\widetilde \Fc=(\widetilde X,\widetilde\Lc)$  of  a singular foliation $\Fc$  is, in some  sense, its  universal cover.
We give  here its construction.
For  every leaf $L$ of $\Fc$ and every point  $x\in L,$  let $\pi_1(L,x)$ denotes   the first  fundamental  group of
 all continuous closed paths $\gamma:\  [0,1]\to L$ based  at $x,$ i.e. $\gamma(0)=\gamma(1)=x.$  Let   $[\gamma]\in \pi_1(L,x)$ be  the   class of   a    closed path $\gamma$ based  at $x.$
 Then the pair  $(x,[\gamma])$ represents      a point  of   $ \widetilde X.$
  Thus
the set of points  $ \widetilde X$ of $\widetilde \Fc$  is well-defined. The  leaf  $\widetilde L$ passing through a given point $(x,[\gamma])\in\widetilde X,$ is  by definition, the set
$$
 \widetilde L:=\left\lbrace (y, [\delta]):\ y\in L_x,\  [\delta]\in \pi_1(L,y)  \right\rbrace,
$$
which is the universal cover of $L_x.$
 We  put the following  topological structure on  $\widetilde X$ by describing
  a  basis of open  sets. Such a  
   basis  consists of all  sets $\Nc(U,\alpha),$
  $U$ being an open subset of  $X\setminus E$ and $\alpha:\ U\times [0,1]\to X$ being   a  continuous  function such that  $\alpha_x:=\alpha(x,\cdot)$ is a  closed  path in $L_x$ based at  $x$ for each $x\in U,$ and  
  $$
  \Nc(U,\alpha):=  \left\lbrace (x,[\alpha_x]):\ x\in U \right\rbrace.
  $$
   The projection $\pi : \widetilde X\to  X\setminus E$ is defined by $\pi(x,[\gamma]) :=x. $  It is  clear that $\pi$ is  locally homeomorphic and  is  a leafwise map.
   By pulling-back the foliation atlas $\Lc$ of $ \Fc $ as well as  the Poincar\'e metric $g_P$ via  $\pi,$   we obtain  a  natural  foliation atlas $\widetilde \Lc$ for the hyperbolic  foliation 
$  \widetilde \Fc $ endowed with the leafwise metric $\pi^*g_P.$
 Denote  by $\widetilde \Omega$ the sample-path space $\Omega( \widetilde \Fc)$ associated with the foliation   
 $ \widetilde \Fc .$  
   
Let 
$ x\in X\setminus E$ and  $\tilde x$ an arbitrary point in $\pi^{-1}(x)\subset \widetilde X.$
 Similarly as in (\ref{eq_Omega_x}), let   $\widetilde\Omega_{\tilde x}=\Omega_{\tilde x}( \widetilde\Fc)$ be the  space  of all  paths  
in $\widetilde \Omega$
starting at $\tilde x.$ 
Every  path $\omega\in \Omega_x$ lifts uniquely  to
a path $\tilde\omega\in \widetilde\Omega_{\tilde x}$  in the sense  that $\pi\circ \tilde\omega=\omega.$
In what follows this bijective lifting  is  denoted by $\pi^{-1}_{\tilde x}:\  \Omega_x\to \widetilde\Omega_{\tilde x}.$ So  $\pi\circ (\pi^{-1}_{\tilde x}(\omega))=\omega,$
 $\omega\in \Omega_x.$ 
\begin{definition} \label{defi_algebras_Ac} Let   
  $ \Ac=\Ac(\Fc)$ be  the  $\sigma$-algebra generated by
all sets  of  following family
 $$ \left  \lbrace \pi\circ \tilde A:\ \text{cylinder set}\ \tilde A\  \text{in} \ \widetilde \Omega    \right\rbrace,$$ 
  where  $\pi\circ \tilde A:= \{ \pi\circ \tilde \omega:\ \tilde\omega\in \tilde A\}.$
    \end{definition}
 Observe that    $\widetilde \Ac\subset  \Ac$ and that the  equality holds if  every   leaf  of  the foliation is
homeomorphic to the  disc $\D.$ 
 Now  we  construct  a family  $\{W_x\}_{x\in M\setminus E}$ of probability Wiener measures   on $(\Omega,\Ac).$ 
  Let $x\in X\setminus E$ and 
$C$ an element of $\Ac.$ Then  we  define the  so-called {\it  Wiener measure}  $W_x$ by the following formula
\begin{equation}\label{eq_formula_W_x}
 W_x(C):= W_{\tilde x}( \pi^{-1}_{\tilde x} C),
  \end{equation}
  where   
   $\tilde x$ is  an arbitrary point in $\pi^{-1}(x),$ and  
   $$\pi^{-1}_{\tilde x}C:=\left\lbrace \pi^{-1}_{\tilde x}\omega:\  \omega\in C\cap \Omega_x\right\rbrace,
$$
and  $ W_{\tilde x}$ is the  probability measure  on  $(\widetilde \Omega,\widetilde \Ac(\widetilde \Fc)) $ which was  defined by (\ref{eq_formula_W_x_without_holonomy}).
Given  a positive  finite Borel  measure $\mu$  on $X\setminus E,$  consider the   measure  $\bar\mu$ on $(\Omega,\Ac)$
defined by   
   \begin{equation}\label{eq_formula_bar_mu}
   \bar\mu(A):=\int_X\left ( \int_{\omega\in A\cap   \Omega_x}  dW_x \right ) d\mu(x),\qquad  A\in\Ac. 
\end{equation}
The  measure $\bar\mu$ is  called  the  {\it Wiener measure with initial distribution $\mu$.}
 Here  are its  important  properties.
\begin{proposition}\label{prop_Wiener_measure}
We keep the above  hypotheses and notation. \\
(i)    The value  of $W_x(C)$ defined  in (\ref{eq_formula_W_x}) is  independent of the choice
 of $\tilde x.$  Moreover, $W_x$ is a  probability  measure     
on $(\Omega,\Ac).$
\\
(ii)   $\bar\mu$ given in (\ref{eq_formula_bar_mu}) is  a positive finite  measure on $(\Omega,\Ac)$ and
$\bar\mu(\Omega)=\mu(X\setminus E).$
\\
(iii) If  $\mu$ is harmonic, then  $\bar\mu$ is  time-invariant, that is,
$$
\int_\Omega F(\sigma_t(\omega))d\bar\mu(\omega)=\int_\Omega F(\omega)d\bar\mu(\omega),
$$
for all $t\in\R^+$ and $F\in L^1(\Omega,\bar\mu),$  where 
  the   shift-transformation
  $\sigma_t:\  \Omega\to\Omega$ is defined by 
   \begin{equation}\label{eq_shift}  \sigma_t(\omega)(s):=\omega(s+t),\qquad  \omega\in \Omega,\  s\in\R^+.
   \end{equation}
  \end{proposition}     
  \begin{proof}
 Assertion (i)  has been proved in    \cite[Theorem 2.15]{NguyenVietAnh1}. Assertion (ii)  has been established in
  \cite[Theorem 2.16]{NguyenVietAnh1}.
  
 By Part 3) of Theorem \ref{thm_harmonic_currents_vs_measures},
$\mu$ is  $D_t$-invariant  for all $t\in\R^+.$
 Consequently,  applying  \cite[Theorem  2.20]{NguyenVietAnh1} to $\mu$ yields that $\bar\mu$ is time-invariant.
\end{proof}

\subsection{Holonomy cocycles}
\label{SS:holonomy_cocycle}

Now  we define  the holonomy cocycle  of a hyperbolic foliation $\Fc=(X,\Lc,E)$ on a Hermitian  complex surface $X.$
For each point $x\in X\setminus E,$ let $T_x(X)$ (resp.  $T_x(L_x)\subset T_x(X)$) be    the tangent space  of $X$ (resp. $L_x$) at $x.$
 For every transversal    $S$ at a point $x$ (that is, $x\in S$), let  $T_x(S)$ denote the tangent space of $S$ at $x.$  

Now fix a point $x\in X\setminus E$ and a path $\omega\in\Omega_{x}$  and a time $t\in \R^+,$ and  let $y:=\omega(t).$
Fix  a transversal $S_x$ at $x$ (resp. $S_y$  at $y$) such that the complex line
$T_{x}(S_x)$ is the  orthogonal complement  of the complex line $T_{x}(L_{x})$ in the Hermitian space $(T_x(X),g(x))$
(resp.        $T_{y}(S_y)$ is the  orthogonal complement  of $T_{y}(L_{y})$ in $(T_{y}(X),g(y))$). 
Let  $\hol_{\omega,t}$  be the  holonomy map along the path  $\omega|_{[0,t]}$ from   an open  neighborhood of $x$ in $S_x$   onto
 an open  neighborhood of $y$ in $S_y.$     The  derivative  $D \hol_{\omega,t}:\
T_{x}(S_x)\to T_{y}(S_y) $ induces 
the so-called {\it holonomy cocycle} $\mathcal H:\ \Omega\times\R^+\to\R^+$ given by  $$\mathcal H(\omega,t):=
  \|D \hol_{\omega,t}(x)\|.$$ 
The last  map depends  only on the path  $\omega|_{[0,t]},$ in fact, it
depends only on the  homotopy class of  this path.
 In particular, it is  independent of   the choice  of transversals  $S_x$ and $S_y.$ 
 We see easily that
 $$
\mathcal H(\omega,t) =\lim_{z\to x,\ z\in S_x} \dist(\hol_{\omega,t}(z), y)/ \dist(z,x).
  $$
 On the other hand, we note the following additive property   which  is  an immediate consequence  of the  definition of $
\mathcal H(\omega,t)$ (see also   \cite[Proposition 3.3]{NguyenVietAnh1}):
\begin{equation} \label{eq_multiplicative_cocycles}
\log {\| \mathcal{H}(\omega,t+s)   \| }
= \log {\| \mathcal{H}(\omega,t)   \| }+\log {\| \mathcal{H}(\sigma_t(\omega),s)   \| },\qquad t,s\in\R^+, \omega\in\Omega,
\end{equation}
  where  $\sigma_t:\  \Omega\to\Omega$ is the   shift-transformation
    given by 
\eqref{eq_shift}.

 \section{Holonomy cocycle vs Poincar\'e metric}
  \label{S:holonomy_cocycle} 
 In this  section let $\Fc=(X,\Lc,E)$    be  
a   holomorphic Brody hyperbolic foliation   with     linearizable  singularities $E$ in a   Hermitian compact complex surface $X.$
 Let $\mathcal H$ be the holonomy cocycle of the foliation. In order to study the behavior of $\mathcal H$ near a (hyperbolic) singular point, 
  we  use the local model  $(\D^2,\Lc,\{0\})$ introduced in  Subsection   \ref{SS:local_model}.  This  is  the restriction to $\D^2$ of the foliation associated with the vector field $$F(z,w) = z {\partial\over \partial z}
+ \lambda w{\partial\over \partial w}\qquad\text{with some complex number}\qquad \lambda\not =0.$$ 
Since  the  main results of the article do not depend on the choice of a Hermitian metric on $X,$  we  can fix  a metric  which is  equal  to the Euclidean one in  each singular  flow box.
This  will  simplify  our presentation.

For $x=(z,w)\in \D^2\setminus\{0\}$,  the holomorphic map $\psi_x:\Pi_x\rightarrow\D^2\setminus\{0\}$
given by (\ref{eq_leaf_equation}) may be  rewritten as 
\begin{equation}\label{eq_leaf_equation_new}
\psi_x(\zeta):=\Big(ze^{\zeta},we^{\lambda\zeta}\Big)\quad \mbox{for}\quad \zeta\in\Pi_x.
\end{equation}
\begin{proposition}\label{P:holonomy} Let $\D^2$ be endowed with the  Euclidean metric.
For each  $x=(z,w)\in\D^2$ and $\zeta\in \Pi_x,$ consider a path $\omega\in \Omega$ (if it exists) such  that
$$
\omega(t)= \psi_x(t\zeta)= (z e^{\zeta t},we^{\lambda\zeta t})\subset \D^2
$$
for all $t\in [0,1]$ (see \eqref{eq_leaf_equation_new} above). Then 
$$
\mathcal H(\omega,1)=|e^\zeta||e^{\lambda \zeta}|{\sqrt{|z|^2+|\lambda w|^2}\over \sqrt{ |ze^\zeta|^2+|\lambda we^{\lambda \zeta}|^2           }}. 
$$
\end{proposition}
  \begin{proof} Let  $ y:= \omega(1)=(z e^{\zeta},we^{\lambda\zeta}).$
  Since  the vector  $(z,\lambda w)$ is  tangent to the leaf $L_x$  at $x ,$ 
  the vector $N_x:=(-\bar \lambda\bar w,\bar z)$ is  normal to $L_x$ at $x,$ and hence,
the complex normal line $S_x$  to  $L_x$  at $x$ is the set   
$$\{x+s\cdot N_x:\   s\in \C\}=  \{(z-\bar \lambda\bar w s, w+\bar zs):\   s\in \C\}.$$
Similarly, let $N_y:= (-\bar \lambda\bar we^{\bar\lambda\bar \zeta},\bar ze^{\bar \zeta})$
be  the vector normal to $L_y$ at $y,$ and let $S_y:=\{ y+s\cdot N_y:\ s\in\C\}$ be the complex normal line   to  $L_y$  at $y.$  
Since $N_x$  (resp. $N_y$) may be regarded,  in a sufficiently small open neighborhood of $x$ (resp. $y$), as a transversal,
we can describe the holonomy map $\hol_{\omega,t}$ using them. Indeed,   
  for  each $s\in \C$ with $|s|$ small enough, we  want to find $\xi\in\C$ close to $\zeta$  such that
 $\big ((z-\bar \lambda\bar w s)e^{\xi},( w+\bar zs)e^{\lambda\xi}\big)$ belongs to  $S_y.$
This  is  equivalent  to the fact that the following  two vectors   
 $$
V_s:= \Big ((z-\bar \lambda\bar w s)e^\xi -ze^\zeta,  (w+\bar zs)e^{\lambda\xi}-  we^{\lambda\zeta}\Big)\ \text{and}\  
  N_y$$
  are colinear.
    Write  $\xi=\zeta+as+O(s^2).$ So
  $
  e^\xi= e^\zeta(1+as+O(s^2))
  $  and $e^{\lambda\xi}=e^{\lambda\zeta}(1+\lambda as+O(s^2)).$
  In order  to determine $a,$
   we insert the last two identities into the expression of $V_s$ and get that
  \begin{equation}\label{eq_V_s}
  V_s= s\cdot \Big ((za-\bar \lambda\bar w)e^\zeta,  (\bar z+a\lambda w)e^{\lambda\zeta}\Big)+O(s^2).
  \end{equation}
So the above colinearity condition reduces to the colinearity of the following two vectors
$$
\Big ((za-\bar \lambda\bar w)e^\zeta,  (\bar z+a\lambda w)e^{\lambda\zeta}\Big)
\ \text{and}\  
 (-\bar \lambda\bar we^{\bar\lambda\bar \zeta},\bar ze^{\bar \zeta}).
$$
Solving this equation yields that
$$
a={  \bar\lambda \bar z\bar w  (|e^\zeta|^2- |e^{\lambda\zeta}|^2)\over  |z|^2|e^\zeta|^2+|\lambda|^2|w|^2 |e^{\lambda\zeta}|^2} .
$$
Recall  that $T_x(S_x)=S_x$ is orthogonal to  $T_x(L_x)$ at $x =\omega(0)$
and  $T_y(S_y)=S_y$ is orthogonal  to $T_y(L_y)$ at $y= \omega(1).$ Moreover, 
 $x+s\cdot N_x$ and  $y+V_s$  are on the same leaf for all $s\in\C$ with $|s|$ small enough.
Consequently,
a geometric argument shows that
$$
\mathcal H(\omega,1)=\lim_{s\to 0}  \| V_s\|/ \| s\cdot N_s\|= { \Big\| \Big ((za-\bar \lambda\bar w)e^\zeta,  (\bar z+a\lambda w)e^{\lambda\zeta}\Big)\Big\| \over  \|(-\bar \lambda\bar w,\bar z)\|  },  
$$
where the last  equality holds by (\ref{eq_V_s}).
Inserting   the above value  of $a$ into the  last expression, a straightforward calculation gives  the  desired result.
  \end{proof}
  
 Now  we define  a new variant of  Poincar\'e ``distance" $\dist_P$   which  takes into account the holonomy phenomenon.  
  Let    $\omega\in\Omega$ and $0\leq t\leq  s.$ Put $x:=\omega(t)$ and $y:=\omega(s).$
  Let $\phi_x:\ \D\to  L_x$ be a universal covering map with $\phi_x(0)=x.$  The  path  $ [0, s-t]\ni r \mapsto\omega(t+r)$
is lifted  by $\phi_x$ to a continuous path $\beta:\  [0, s-t]\to\D$   such that $\beta(0)=0.$ Let $\tau:= \beta(s-t)\in\D.$
So $\phi_x(\tau)=\omega(s)=y.$
  Now  we are in the position to define  the new Poincar\'e function
  \begin{equation}\label{e:distance}
  \dist_P(\omega:t,s):=  \dist_P(0, \tau)=\log\left( {1+|\tau|\over 1-|\tau|}\right) ,
  \end{equation}
 where on the right hand side  $\dist_P$ is the usual Poincar\'e distance  on $\D.$  Note that $
  \dist_P(\omega:t,s)$ is  independent  of the choice   of $\phi_x.$ Moreover, it is  uniquely determined 
  by   $x=\omega(t),$ $y=\omega(s)$ and  the  homotopy class (two end-points  being fixed) of the path  $ [0, s-t]\ni r \mapsto\omega(t+r).$  There is  exactly one  homotopy class for which 
$\dist_P(\cdot:t,s)$ coincides with  $\dist_P(x,y).$

 The following lemma  shows us how  deep a leaf can go into  a singular flow  box before the hyperbolic  time $R.$
\begin{lemma}\label{lem_go_deeper}
There is a constant $c>0$ with the following property.  Let  $\omega\in \Omega$ be such that
$\omega[0,1]\subset (1/2 \D)^2 $ and that $\omega[0,1]$ is (locally) geodesic with respect to the leafwise Poincar\'e metric $g_P.$
Write $(z,w):=x=\omega(0)$ and $R:=\dist_P(\omega:0,1).$
Then there  exists  $\zeta\in \Pi_x$  (see \eqref{eq_leaf_equation_new} above) such that 
  $    \omega(1)   =      (z e^{\zeta },we^{\lambda\zeta }) $
  and that
  $$|\zeta|\leq  e^{c R}|\log \|x\||.$$
 \end{lemma}
\begin{proof}
First  we show  that there is  $r>0$ such that if $R=\dist_P(\omega:0,1)\leq r$
then  there  exists  $\zeta\in \Pi_x$   such that 
  $    \omega(1)   =      (z e^{\zeta },we^{\lambda\zeta }) $
  and that
  \begin{equation}\label{eq_case_easy}
|\zeta|\leq {|\log \|x\||\over 2|\lambda|} .
\end{equation}
Indeed, let $\omega\in\Omega$ be a path  such that 
\\
$\bullet$ $\omega[0,1]$ is locally geodesic;
\\
$\bullet$  
 for all $t\in[0,1],$
  $\omega(t):= (z e^{\zeta(t) },we^{\lambda\zeta(t) })\in (1/2 \D)^2  ;$
  \\
  $\bullet$  
     $\zeta(0)=0$  and   $|\zeta(t)|\leq {|\log \|x\||\over 2|\lambda|}  $ for all $t\in[0,1]$
     and
  $|\zeta(1)|= {|\log \|x\||\over 2|\lambda|}  .$ 
  
We only need to show that $R=\dist_P(\omega:0,1)\leq r$ for some $r>0$ independent of $\omega.$
Indeed, it follows from  the second  and third $\bullet$ above  that  $|\log \| \omega(t)\||\approx  |\log\|x\||$ for $t\in[0,1].$ 
Therefore,  by integrating  along  the  path $[0,1]\ni t\mapsto \zeta(t)$ 
and using  the first $\bullet$ above, and  
applying Part 2) of  Lemma  \ref{lem_poincare}, we get that
\begin{multline*}
\dist_P(\omega:0,1)=\int_{\omega[0,1]} \sqrt{g_P(z)} = \int_0^1
\zeta_t^*(\psi^*_x  (\sqrt{g_P})) \\
 \geq c_1 \int_0^{|\log \|x\||\over 2|\lambda|}  |\log\|x\||^{-1} ds= {c_1\over 2\lambda}=:r, 
\end{multline*}
 where $c_1>0$ is  a constant. This proves (\ref{eq_case_easy}).

Next, we  prove  the lemma  for  a general $R>0.$
Suppose  without loss of generality that $r=1.$
Let $0=t_0<\cdots<t_n=1$ be  a  subdivision  of $[0,1]$ such that  $\dist_P(\omega:t_j,t_{j+1})\leq 1$ for each $0\leq j\leq n-1$
and that $n$ is  as smallest  as possible. So $n$ is  the smallest integer  $\geq R.$
Let $x_j:=\omega(t_j).$ So $x_0=\omega(0)=x=(z,w).$ Applying (\ref{eq_case_easy}) repeatedly,  we obtain, for each $0\leq j\leq n,$
 $\zeta_j\in\C$ and  $x_j=(z_{j},w_{j})\in  (1/2 \D)^2 $ such that
$x_{j+1}=(z_{j}e^{\zeta_j},w_{j} e^{\lambda \zeta_j})$ for $0\leq j<n$ and that $ |\zeta_j|\leq  c_2|\log \|x_j\||.$
So $ |\log \|x_{j+1}\||\leq  c_3|\log \|x_j\||$ for some constant $c_3>1$ which depends only on $c_2$ and $\lambda.$
Thus, 
$$|\log \|x_j\||\leq c_3^j |\log \|x\||\quad\text{and}\quad
 |\zeta_j|\leq  c_2 c_3^j |\log \|x\||
.$$ Writing $
\omega(1)=x_n= (z_{0}e^\zeta, w_{0} e^{\lambda\zeta})=(ze^\zeta, w e^{\lambda\zeta})$   with 
$\zeta:= \zeta_1+\cdots+\zeta_{n-1}
$
 and using the last  estimate,  the  desired  conclusion of the lemma follows.
\end{proof}

 The  following   result gives  an estimate on  the  expansion rate of $\mathcal H(\omega,\cdot)$ in terms of the Poincar\'e 
function  $\dist_P(\omega: \cdot,\cdot)$
  and  the  distance  $\dist(\omega(0),E).$  
\begin{proposition}\label{prop_expansion_rate}
  There  is a  constant $c>0$ such that
 $$  \big |\log {\| \mathcal{H}(\omega,t)   \|} \big |
 \leq  c \lof \dist(\omega(0),E)\ \cdot\  \exp{\Big (c\, \dist_P(\omega: 0,t)\Big)} ,\qquad
 \omega\in\Omega,\ t\in\R^+.
$$  
\end{proposition}
\begin{proof}  We may  suppose without loss of generality  that $t=1.$
  Let $\omega\in\Omega,$ and  put $x:=\omega(0)$ and $y:=\omega(1).$
 Since   $\mathcal{H}(\omega,1)$ depends only on the homotopy class of the  path $\omega|_{[0,1]},$    we may  assume  without loss of generality that the segment $\omega[0,1]$ is  (locally)  geodesic with respect to the Poincar\'e metric  on $L_x.$
Let $\Uc$ be the    finite  cover   of $M$  by  regular  and  singular  flow boxes  given in Subsection \ref {SS:local_model}.
We consider  three steps.\\
{\bf Step   1:}  {\it If  there is a singular  flow  box $\U$  which  contains  the whole segment $\omega([0,1]),$ then the proposition is  true for  $c=c_1,$
where $c_1>0$ is  a constant large enough.}

Write $x=(z,w)$ and $y:=\omega(1).$
 Let $R:= \dist_P(\omega:0,1).$ By Lemma \ref{lem_go_deeper}, we may write
$y =(ze^\zeta,we^{\lambda\zeta})$ for some $\zeta\in\C$   such that
 $$|\zeta|\leq  e^{c_2R}.$$
 Inserting   this  into the  expression  for the holonomy map given in Proposition \ref{P:holonomy},
a straightforward computation shows  that
$$
\big |\log {\| \mathcal{H}(\omega,1)   \|} \big |
 \leq  c_3  |\log \|x\||  e^{c_3R} 
$$  
  for a constant $c_3>0$ independent of $\omega.$ Choosing $c_1>c_3$ large enough, Step 1 follows from the last estimate.

\noindent
{\bf Step   2:}  {\it   If  the whole segment $\omega([0,1])$
is contained  in a single   regular  flow  box $\U\in\Uc,$  then $\big | \log {\| \mathcal{H}(\omega,1)   \|}  \big|\leq c_4,$
 where $c_4>0$ is a constant independent of $\omega.$ In particular, the proposition  is true in this  case for $c=c_1,$ where
 $c_1>0$ is  a constant large enough.}
 
 Observe  that the geodesic segment  $\omega[0,1]$ is contained in the unique  plaque  of $\U$     which   passes through $x.$
 This, combined   with  the  description of the  holonomy map  on $\U,$ implies that $\| \mathcal{H}(\omega,1)   \|\leq  e^{c_4}$
 for a  constant $c_4>0$ independent of $\omega.$ Hence, $\big | \log {\| \mathcal{H}(\omega,1)   \|}  \big|\leq c_4.$
Therefore, choosing $c_1>c_4$ large  enough,   we have that 
$$c_1 \lof \dist(\omega(0),E)\geq c_4\geq  \big | \log {\| \mathcal{H}(\omega,1)   \|}  \big|.$$
This  proves the proposition  in  Step  2.
  
 \noindent
{\bf Step  3:}  {\it  Proof of the proposition in the general case.}

Consider the family  of all   finite subdivisions of 
 $[0,1]$ into  intervals $[t_{j-1},t_j]$  with  $1\leq j\leq n$  such that
 $t_0=0,$ $t_n=1$ and that   each segment  $\omega([t_{j-1},t_j])$ is  contained in a single (regular or singular) flow box $\U_j$ for each $j$.
 Fix a  member of this  family  such  that the  number $n$ is  smallest possible.
 We may assume  without loss of generality that $n>1$ since     the  case  $n=1$ follows  either from  Step 1 (if
 $\U_1$  is  singular)  or from Step 2 (if $\U_1$ is  regular).
The minimality of $n$ implies that all $\omega(t_1),\ldots\omega(t_{n-1})$ belong to the union of all regular  flow boxes of $\Uc.$
Therefore,  there is a constant $r_0>0$ independent of $\omega$ such that 
$$
  \dist_P(\omega:t_j,t_{j+1})\geq  r_0,\qquad 1\leq j\leq n-1. 
$$
 Thus  
\begin{equation}\label{eq_n}
n\leq  1+   r_0^{-1} \dist_P(\omega: 0,1).
\end{equation}
Moreover,
 there is a constant $c_5>1$ independent of $\omega$ such that 
$$
1\leq \lof \dist(\omega(t_j),E)\leq  c_5,\qquad 1\leq j\leq n-1. 
$$
 Using this and  applying Step 1 to each singular box in the family   $(\U_j)_{j=1}^n$ and  applying Step 2 to each regular 
 flow  box in  the above family, we obtain that
 \begin{eqnarray*}
\big | \log {\| \mathcal{H}(\omega,t_1)   \| }\big|
 &\leq & c_1 \lof \dist(\omega(t_0),E)\cdot\exp{\Big (c_1 \dist_P(\omega:t_0,t_1)\Big)} ,\\
\big |\log {\| \mathcal{H}(\sigma_{t_{j-1}}(\omega),t_j-t_{j-1})} \big |&\leq&   c_1c_5\exp{\Big (c_1\dist_P(\omega:t_{j-1},t_j)\Big)}   ,\quad 2\leq j\leq n.
\end{eqnarray*}
 Summing  up the above  estimates,  we  get that
 \begin{multline*}
 \sum_{j=1}^n\big|\log {\| \mathcal{H}(\sigma_{t_{j-1}}(\omega),t_j-t_{j-1})   \| }\big|\leq  c_1 \lof \dist(\omega(t_0),E)\cdot\exp{\Big ( c_1 \dist_P(\omega:t_0,t_1)\Big)}\\
+\sum_{j=2}^n c_1c_5\exp{\Big (c_1 \dist_P(\omega:t_{j-1},t_j)\Big)}.
\end{multline*}
 On the other hand, we infer from   (\ref{eq_multiplicative_cocycles}) that  
 $$
\big |\log {\| \mathcal{H}(\omega,1)   \| }\big|
= \sum_{j=1}^{n}  \big| \log {\| \mathcal{H}(\sigma_{t_{j-1}}(\omega),t_{j}-t_{j-1})   \| }\big|.
$$
 This,  coupled   with the previous estimate, gives that  
  \begin{equation}\label{eq_two_sums_new}
  \begin{split}
\big| \log {\| \mathcal{H}(\omega,1)   \| }\big| &\leq c_1\lof \dist(\omega(t_0),E)\cdot\exp{\Big ( c_0 \dist_P(\omega:t_0,t_1)\Big)}\\
&+
\sum_{j=2}^n c_1c_5\exp{\Big (c_1 \dist_P(\omega:t_{j-1},t_j)\Big)}.
\end{split}
\end{equation}
  Since    $
 \lof \dist(x,E)\geq 1$ for all $x\in M\setminus E,$   the right hand side of the last line is  dominated by
a constant times $\lof \dist(\omega(t_0),E)$ times
$$
\sum_{j=1}^n  \exp{\Big( c_1 \dist_P(\omega:t_{j-1},t_j)\Big)}\leq   n\cdot \exp{\Big (c_1 \dist_P(\omega:0,1)\Big)},
$$
where the last inequality holds because  of the identity
 $$
  \dist_P(\omega:0,1)=\sum_{j=1}^n\dist_P(\omega:t_{j-1},t_j).
  $$
  Inserting (\ref{eq_n}) into the right hand side  of the last inequality and  choosing $c>c_1$ large enough, 
we find that its left hand  side  is bounded by   $c\, \exp{\Big (c\, \dist_P(\omega:0,1)\Big)}.$
So  the  right hand  side   of (\ref{eq_two_sums_new}) is  also bounded by a constant times $\lof \dist(\omega(t_0),E)\cdot \exp{\Big (c \,\dist_P(\omega:0,1)\Big)},$ and the proof is  thereby completed.
\end{proof}

\section{Proof of  the  main results modulo the integrability condition} \label{section_Main_Theorem_1}
 
  This  section is  devoted to the proofs of Theorem    \ref{thm_main} and Corollary \ref{C:extremal}  modulo the integrability  condition \eqref{e:necessary_integrability}, i.e., modulo
  Theorem \ref{thm_main_2}. 
 We need  the  following result. 
\begin{lemma}\label{lem_Candel} There is a  constant $c>1$ such that for all  $x\in M\setminus E$ and all $s\geq 1,$
 $$ W_x\left\{\omega\in\Omega :\  \sup_{t\in[0,1]}\dist_P(\omega:0, t)>s \right \}< ce^{-c^{-1}s^2}.$$
\end{lemma}
\begin{proof} Let $\phi_x:\D\to L_x$ be a universal covering map  with $\phi_x(0)=x$. 
We have to show that
$$
 W_0\left\{\omega\in\Omega(\D) :\  \sup_{t\in[0,1]}\dist_P(\omega(0),\omega(t))>s \right \}< ce^{-c^{-1}s^2}
$$
 where  $W_0$
is  the Wiener  measure  at $0$ of  the unit disc $\D $ endowed with the Poincar\'e metric  $g_P, $ and $\dist_P(\cdot,\cdot)$ is the  Poincar\'e distance.  
Since  the Poincar\'e metric   is   complete and of    bounded  geometry,
 the  last  estimate
 holds  by combining  \cite[Lemma 8.16 and Corollary 8.8]{Candel2}.  
\end{proof}

Now we arrive  at the   
 \\
\noindent {\bf End of the proof of  Theorem    \ref{thm_main} modulo the integrability  condition \eqref{e:necessary_integrability}.}  
 By  Proposition  
 \ref{prop_expansion_rate},  we  get a constant $c_1>0 $ such that 
   \begin{equation*}
   \Ic(\omega)\leq  c_1\mathcal G(\omega) ,
   \end{equation*}
   where  the  function  $\mathcal G  :\Omega\to  \R^+$  is  given by
  \begin{equation*}
 \mathcal G(\omega):= \lof \dist(\omega(0),E) \cdot \exp{\Big(c_1\cdot\sup_{t\in[0,1]}\dist_P(\omega: 0,t)\Big)} ,\qquad  \omega\in\Omega. 
\end{equation*}  
Consequently, we only need  to show that  $\mathcal G$ is $\bar\mu$-integrable.

To  do this  we  write   using  formula (\ref{eq_formula_bar_mu}) 
\begin{equation}\label{eq_integral_estimate}
\int_\Omega \mathcal G(\omega)d\bar\mu(\omega) =\int_X  \lof \dist(x,E)  \cdot \Big (\int_{\Omega_x} \exp{\Big(c_1\cdot\sup_{t\in[0,1]} \dist_P(\omega:0,1)\Big)}d W_x(\omega) \Big )d\mu(x).
\end{equation}
 Next, we will show  that   the inner integral is  uniformly bounded  by  a constant $c_2>0$  independent  of $x,$  that is,
\begin{equation}\label{eq_uniform_estimate}
\int_{\Omega_x} \exp{\Big(c_1\cdot\sup_{t\in[0,1]} \dist_P(\omega:0,1)\Big)}d W_x(\omega)  <c_2.
\end{equation}
 To this  end 
  we focus on a single leaf $L$ of $\Fc$  passing through a  given point $x\in X\setminus E.$ Observe  that
 \begin{multline*}
\int_{\Omega_x}  \exp{\Big(c_1\cdot\sup_{t\in[0,1]} \dist_P(\omega:0,1)\Big)}  dW_x(\omega)\\=  \int_0^\infty  W_x\left\{\omega\in \Omega_x:\ \exp{\Big(c_1\cdot\sup_{t\in[0,1]} \dist_P(\omega:0,1)\Big)}>s  \right\} ds. 
\end{multline*}
   The  integrand on the  right-hand side is  equal  to 
 $$W_x\left\{\omega\in \Omega_x:\  \sup_{t\in[0,1]} \dist_P(\omega:0,1)> \log s/c_1 \right \}.$$
 For $0\leq  s\leq e^{c_1},$ this  quantity is  clearly  $\leq 1$ since $W_x$ is  a  probability measure by  Proposition \ref{prop_Wiener_measure}
 (i). 
 For $s\geq e^{c_1},$ this  quantity is  dominated, thanks to  
 Lemma \ref{lem_Candel},  by  $c_3 \exp{\Big ( - c_3^{-1}\big({   \log s\over c_1}\big)^2\Big)}$ for some  constant $c_3>0.$
 Since  $\int_{e^{c_1}}^\infty    \exp{\Big ( - c_3^{-1}\big({   \log s\over c_1}\big)^2\Big)} ds<\infty,$
   we have  established (\ref{eq_uniform_estimate}). 
 
 We infer from \eqref{eq_integral_estimate} and \eqref{eq_uniform_estimate} that
  \begin{equation*} 
 \int_\Omega \mathcal G(\omega)d\bar\mu(\omega)\leq c_2\int_X\lof \dist(x,E)d\mu(x)  .
 \end{equation*}
 By   assumption \eqref{e:necessary_integrability}, the integral on the  right hand is is finite.
 Hence,  the proof of the  theorem  is  complete.
 \hfill $\square$ 
 
 

\noindent  {\bf End of the proof of   Corollary \ref{C:extremal} modulo the integrability  condition \eqref{e:necessary_integrability}.} 
Using Theorem    \ref{thm_main}, we may apply  \cite[Theorem 3.7]{NguyenVietAnh1} to the  holonomy cocycle $\mathcal H$
of rank $1.$ Consequently,  we obtain   a unique  Lyapunov  exponent  function  $\lambda(T):\ X\to\R$
 which  is  measurable and  leafwise constant and  which,  for $\mu$-almost  every $x\in X,$ satisfies
 $$
\lim_{t\to\infty} {1\over t}\log\|\mathcal H(\omega,t)\|= \lambda(T)(x)
 $$
 for $W_x$-almost every path $\omega\in \Omega_x.$
Since $\mu$ is  ergodic and  the   function $\lambda(T)$ is  leafwise constant and measurable,  it follows that  for  all $a,b\in\R$ with $a\leq b,$
the $\mu$-measure of the  leafwise saturated set $\{ x\in X:\  a\leq \lambda(T)(x)\leq  b\}$ is 
either $0$ or  $\mu(X).$  Consequently, $\lambda(T)$ is 
constant $\mu$-almost  everywhere. 
The proof is  thereby completed.
\hfill $\square$

\section{Harmonic currents on the local model}
 \label{S:parametrization} 

We collect in this  section several  known  results about the mass-clustering  of harmonic measures
near  hyperbolic  singularities. More  concretely, 
    we  first recall  a special  parametrization
of leaves which is  due to  Forn\ae ss-Sibony \cite{FornaessSibony3}. Next, using  this parametrization,
we state   a mass-clustering result of harmonic measures
near  hyperbolic  singularities   which is also  due to Forn\ae ss-Sibony \cite{FornaessSibony3}.
Finally, we recall our recent  estimate  about the behaviour of some integral operators of  ``Poisson kernel" type 
near  hyperbolic  singularities. These results will  thoroughly be used in the  subsequent sections
when we  prove the  basic  estimates   stated in Section \ref  {S:Main_Theorem_2}.

 Following  \cite[Section 2]{FornaessSibony3},
consider  the foliation associated to the vector field $F(z,w) = z {\partial\over \partial z}
+ \lambda w{\partial\over \partial w}$  with some complex number $\lambda = a + ib,$ $b \not= 0.$ Note that
if we flip $z$ and $w,$ we replace $\lambda$ by $1/\lambda = \bar\lambda/|\lambda|^2 = a/(a^2 + b^2) - ib/(a^2 + b^2).$  Therefore, we
may assume without loss of generality  that  $b > 0.$ 
   We now
describe the portion of a general leaf inside $\D^2.$ 
There are two separatrices, $(w = 0),$ $ (z = 0).$ Other than that the  Riemann surface $\widehat L_\alpha$ defined  in \eqref{e:Riemann_surface_L_x} can be
reparametrized by 
\begin{equation}\label{eq_parametrization_DS}
(z,w) = \psi_\alpha(\zeta),\ z = e^{i(\zeta+(\log {|\alpha|})/b)},\ \zeta = u+iv,\ w = \alpha e^{i\lambda(\zeta+(\log {|\alpha|})/b)}.
\end{equation}
The  reader  is invited to compare this special parametrization with the ones given in \eqref{e:Riemann_surface_L_x} and \eqref{eq_leaf_equation_new}.
Consider  the new  variable
\begin{equation}\label{e:t}
 t:=bu+av.
\end{equation}
So we have 
\begin{equation}\label{eq_u,v_vs_z,w}
|z| = e^{-v}, |w| = e^{-bu-av}=e^{-t}  .
\end{equation}
Observe that as we follow $z$ once counterclockwise around the origin, $u$ increases
by $2\pi$, so the absolute value of $|w|$ decreases by the multiplicative factor of $e^{-2\pi b}.$
Hence, we cover all leaves  when  $\alpha$ ranges over $\T,$  where
\begin{equation}\label{e:cano_trans_T}  
\T:=\{ \alpha\in\C:\ e^{-2\pi b }\leq  |\alpha| \leq 1\}.
\end{equation}
 We
notice that with the above parametrization, the intersection with the unit bidisc $\D^2$
of the leaf is given by  the  domain $\{ (u,v)\in\R^2:\ v > 0, u > -av/b\}.$ 
The main   point of this special parametrization is that  the above domain  is  independent of $\alpha.$ In the $(u, v)$-plane this domain
corresponds to a sector $ S_\lambda$ with corner at $0$ and given by $0 < \theta < \arctan(-b/a)$
where the $\arctan$ is chosen to have values in $(0, \pi),$ that is,
\begin{equation}\label{e:S_lambda}
S_\lambda:=\left\lbrace \tau=re^{i\theta}\in\C:\ r>0\ \text{and}\ 0 < \theta < \arctan(-b/a)   \right\rbrace.
\end{equation}
Let $\gamma := {\pi\over
\arctan(-b/a)} .$ Then the
map 
\begin{equation}\label{eq_u,v_vs_U,V}
\phi : \tau=u+iv \mapsto \tau^\gamma=(u+iv)^\gamma =:U+iV
\end{equation}
 maps this sector to the upper half plane with coordinates $(U, V ).$
The fact that $\gamma > 1$ will be crucial, this is where the hyperbolicity of singularities is
used.

The local leaf clusters on both separatrices. To investigate the clustering on the
$z$-axis, we use a transversal $\T_{z_0} := \{(z_0,w): e^{-2\pi b }\leq |w| \leq 1\}$ for some $z_0$ with $|z_0| = 1.$ We
can normalize so that $h_\alpha(z_0,w) = 1$  for $(z_0,w)\in \T_{z_0}.$  
Solving   the  equation $(z_0,w) = \psi_\alpha(\zeta_0) = \psi_\alpha(u_0 + iv_0)$ with unknown
 variables $(u,v,\alpha)$ yields the unique
solution $u_0=-b^{-1}\ln|w|,  $ $ v_0 = 0$ and $\alpha=w.$
Consequently, by  identifying  $\alpha\in\T$ with  $(z_0,\alpha)\in\T_{z_0},$
we  may identify $\T$  with $\T_{z_0},$ and  hence   $\T$  can be  regarded as a  transversal.
We call $\T$ the {\it distinguished transversal}.
Let $T$ be  a  harmonic  current of mass $1$  directed by $\Fc.$ 
Let $\U$ be a  flow box which   admits  $\T_{z_0}$ as a transversal.
Then by Proposition \ref{prop_current_local},   we can write in $\U$
\begin{equation}\label{eq_local_description}
T=\int h_\alpha[V_\alpha] d\nu(\alpha),
\end{equation}
where, for each  $\alpha \in\T,$   $h_\alpha$ denotes the harmonic function associated to the current $T$ on the plaque  $V_\alpha $ which is  contained  in the leaf $L_\alpha.$ We still denote by $h_\alpha$
its harmonic continuation along $L_\alpha.$
 Define 
$$\tilde h_\alpha(\zeta) :=h_\alpha \left (
e^{i(\zeta+(\log |\alpha|)/b)}, \alpha e^{i\lambda (\zeta+(\log |\alpha|)/b)}\right)\quad\text{
on}\quad S_\lambda.$$
  Consider the harmonic  function  
$$\tilde H_\alpha:=\tilde h_\alpha\circ \phi^{-1}\quad\text{defined  on the upper half plane}\quad  \{  U+i V:\  V>0\}.$$   The following mass-clustering estimate 
of Fornaess--Sibony \cite{FornaessSibony3} is needed.
\begin{lemma} \label{lem_FS}  
1) The harmonic  function $\tilde H_\alpha $ is the Poisson integral of its boundary values.
So in the upper half plane  $\{  U+iV:\  V>0\}$,
$$
\tilde H_\alpha(U+iV)={1\over\pi} \int_{-\infty}^\infty \tilde H_\alpha( y){V\over V^2+( y-U)^2} d y
$$ for $\nu$-almost every $\alpha.$ Moreover,
$$
\int_{ \alpha\in\T}\int_{-\infty}^\infty\tilde H_\alpha( y)(1+ | y|)^{1/\gamma -1}d y d\nu(\alpha)<\infty.
$$
2) If, moreover,  $T$ gives no mass to every invariant analytic  curve, then  $\nu$ is  diffuse, that is, $\nu(\alpha)=0$ for every $\alpha.$
\end{lemma}
\begin{proof}
The first part is proved in \cite[Proposition 1]{FornaessSibony3}.

When  $\Fc$ has no invariant analytic  curve, the second part is  proved in 
 \cite[ Corollary 2]{FornaessSibony3}. But that  proof still  works in the more general context of Part 2)
 making the obviously  necessary changes.
\end{proof}

\section{Proof of  the  integrability condition: First reduction} \label{S:Main_Theorem_2}

   In this  section we reduce  Theorem   \ref{thm_main_2}  to  Theorem \ref{T:main_estimate}.  Let $\Fc=(X,\Lc,E)$    be  
a   holomorphic  hyperbolic foliation   with     hyperbolic singularities $E$ in a   compact complex projective  surface $X$
such that
   the foliation is  Brody hyperbolic. Let $T$ be  a  harmonic  current tangent to $\Fc.$
Fix  $x_0\in E.$
 Since $x_0$ is a hyperbolic  singular point, 
there is a  holomorphic  coordinate system $(z,w)$  near  $x_0 $ in which $x_0$ is  identified  with $0$  and 
the  foliation
$\Fc$    is
associated with the vector field $F(z,w) = z {\partial\over \partial z}
+ \lambda w{\partial\over \partial w}$  on  $\D^2$ with some complex number $\lambda = a + ib,$ $b > 0.$ So two  analytic curves $\{z=0\}$ and $\{w=0\}$  describe
 two separatrices of $\Fc$  at $x_0=0.$
 Let $\T$ be  the distinguished transversal  defined in \eqref{e:cano_trans_T}.
Consider  the  function $G:\  E\times (0,1)\to \R^+$ given by
\begin{equation}\label{e:G}
G(x,r):=  {1\over 2\pi r^2}\int_{\B(x,r)}  T\wedge i \ddbar \|y\|^2,
\end{equation} 
where $\B(x,r)$ is the ball of center $x$ and radius $r$ in $X.$
By Skoda  \cite{Skoda}, $G(x,r)$  is increasing in $r$ and $\lim_{r\to 0} G(x,r)$
is  equal to the Lelong number  of $T$ at $x.$   By our recent work  \cite{NguyenVietAnh3}, this number
is  $0,$ that is,
\begin{equation}\label{e:Lelong}
\lim_{r\to 0} G(x,r) =0.
\end{equation}
When $x=x_0,$ we  write $G(r)$ instead of $G(x_0,r).$ Using the above map $\Psi,$
we are reduced to the  local model considered in the previous section.   
For  every $s>0,$ consider the function $K_s:\ \R\to\R^+$ given by
\begin{equation}\label{eq_K_s}
K_s(y):= 
\begin{cases}
s^{1-\gamma}, & \text{if}\ s\geq (1+|y|)^{1/\gamma};\\
(1+|y|)^{1/\gamma-1}, & \text{if}\ s\leq (1+|y|)^{1/\gamma}.
\end{cases}
\end{equation}
The following result  gives a precise  estimate of  $G(r)$ in terms of the function $K_s.$  
\begin{lemma}\label{lem_estimate_G} 
There is a constant $c>0$  such that for every $0<r<1,$  we have 
$$
c^{-1}G(r)\leq \int_{ \alpha \in\T}\Big(\int_{-\infty}^\infty 
K_{-\log r}(y)\tilde H_\alpha( y)
d y \Big)d\nu(\alpha)\leq c G(r).
 $$
\end{lemma}  
\begin{proof} It follows  from combining 
  \cite[Proposition 3.5]{NguyenVietAnh3} and   \cite[Lemma 3.2]{NguyenVietAnh3}.
 \end{proof}

  We are in the position  to state  the  main estimate  of this  article.
 \begin{theorem}\label{T:main_estimate}
 There are constants   $c_0,\kappa>1$ such that for every $x\in E$ and  $0<r<1/2,$ 
$$
 G(x,r)\leq c_0|\log(-\log r)||\log r|^{-1}+c_0\int_{ \alpha \in\T}\Big(\int_{(1+|y|)^{1/\gamma}\leq - \kappa\log r} 
K_{-\log r}(y)\tilde H_\alpha( y)
d y \Big)d\nu(\alpha) .
 $$  
 \end{theorem}
 The  proof of  Theorem  \ref{T:main_estimate} will be  given  at the  end of Section \ref{S:test_curves}.
 \begin{remark}\label{R:speed}\rm
   Using   Proposition \ref{P:comparison} below (for $\delta=1$), Theorem \ref{T:main_estimate}
   is  equivalent to  the  assertion that 
\begin{multline*}
   \int_{\B(0,r)}T\wedge [z=r] \leq c_0|\log(-\log r)||\log r|^{-1}\\
+c_0\int_{ \alpha \in\T}\Big(\int_{(1+|y|)^{1/\gamma}\leq - \kappa\log r} 
K_{-\log r}(y)\tilde H_\alpha( y)
d y \Big)d\nu(\alpha) .
 \end{multline*}
 This  is  the  precise  meaning  of the speed  that  we mention  in Subsection  \ref{SS:outline} (see the  discussion
 following    \eqref{e:reduction_intro}). The  integral on  the  right hand side of the last line  decays, in some sense, 
 very quickly as $r\to 0.$  Indeed, it is, up to a multiplicative constant, equal to
$$
\int_{ \alpha \in\T}\Big(\int_{(1+|y|)^{1/\gamma}\leq - \kappa\log r} 
 (-\log r)^{1-\gamma}\tilde H_\alpha( y)
d y \Big)d\nu(\alpha).$$
Rewrite  the last line as  follows:
 $$\int_{ \alpha \in\T}\Big(\int_{(1+|y|)^{1/\gamma}\leq - \kappa\log r} 
{(1+|y|)^{1-1/\gamma}\over (-\log r)^{\gamma-1} }(1+|y|)^{1/\gamma-1} \tilde H_\alpha( y)
d y \Big)d\nu(\alpha)  .
$$
Since for every $y\in\R,$ $ {(1+|y|)^{1-1/\gamma}\over (-\log r)^{\gamma-1} }\to 0$ as $r\to 0,$
it follows from
  Lemma \ref{lem_FS}   and the  dominated convergence that  the last  integral  tends to $0$ as $r\to 0$  (see \cite{NguyenVietAnh3} for  details). 
\end{remark}

  Taking for granted  this  result,   we  arrive at  the

 \noindent{\bf End of the proof of  Theorem  \ref{thm_main_2}.}
Fix  a point $x_0\in E$ and a  holomorphic coordinate system $x=(z,w)$ as at the beginning  of this  section.
 So $x_0$ is  identified with $0\in\D^2.$  Since   the two Hermitian  metrics $g_X$ and $i \ddbar \|x\|^2$ are  equivalent on $\D^2,$ that is,
$g_X\approx i \ddbar \|x\|^2,$  we may regard
$i \ddbar \|x\|^2$   as   $g_X.$
Moreover, in the remainder of the proof, we  will write $\B_r$ (resp. $G(r)$)  instead of $\B(x_0,r)$
(resp. $G(x_0,r)$)  for $0<r<1.$
 Next, recall from 
(\ref{eq_relation_Poincare_Hermitian_metrics}) that
$$
i\ddbar \|x\|^2=\eta^2(x) g_P(x),
$$
where  we know  from Part 1) of Lemma  \ref{lem_poincare} that $\eta(x)\approx \|x\| \log\|x\|$ for $0<\|x\|<1/2.$
Therefore, we infer that 
$$\mu:= T\wedge  g_P\approx  { T\wedge i\ddbar \|x\|^2\over  \|x\|^2 (\log\|x\|)^2}\quad\text{on}\  \B_{1/2}.$$ 
Moreover,   we infer  from \eqref{e:G}  that for every smooth function $h:\ [0,1]\to \R^+,$
$$
\int_{\B_{1/2}} { T\wedge i\ddbar \|x\|^2\over h(\|x\|)} =
\int_0^{1/2} {d(r^2G(r))\over h(r)}.
$$
Consequently,
$$
\int_{\B_{1/2}} | \lof \dist(x,E)| \cdot d\mu(x)\approx \int_{\B_{1/2}}{ T\wedge i\ddbar \|x\|^2\over  \|x\|^2 (\log\|x\|)}
=\int_0^{1/2} {d(r^2G(r)) \over  -r^2 \log r}.
$$
 Performing an integration by part to the last expression yields that
\begin{eqnarray*}
 \int_0^{1/2}  { d(r^2G(r)) \over - r^2 \log r}
 &=&\left \lbrack { G(r) \over -  \log r}\right\rbrack^{1/2}_0\\
&-&2 \int_0^{1/2}  { G(r)dr \over r \log r}-\int_0^{1/2}  { G(r)dr \over r (\log r)^2}.
\end{eqnarray*}
Since  $G(r)$ tends to the Lelong number of $T$ at $0$ as $t\to 0,$ the  expression in brackets
is  finite. Therefore, in order to show  that $
\int_{\B_{1/2}} | \lof \dist(x,E)| \cdot d\mu(x)<\infty,$  it suffices to prove that
\begin{equation}\label{eq_reduction}
 \int_0^{1/2}  { G(r)dr \over - r \log r}<\infty.
\end{equation}
The remaining part is  devoted to the proof of (\ref{eq_reduction}). 
By   Theorem \ref{T:main_estimate}, the integral in  (\ref{eq_reduction}) is  bounded by a constant times
$(I)+(II),$  where
$$
I:=\int_0^{1/2}  { |\log(-\log r)| dr \over r |\log r|^2}<\infty,
$$
and by Fubini's theorem,
$$
II:=\int_{ \alpha \in\T}\Big(\int_{y=-\infty}^\infty  
\big (\int_{(1+|y|)^{1/\gamma}\leq -  \kappa\log r}  { K_{-\log r}(y)dr  \over - r \log r}\big) \tilde H_\alpha( y)
d y \Big)d\nu(\alpha).
 $$ 
On the other hand,  we infer  from  \eqref{eq_K_s}  the  existence of a  constant $c>0$ such that 
for all $y\in \R,$  
$$
       \int_{ s\geq \kappa^{-1} (1+|y|)^{1/\gamma}}  s^{-1} K_s(y)ds\leq  c(1+ | y|)^{1/\gamma -1} . 
$$
 Performing  the  change of variable $s:=-\log r$ in the last line, the most inner integral
 of $(II)$
is  dominated  by a constant times  $(1+ | y|)^{1/\gamma -1}.$ 
Consequently,  $(II)$ is  bounded by
$$
\int_{ \alpha \in\T}\int_{-\infty}^\infty\tilde H_\alpha( y)(1+ | y|)^{1/\gamma -1}d y d\nu(\alpha),
$$
which is  finite  by Part 1) of Lemma  \ref{lem_FS}.
 This  completes the proof of (\ref{eq_reduction}), and hence  the proof of the theorem.
 \hfill $\square$
 \begin{remark} \label{rem_log_integrability} \rm
As  remarked in the Introduction,  the method  employed in  Dinh-Nguyen-Sibony \cite{DinhNguyenSibony1}
seems to only  give   a  weaker  inequality
  $$
\int | \lof \dist(x,E)|^{1-\delta}\cdot( T\wedge g_P)(x)<\infty, \qquad \delta>0.
  $$ 
 Indeed, arguing as in the proof of Theorem  \ref{thm_main_2} and using  the weight $| \lof \dist(x,E)|^{1-\delta}$
instead of  $  | \lof \dist(x,E)|,$ the above  inequality
 is  reduced  to  the following one
$$
 \int_0^{1/2}  { G(r)dr \over - r (\log r)^{1+\delta}}<\infty.
$$ 
  In \cite{DinhNguyenSibony1}    $G(r)$ is  replaced  by a positive constant, and hence  the above integral is finite
  if and only if $\delta>0.$
 \end{remark}
 
\section{Geometric intersection and  interpretations} \label{S:intersection_interpretation}
 
 Let $\Fc=(X,\Lc,E)$    be  
a   holomorphic  hyperbolic foliation   with     hyperbolic singularities $E$ in a   compact complex   
surface $X.$ Let $T$ be a harmonic  current tangent to $\Fc,$  and let $\Cf$ be an analytic curve  on an open subset $\U\subset X.$ The main purpose of the section is  to give
a  reasonable  meaning  to  the intersection measure $T \wedge[\Cf],$ and to   obtain a  procedure
in order to estimate the mass of the last measure.  
We are inspired by the recent works in 
  \cite{FornaessSibony1,FornaessSibony2,FornaessSibony3}.

 Let $\U\simeq \B\times\T$ be a flow
box which is relatively compact in $X\setminus E.$
Let $\mathfrak{C}$ be an analytic  curve  on $\U$ such that
for every $\alpha\in\T,$  $\mathfrak{C}$ intersects the plaque $V_\alpha$ at
at most one point  (which is possibly a multiple point). We say  that $\mathfrak{C}$   is  {\it transversal} in  $\U.$
We  define the {\it geometric  intersection } of $T$  and $[\mathfrak{C}]$
as  the positive Radon measure  on $\U$  given by:
\begin{equation}\label{e:local_inter}
\langle T\wedge [\mathfrak{C}],\phi\rangle=\langle T\wedge [\mathfrak{C}],\phi\rangle|_{\U}  :=\int_{\alpha\in\T:\  \xi_\alpha\not=\varnothing}
h(\xi_\alpha)\phi(\xi_\alpha) d\nu(\alpha),
\end{equation}
where $\phi$ is a continuous test  function compactly supported in $\U,$ and

$\bullet$ $\xi_\alpha:= V_\alpha\cap\mathfrak{C}$ if this intersection is  non empty and $\xi_\alpha=\varnothing$
otherwise;

$\bullet$ the decomposition consisting of 
 the  positive Radon measure  $\nu$ on $\T,$ 
 and the positive harmonic function $h_\alpha$ on $\B$
 for $\nu$-almost every $\alpha\in \T$   is given by  Proposition  \ref{prop_current_local}. 

The  reader can easily check  the following result.
\begin{proposition} \label{P:local_inter}
$T\wedge [\mathfrak{C}]$ is  a well-defined positive Radon measure on $\U.$
It is independent of the choice of a decomposition given by  Proposition  \ref{prop_current_local}.
Its mass is
 $$
 \|T\wedge [\mathfrak{C}]\| =\|T\wedge [\mathfrak{C}]\|_{\U}=\int_{\alpha\in\T:\  \xi_\alpha\not=\varnothing}
h(\xi_\alpha)  d\nu(\alpha)<\infty.
 $$
\end{proposition}

Now  let $\U$ be an  an arbitrarily  open subset of $X$ and  $\mathfrak{C}$ an analytic curve on $\U.$
We say  that $\mathfrak{C}$   is  {\it almost transversal} in  $\U$ if $\mathfrak{C}$ intersects with  each plaque in every  regular flow box in $\U$  transversally at  at most finite points.
We leave  the reader  to verify the following result.
 \begin{lemma}\label{L:local_transversal}
  $\mathfrak{C}$ is almost transversal  if and only if  $\mathfrak{C}$ is locally transversal in $\U,$ that is,  for every $x\in\mathfrak{C}\cap \U,$ there is
a flow box $\U_x\subset \U$ containing $x$  such that
 $\mathfrak{C}$ is transversal in $\U_x.$ 
 \end{lemma}
 Assume that $\mathfrak{C}$ is almost transversal. 
By Lemma \ref{L:local_transversal}, there is 
 an at most countable cover $\Uc:=(\U_j)_{j\in J}$  of $\U\setminus E$ by its open subsets  such that
$\Uc$ is locally finite and that  each $\U_j$ ($j\in J$) is  a flow  box   which is  relatively compact in $\U\setminus E$ and that  $\mathfrak{C}$ is  transversal in $\U_j.$
 Let $\Theta:=(\theta_j)_{j\in J}$ be a  partition of unity subordinate to
$\Uc.$

The mass of the intersection $T\wedge [\mathfrak{C}]$
is  
$$
\|T\wedge [\mathfrak{C}]\|=\sum_{j\in J}  \langle T\wedge [\mathfrak{C}], \theta_j\rangle|_{\U_j}\in[0,\infty].  
$$
Apparently,  the mass $\|T\wedge [\mathfrak{C}]\|$  depends on the choice of  a cover $\Uc$ and a partition
of unity $\Theta.$ However, it turns out that this mass is  independent of such a choice. More precisely,
we can show the following properties. 
\begin{proposition}\label{P:e:global_inter}
\begin{enumerate}
\item[(i)] The mass $\|T\wedge [\mathfrak{C}]\|$ does not depend on any choice we made. 
\item[(ii)] If $\U\cap E=\varnothing,$ then  $\|T\wedge [\mathfrak{C}]\|<\infty.$
\item[(iii)]  When $\|T\wedge [\mathfrak{C}]\|<\infty,$  we   define the {\it geometric  intersection } of $T$  and $[\mathfrak{C}]$
as  the positive Radon measure  on $\U$  given by:
 \begin{equation}\label{e:global_inter}
\langle T\wedge [\mathfrak{C}],\phi\rangle:= \sum_{j\in J} \langle T\wedge [\mathfrak{C}],\theta_j\phi\rangle_{\U_j},
\end{equation}
where $\phi$ is a continuous test  function compactly supported in $\U,$
\item[(iv)]  When $\|T\wedge [\mathfrak{C}]\|<\infty,$ 
 the measure  $T\wedge [\mathfrak{C}]$ defined by \eqref{e:global_inter}
 does not depend on any choice of $\Uc$ and $\Theta$  we made.
\end{enumerate}
\end{proposition} 

Next, we  prove a  cohomological invariant property.
 
 \begin{proposition}\label{P:coh_inv}
 Let $\Cf$ and $\Df$ be  two algebraic  curves on $X$ which are cohomologous  (in the cohomology group  $H^{1,1}(X,\R)$). Suppose that $\Cf\cap E=\Df\cap E=\varnothing$
 and that both $\Cf$ and $\Df$ are almost transversal. 
 Then $\|  T\wedge [\Cf] \|_X=\|  T\wedge [\Df] \|_X.$ 
 \end{proposition}
 \begin{proof} 
Since $\Cf\cap E=\Df\cap E=\varnothing$ and  both $\Cf$ and $\Df$ are almost transversal, we  may find
 a finite  cover $\Uc:=(\U_j)_{j\in J}$  of $ X$ by its open subsets  such that
 
 $\bullet$ if $\U_j\cap E\not=\varnothing,$ then this intersection is  a single point and
$\Cf\cap \U_j=\Df\cap \U_j=\varnothing;$
 
 $\bullet$ 
   each $\U_j$ with $U_j\cap E=\varnothing$ is  a regular flow  box    such that both  $\Cf$
and $\Df$  are  transversal in $\U_j.$
 Let $(\theta_j)_{j\in J}$ be a  partition of unity subordinate to
$\Uc.$   
 
Consider a  smooth  Hermitian metric $\|\cdot\|$ on the line  bundle generated  by the  divisor
$[\Cf]$ (resp.  $[\Df]$) on $X.$ Let  $\sigma$ (resp. $\sigma'$) be a  holomorphic  section
having   $[\Cf]$ (resp.  $[\Df]$) as  its  divisor. Then
$$\phi:=\log \|\sigma\|\qquad\text{and}\qquad\psi:=\log \|\sigma'\|
 $$
 are  quasi-plurisubharmonic functions on $X.$
 Recall here that a quasi-plurisubharmonic function is
locally   the sum of a plurisubharmonic  function and  a smooth one.
 Lelong-Poincar\'e formula  says that
  \begin{equation}\label{e:Lelong-Poincare}
 [\Cf]=  i\ddbar \phi+\Theta\qquad\text{and}\qquad [\Df]=  i\ddbar \psi+\Theta',
 \end{equation}
  where  $\Theta$ and  $\Theta'$  are some closed smooth real $(1,1)$-forms   on $X.$ 
Since  
    $[\Cf]$  and
 $[\Df]$ are  cohomologous, it follows that  so are  $\Theta$ and     $\Theta'.$
So by the $\ddbar$-lemma  for compact K\"ahler manifolds, there is a smooth real function $u$
on $X$  such that $\Theta'-\Theta=i\ddbar u.$ Therefore, replacing  the metric $\|\cdot\|$ 
on the line bundle
associated  with $[\Cf]$  by $\| \cdot\|e^{-2u},$ we may assume without loss of generality  that
  $\Theta'=\Theta.$

Observe  that $\phi$  (resp. $\psi$)
  is  smooth outside 
  the curve $ \Cf $ (resp. $\Df$). 
Since
$\Cf\cap E=\Df\cap E=\varnothing,$ we infer that
   both $\phi$ and $\psi$
  are smooth in a  neighborhood of $E.$
 
 Fix the following  decreasing sequence (as $\epsilon\searrow 0$) of quasi-plurisubharmonic smooth functions $(\phi_\epsilon)_{0<\epsilon<1}$ (resp.
  $(\psi_\epsilon)_{0<\epsilon<1}$) on $X:$
\begin{equation*}
\phi_\epsilon:={1\over 2}\log (\|\sigma\|^2+\epsilon)\qquad\text{and}\qquad\psi:={1\over 2}\log (\|\sigma'\|^2+\epsilon).
\end{equation*}
Observe  that 
$$\lim_{\epsilon\to 0+} \phi_\epsilon =\phi\qquad\text{and}\qquad   \lim_{\epsilon\to 0+} \psi_\epsilon =\psi$$  and that  there  is  
  a  closed smooth real $(1,1)$-form $\Xi$ on $X$ such that
$$i\ddbar\phi_\epsilon\geq \Xi \qquad\text{and}\qquad     i\ddbar\psi_\epsilon\geq \Xi$$
  in the sense of currents and independent of $\epsilon.$

If  $\U_j\cap E\not=\varnothing,$ we deduce from
   the  first  $\bullet$ above as well as the properties of $\phi$ and $\psi$ discussed  above  that 
$  \phi_\epsilon$ (resp. $\psi_\epsilon$) converges uniformly  to  $\phi$
 (resp. $\psi$)  as $\epsilon\searrow 0$ on  $\U_j.$

 If  $\U_j\cap E\not=\varnothing,$ we need the following result whose proof  will be given later on.  
\begin{lemma}\label{L:uniform}
For every $\U_j\simeq \B_j\times \T_j\in\Uc$
with $\U_j\cap E=\varnothing,$ we  have  that  
 \begin{equation}\label{e:L1_uniform}
\lim_{\epsilon\to 0} \sup_{\alpha\in \T}\|\phi_\epsilon-\phi\|_{L^1(V_\alpha)}= 0
\quad\text{and}\quad  \lim_{\epsilon\to 0} \sup_{\alpha\in \T}\|\psi_\epsilon-\psi\|_{L^1(V_\alpha)}= 0.
\end{equation}
\end{lemma}
Resuming the proof of Proposition \ref{P:coh_inv},
 let $\chi$ be  a continuous test function on $X.$ In what  follows  we  drop the index $j$
for simplicity, e.g. we will write $\U\simeq \B\times \T,$ $\theta$  instead of  $\U_j\simeq \B_j\times \T_j,$ $\theta_j$  respectively.
Let $\U\in\Uc$ be  such that $\U\cap E=\varnothing.$ Write $\U\simeq \B\times \T.$
Using \eqref{e:Lelong-Poincare} and  noting that  $\Theta'=\Theta,$ and applying  Lelong-Poincar\'e formula on each plaque $V_\alpha,$ $\alpha\in \T$ of $\U,$   we get  that
 \begin{eqnarray*}
  h(\xi_\alpha)(\theta\chi)(\xi_\alpha)
  &=& \langle\Theta|_{V_\alpha}+i\ddbar\phi|_{V_\alpha}, h\theta\chi\rangle|_{V_\alpha}\\
  &=&\langle \Theta|_{V_\alpha}h\theta\chi\rangle|_{V_\alpha}+
 \langle i\ddbar\phi|_{V_\alpha}, h\theta\chi\rangle|_{V_\alpha}\\
  &=&\langle \Theta|_{V_\alpha}h\theta\chi\rangle|_{V_\alpha}+\lim_{\epsilon\to 0}\langle i\ddbar\phi_\epsilon|_{V_\alpha}, h\theta\chi\rangle|_{V_\alpha}\\
&=&\langle \Theta|_{V_\alpha}h\theta\chi\rangle|_{V_\alpha}+
 \lim_{\epsilon\to 0}\langle \phi_\epsilon|_{V_\alpha},i\ddbar( h\theta\chi)\rangle|_{V_\alpha},
  \end{eqnarray*}
where the  third  equality holds since $\phi_\epsilon\to\phi$ weakly  on $V_\alpha,$ and the last  equality is obtained  by  Stokes' theorem.
Since  \eqref{e:L1_uniform} says that  the last limit is    uniform  in $\alpha,$
we can integrate both extreme sides of the last chain of equalities with respect to the measure $d\nu$ and  obtain that
  \begin{eqnarray*}
 \langle T\wedge [\mathfrak{C}],\theta\chi\rangle_{\U}
 &=&  \int_{\alpha\in\T} h(\xi_\alpha)(\theta\chi)(\xi_\alpha)d\nu(\alpha)\\
  &=&  \int_{\alpha\in\T}\langle \Theta|_{V_\alpha},h\theta\chi\rangle|_{V_\alpha}d\nu(\alpha)
 +\int_{\alpha\in\T}\lim_{\epsilon\to 0}\langle \phi_\epsilon , i\ddbar(h\theta\chi)\rangle|_{V_\alpha}d\nu(\alpha)\\
 &=&\langle  T\wedge \Theta, \theta\chi\rangle+\lim_{\epsilon\to 0}\int_{\alpha\in\T} \langle i\ddbar\phi_\epsilon  ,  h\theta\chi\rangle|_{V_\alpha}d\nu(\alpha) \\
 &=& \langle  T\wedge \Theta, \theta\chi\rangle+\lim_{\epsilon\to 0}\langle T\wedge i\ddbar\phi_\epsilon  , \theta\chi\rangle.
  \end{eqnarray*}
 On the other hand,  for  $\U\in\Uc$ with $\U\cap E\not=\varnothing,$ we  have that
  \begin{equation*}
 \langle T\wedge [\mathfrak{C}],\theta\chi\rangle_{\U}=0
    =  \langle  T\wedge \Theta, \theta\chi\rangle+\lim_{\epsilon\to 0}\langle T\wedge i\ddbar\phi_\epsilon  ,\theta\chi\rangle,
  \end{equation*}
  where we use the first $\bullet$  above and  the fact that $\phi$ is  smooth on a neighborhood of the support  of $\theta\chi.$
  
  Summing up the above  equalities over  all $\U\in\Uc$ and using   Definition  \eqref{e:global_inter},
we infer that
$$
\langle T\wedge [\mathfrak{C}], \chi\rangle=\langle  T\wedge \Theta,  \chi\rangle+\lim_{\epsilon\to 0}\langle T\wedge i\ddbar\phi_\epsilon  , \chi\rangle.
$$
When $\chi\equiv 1,$ the  last equality  becomes
$$
\| T\wedge [\mathfrak{C}]\|_X=\langle  T, \Theta  \rangle+\lim_{\epsilon\to 0}\langle T ,i\ddbar\phi_\epsilon   \rangle= \langle  T, \Theta  \rangle,
$$
 where the last equality is obtained since $\langle T ,i\ddbar\phi_\epsilon   \rangle=0$ 
as $T$ is  harmonic and $\phi_\epsilon$ is smooth on $X.$
   Hence, $
\| T\wedge [\mathfrak{C}]\|_X=\langle  T, \Theta  \rangle.$
Similarly, we also  get that   $
\| T\wedge [\mathfrak{D}]\|_X=\langle  T, \Theta \rangle.$
The proof is  thereby completed.
  \end{proof}
\smallskip
\noindent {\bf  End of the  proof of Lemma \ref{L:uniform}.}
  We only need  to show  that
\begin{equation}\label{e:unif_red}
\lim_{\epsilon\to 0} \sup_{\alpha\in \T}\|\phi_\epsilon-\phi\|_{L^1(V_\alpha)}= 0
\end{equation}
since the  other assertion can be proved  similarly.
 Since $\U_j\cap E=\varnothing,$ the second  $\bullet$ above  says that $\Cf$ is  transversal in $\U_j.$
Therefore, we are reduced to the following model where $ \U_j\simeq \B_j\times \T_j\simeq
(1/2\D)^2$ and  
$$\Cf\cap \U_j=\left\lbrace  (w,f(w)):\qquad w\in 1/2\D \right\rbrace,$$
where $f:\  1/2\D\to 1/2\D$  is a holomorphic  function.    

In this  model, we  see easily that modulo a smooth function
$
\phi_\epsilon(z,w)={1\over 2}\log(|z-f(w)|^2+\epsilon)$ for $(z,w)\in(1/2\D)^2. $
So \eqref{e:unif_red} becomes
$$
\sup_{w\in 1/2\D} \int_{z\in 1/2\D} \big({1\over 2}\log(|z-f(w)|^2+\epsilon)-\log|z-f(w)|\big) idz\wedge d\bar z\to 0\quad\text{as}
\quad \epsilon\searrow 0.
$$
Since the  right hand side is  bounded from above  by 
$$
 \int_{z\in\D} {1\over 2}\big(\log(|z |^2+\epsilon)-\log|z|\big) idz\wedge d\bar z
$$
and that this integral  converges to $0$ as $\epsilon\searrow 0,$ the desired estimate  follows. 
\hfill  $\square$
 
 \medskip
 
 The  rest of the section is devoted to the case when the foliation $\Fc$ on the open set $\U$ is
holomorphically equivalent to the foliation associated with the vector field $F$
in $\D^2$  introduced in Section  \ref{S:parametrization}.
  So  we are in the local  model  considered  in  Section  \ref{S:parametrization}    and $0\in\U=\D^2.$
 We keep the notation introduced in Section  \ref{S:parametrization}.
 Recall that $\T\simeq \{\alpha\in\C:\ e^{-2\pi b}\leq |\alpha|\leq 1\}.$
 Let $\mathfrak{C}$ be an analytic  curve  on $\D^2$  which is  locally transversal in $\D^2.$ 
 For every $\alpha \in \T ,$
 let $\{\xi_{\alpha_j}:\ j\in J_\alpha\}$ be the set of all intersections of 
  $\mathfrak{C}$ with the Riemann surface $\widetilde{L}_\alpha.$
 We make the following convention  $J_\alpha:=\{0,1,\ldots, n_\alpha\}$ with $n_\alpha\in\N\cup\{\infty\}.$
 Continuing Proposition \ref{P:e:global_inter} we can prove the following result.
 \begin{proposition}\label{P:inter_model}
  \begin{enumerate}
  \item[(i)]  The following equality holds
  $$
  \|T\wedge [\mathfrak{C}]\|=\int_{\alpha\in\T} \sum_{j\in J_\alpha} h_\alpha(\xi_{\alpha,j}) d\nu(\alpha).
  $$
   \item[(ii)] If  $
  \|T\wedge [\mathfrak{C}]\|<\infty,$ then the measure $T\wedge [\mathfrak{C}]$ can be extended to a 
continuous linear form on the space $\Cc^b(\D^2)$ of uniformly bounded  continuous functions on $\D^2$ as follows:
$$ \langle T\wedge [\mathfrak{C}],\phi\rangle=\int_{\alpha\in\T} \sum_{j\in J_\alpha} h_\alpha(\xi_{\alpha,j})\phi(\xi_{\alpha,j}  ) d\nu(\alpha),\qquad \phi\in\Cc^b(\D^2).
$$  
\end{enumerate}
  \end{proposition}
 
 For  every $\alpha\in \T$ and $j\in J_\alpha,$  write, using \eqref{eq_parametrization_DS} 
 and \eqref{eq_u,v_vs_U,V},
 \begin{equation}\label{e:change_var_discrete}
 \begin{split}
 \xi_{\alpha,j}&=\psi_\alpha(\zeta_{\alpha,j}),\qquad  \zeta_{\alpha,j}=u_{\alpha,j}+iv_{\alpha,j},\\
 U_{\alpha,j}+iV_{\alpha,j}&:=(u_{\alpha,j}+iv_{\alpha,j})^\gamma.
\end{split} 
\end{equation}
 Recall from Section    \ref{S:parametrization} that the harmonic function
 $\tilde h_\alpha(\zeta) :=h_\alpha \left (
 \psi_\alpha(\zeta)\right)$
is defined on $S_\lambda$  and that the harmonic  function  $\tilde H_\alpha:=\tilde h_\alpha\circ \phi^{-1}$ is defined  in the upper half plane  $\{  U+i V:\  V>0\}.$ 
 Applying the Poisson representation formula   the upper half plane yields that
\begin{equation}\label{e:Poisson}
h_\alpha(\xi_{\alpha,j})=\tilde h_\alpha(\zeta_{\alpha,j})=\tilde H_\alpha( U_{\alpha,j}+iV_{\alpha,j})=
{1\over\pi} \int_{-\infty}^\infty \tilde H_\alpha( y){V_{\alpha,j}\over V_{\alpha,j}^2+( y-U_{\alpha,j})^2} d y.
 \end{equation}
 For $\nu$-almost every $\alpha\in \T,$  write
 \begin{equation}\label{e:slice_mass}
  \|T\wedge [\mathfrak{C}]\|_\alpha:=
 {1\over\pi}\int_{-\infty}^\infty \tilde H_\alpha( y)\sum_{j\in J_\alpha}{V_{\alpha,j}\over V_{\alpha,j}^2+( y-U_{\alpha,j})^2} dy.
 \end{equation}
 We obtain the  following formula
 \begin{equation}\label{e:mass_vs_slice-mass}
\|T\wedge [\mathfrak{C}]\|=\int_{\alpha\in\T} \|T\wedge [\mathfrak{C}]\|_\alpha d\nu(\alpha).
 \end{equation}
 
 Recall from \eqref{e:S_lambda} the sector $S_\lambda$ in the upper-half plane.
 \begin{proposition}\label{P:interpretation}
 Let $c,\rho>1$ and $m>0$ be three constants.
 For $\nu$-almost  every $\alpha\in\T$ assume  that
\begin{enumerate}
\item[$\bullet$]
there is  a $\Cc^1$-map $\chi^\alpha:\ D^\alpha\to  S_\lambda$ be defined on a closed interval $D^\alpha\subset \R$
such that $ c^{-1}\leq |(\chi^\alpha)'(t)|\leq c;$

  \item[$\bullet$] there is a sequence of points  
  $ (t_{\alpha,j})_{j\in J_\alpha}\subset  D^\alpha$    such that the  intervals 
$[t_{\alpha,j}-\rho^{-1}m,t_{\alpha,j}+\rho^{-1}m]$ for $j\in J_\alpha$ are  pairwise disjoint
 and that  $$
  \bigcup_{j\in J_\alpha} [t_{\alpha,j}-\rho^{-1}m,t_{\alpha,j}+\rho^{-1}m]\subset D^\alpha\subset
   \bigcup_{j\in J_\alpha} [t_{\alpha,j}-\rho m,t_{\alpha,j}+\rho m].
  $$
\end{enumerate}
  Write, using \eqref{eq_parametrization_DS} 
 and \eqref{eq_u,v_vs_U,V}, for $t\in D^\alpha,$
 \begin{equation}\label{e:change_var_cont}
 \begin{split}
 \xi_{\alpha}(t)&=\psi_\alpha(\chi^\alpha(t)),\qquad  \chi^{\alpha}(t)=u_{\alpha}(t)+iv_{\alpha}(t),\\
 U_{\alpha}(t)+iV_{\alpha}(t)&:=(u_{\alpha}(t)+iv_{\alpha}(t))^\gamma.
\end{split} 
\end{equation}
This  is the continuous  version of \eqref{e:change_var_discrete}.
 Consider the function $K^\alpha:\ \R\to\R^+$ and the  real number $\kappa\in\R^+$  defined by
\begin{equation}\label{e:interpretation}
\begin{split}
K^\alpha(y)&:={1\over m} \int_{D^\alpha}{V_{\alpha}(t)\over V_{\alpha}(t)^2+( y-U_{\alpha}(t))^2},\qquad y\in\R;\\
\kappa&:={1\over \pi}\int_{-\infty}^\infty \tilde H_\alpha(y) K^\alpha(y)dy.
\end{split}
\end{equation}
\begin{itemize}
\item[1)] If, moreover, $ c^{-1}  h_\alpha(\xi_{\alpha,j})\leq  h_\alpha(\xi_\alpha(t)) \leq  c  h_\alpha(\xi_{\alpha,j}) $
for all $j\in J_\alpha,$ $ t\in [t_{\alpha,j}-\rho m,t_{\alpha,j}+\rho m],$ then
$$
    c^{-2}\rho^{-1}\kappa\leq \|T\wedge [\mathfrak{C}]\|\leq  c^2\rho  \kappa.
$$
\item[2)]
If, moreover, $$c^{-1}  {V_{\alpha,j}\over V_{\alpha,j}^2+( y-U_{\alpha,j})^2}\leq
 {V_{\alpha}(t)\over V_{\alpha}(t)^2+( y-U_\alpha(t))^2}
\leq  c {V_{\alpha,j}\over V_{\alpha,j}^2+( y-U_{\alpha,j})^2}
  $$ for all $j\in J_\alpha,$ $ t\in [t_{\alpha,j}-\rho m,t_{\alpha,j}+\rho m]$ and $ y\in\R,$ 
 then 
 $$
 c^{-2}\rho^{-1} K^\alpha(y)\leq \sum_{j\in J_\alpha}{V_{\alpha,j}\over V_{\alpha,j}^2+( y-U_{\alpha,j})^2}\leq c^2\rho K^\alpha(y),\qquad y\in\R.
 $$
 In particular, the concluding estimate of Part 1) holds.
 \end{itemize}
 \end{proposition}
 \begin{proof} The idea  is to  approximate  a Riemann sum of a function   by its integral. The proof  follows easily from  
 Proposition \ref{P:inter_model}
 and \eqref{e:change_var_discrete}--\eqref{e:mass_vs_slice-mass}.
 \end{proof}
 
 \begin{definition}\label{D:interpretation}
 \rm  If the  assumption of Part 1)   of Proposition \ref{P:interpretation} holds, then
 we  say that $(K^\alpha)_{\alpha\in\T}$ given by \eqref{e:interpretation} is an {\it interpretation} of the  geometric intersection
$ T\wedge [\mathfrak{C}]$ on $\U$  with  parametrization $(\chi^\alpha)_{\alpha\in\T}$ and
with {\it  size}  $(c,\rho,m).$ Moreover, $m$ is  said to be 
the {\it mesh} of the interpretation.
 
 If the  assumption  of Part 2) of  Proposition \ref{P:interpretation} holds, then
 we  say that $(K^\alpha)_{\alpha\in\T}$ is a {\it coherent interpretation} of the  geometric intersection
$ T\wedge [\mathfrak{C}]$ on $\U.$
 \end{definition}
The following result   
  studies  the behavior of the Poisson kernel ${V\over V^2+( y-U)^2}$ in terms of $u$ and  $v.$ 
\begin{lemma}\label{lem_Poisson_kernel} {\rm (Nguy\^en \cite[Lemma 3.3]{NguyenVietAnh3})}
There are constants $c_1, c_2,c_3 >1$ large enough  with $c_3>c_2$ such that the following properties hold
for all $(u,v)\in\R^2$ with $\min\{v, bu+av\}\geq 1.$
\\
1)  $$  {1\over c_1}\leq {(\max\{v, bu+av\})^\gamma\over \sqrt{V^2+U^2}}\leq c_1\ \text{and}\   {1\over c_1}\leq {(\max\{v, bu+av\})^{\gamma-1}\min\{v, bu+av\} \over V}\leq c_1 .  $$
\\
2)  If  $\max\{v, bu+av\}\geq  c_2(1+|y|)^{1/\gamma},$ then
$$
{1\over c_1}  {\min\{v, bu+av\}\over (\max\{v, bu+av\})^{\gamma+1}}\leq {V\over V^2+( y-U)^2}\leq  c_1  {\min\{v, bu+av\}\over (\max\{v, bu+av\})^{\gamma+1}}.
$$
3)  If $\max\{v, bu+av\}\leq  c_2^{-1}(1+|y|)^{1/\gamma},$ 
then
$$
{1\over c_1}  {  V\over (1+|y|)^2}\leq {V\over V^2+( y-U)^2}\leq  c_1  { V\over (1+|y|)^2}.
$$
4)  If $ c_2^{-1}   (1+|y|)^{1/\gamma}\leq v, bu+av\leq  c_2(1+|y|)^{1/\gamma},$
then
$${1\over c_1}  { 1\over (1+|y|)}\leq {V\over V^2+( y-U)^2}\leq  c_1  { 1\over (1+|y|)}.$$
5)  If $ \min\{v, bu+av\}\leq c_3^{-1}   (1+|y|)^{1/\gamma}$ and  
$c_2^{-1}   (1+|y|)^{1/\gamma}\leq  \max\{v, bu+av\}  \leq  c_2(1+|y|)^{1/\gamma},$
then
$$
 {1\over c_1}\leq {V\over V^2+( y-U)^2}: {(1+|y|)^{1/\gamma-1}\min\{v, bu+av\}   \over (\min\{v, bu+av\} )^2+  (\max\{v, bu+av\} -\rho )^2 } \leq  c_1  ,
$$
where   $\rho$ is a  real number which depends only on $y$ and on  $t:= \min\{v, bu+av\}$
which satisfies 
$ c_2^{-1}   (1+|y|)^{1/\gamma}\leq \rho\leq  c_2(1+|y|)^{1/\gamma}.$

In fact, $\rho(y,t)$ is defined  as  follows. When $c_3>1$ is large  enough,   then for  every $1\leq t\leq c_3^{-1}(1+|y|)^{1/\gamma},$
there exists  a    solution  $u:=u(y,t),\ v:=v(y,t)$ of the following   equation
$$  U=y,\qquad\text{where}\  U+iV=(u+iv)^\gamma$$
which satisfies
   $ c_2^{-1}(1+|y|)^{1/\gamma}\leq \max\{v(y,t), bu(y,t)+av(y,t)\}\leq c_2(1+|y|)^{1/\gamma}.$
  So we  define  $$\rho(y,t):=bu(y,t)+av(y,t).$$ 
\end{lemma}

 \section{Test  curves $\Cf_r,$ $\Cf_{r,N}\ldots$ and  second reduction}
 \label{S:test_curves}

We first introduce some families of algebraic  curves on $X$ and  a family of analytic  curves on 
an open neighborhood of a given  singular  point of $\Fc.$ 
 Next, we state basic estimates  and  deduce the main estimate
 from the former ones.  The proof of the basic estimates will be developed in  subsequent sections.   

Since $X$ is  projective, we may find  a  finite  family of surjective holomorphic maps $\Psi_j:\ X\to\P^2,$ $1\leq j\leq s,$ such that  for every point $x\in X,$ there is  at least one
map $\Psi_j$ which is locally biholomorphic at $x.$ Indeed, it suffices to embed $X$ into  $\P^N$ with $N$ large enough, and  choose a family of central projections from $X$ onto $\P^2.$

Now  fix $x_0\in E$ and assume that $\Psi:=\Psi_{j_0}:\ X\to\P^2$ is locally biholomorphic at $x_0.$ 
Moreover, suppose  without loss of generality that
$\Psi$ maps an open neighborhood $\V$ of $x_0$ biholomorphically onto the bidisc $\D^2\hookrightarrow \P^2$
and that $\Psi(x_0)=0$  with $0:=(0,0)\in\C^2.$
Let $(Z,W)$ be the    canonical coordinates of the  canonical  injection $\C^2\hookrightarrow\P^2,$
i.e.  $\C^2\simeq \{[Z:W:1]:\ (Z,W)\in\C^2\}\subset \P^2.$

  Since $x_0$ is a hyperbolic  singular point, we may assume  without loss of generality that 
there are holomorphic  coordinates $(z,w)$ defined on $\D^2$ such that $(z(0),w(0))=(0,0)=0$ and that the
the  foliation
$(\Psi|_{\V})_*\Fc$    is
associated with the vector field $$F(z,w) = z {\partial\over \partial z}
+ \lambda w{\partial\over \partial w}\quad\text{  with some complex number}\quad \lambda = a + ib,\quad b \not= 0.$$
So two  analytic curves $\{z=0\}$ and $\{w=0\}$  describe
 two separatrices $(\Psi|_{\V})_*\Fc$  at $0.$ By performing a  linear change of  coordinates, we may suppose  without loss of generality that
the complex line  $\{Z=0\} $  (resp.   $\{W=0\} $)
is tangent  to  the separatrice   $\{z=0\}$ (resp. $\{w=0\}$) at $0.$
By  dilating   the coordinates $(Z,W)$ if necessary, we may assume  without loss of generality that
the Jacobian matrix of $(Z,W)$ over $(z,w)$ at $(0,0)$  is the identity matrix, i.e.,
\begin{equation}\label{e:Jac}
 \left( \begin{array}{cc}
{\partial Z\over\partial z}(0,0) & {\partial Z\over\partial w}(0,0) \\
{\partial W\over\partial z}(0,0) &  {\partial W\over\partial w}(0,0) \end{array} \right)
=\left( \begin{array}{cc}
 1 &  0 \\
 0 &  1 \end{array} \right).
\end{equation}
In this  work we use  both systems of coordinates $(Z,W)$ and $(z,w).$ Each system  has  its  own advantages
and  drawbacks. Indeed, the coordinates $(Z,W)$ appears to be very useful  in our cohomological  argument, but this
argument  is   only of global nature. On the  opposite side, 
 although   the coordinates $(z,w)$  are not appropriate for a global  argument as the cohomological one, they seem to be  very convenient for doing a local  analysis near  singular points.

 Recall that $\Psi$ maps $\V$ biholomorphically onto $\D^2.$
 By shrinking $\D^2$   if necessary, the holomorphic  implicit function theorem, applied to  $\{z=0\}$,  allows us to write  for $(Z,W)\in\D^2,$
  \begin{equation}\label{e:z_expansion}
 z=\theta(Z,W)z_\infty(Z,W) ,
 \end{equation}
 where  
 $\theta(Z,W),\ z_\infty(Z,W)$ are holomorphic  functions on $\D^2$ with  
 $$z_\infty(Z,W)=Z-\sum_{j=2}^\infty a_j W^j,\qquad   a_j\in\C,$$
  and 
  \begin{equation}\label{e:z_expansion_bis}
1/2<|\theta(Z,W)|<2\quad\text{ on}\quad \D^2\quad  \text{and}\quad \theta(0,0)=1.
  \end{equation}
 
 \noindent{\bf $\bullet$ Analytic curves $\Cf_r$:}
 For  every $r\geq 0$ small  enough, let $\Cf_r$ be  the complex  analytic curve   in $\V$ given by
\begin{equation}\label{e:Cf_r}
\left\lbrace  x\in \V :\ z(\Psi(x))=r    \right\rbrace.
\end{equation}
Clearly, $\Cf_r\cap E=\varnothing$ for $r\not=0.$ 

 \noindent{\bf $\bullet$ Analytic curves $\Cf^d_N$:}
 For  every $N\in\N$ with $N>2$ and $d\in\C\setminus\{0\},$  let $\Cf^d_N$ be  the complex  analytic curve   in $\V$ given by
\begin{equation}\label{e:Cf_N}
\left\lbrace  x\in \V :\ z(\Psi(x))= d(w(\Psi(x)))^N    \right\rbrace.
\end{equation}
Clearly, $\Cf^d_N\cap E=\{x_0\}.$ 

\noindent{\bf $\bullet$ Algebraic curves $\Cf'_{r,N},$ $\Cf_{r,N}$:}
Given $r\geq 0$ and   $N\in \N$ with $N>2,$ 
 define $\Cf'_{r,N}$ to be  the  algebraic  curve  in $\P^2$ which
is the closure in $\P^2$
of the  following  affine   curve  
\begin{equation}\label{e:Cf'_r,N}
\left\lbrace (Z,W)\in\C^2:\ z_N(Z,W)=r   \right\rbrace\subset \C^2,
\end{equation}
where $z_N(Z,W)$ is the Taylor expansion  of order $N$   of $z(Z,W),$ i.e., 
\begin{equation}\label{e:z_N_expansion}
z_N(Z,W):= Z -\sum_{j=2}^{N-1}a_j W^j,\qquad  (Z,W)\in \D^2.
\end{equation}
Let $\Cf_{r,N}$ be  the  algebraic  curve  in $X$ given by
\begin{equation}\label{e:Cf_r,N}
\Cf_{r,N}:=(\Psi|_{\V})^*(\Cf'_{r,N}).
\end{equation}

Basic  geometric  properties of these  algebraic  curves  are collected in the following.
\begin{proposition}\label{P:cohomology}
For every $N\in \N$ with $N>2,$  there exists $0\leq r_N<1/2$  such that  
 \begin{enumerate}
\item[(i)] $\Cf_{0,N}\cap E=\{x_0\};$
 \item[(ii)]
  $\Cf_{r,N}\cap E=\varnothing $ for every  $0<r\leq r_N;$  
  \item[(iii)] $[\Cf_{r,N}]$ is  cohomologous  to $[\Cf_{0,N}]$ in $X$ for every  $0\leq r\leq r_N.$
\end{enumerate} 
\begin{proof}
First, recall  the  equation $
\Cf_{r,N}:=(\Psi|_{\V})^*(\Cf'_{r,N}).$
Consequently, observe that   $x_0\in 
\Cf_{0,N}$ as  $0\in  \Cf'_{0,N}$
and that  $x_0\not\in 
\Cf_{r,N}$ for $r>0$  as $0\not\in   \Cf'_{r,N}$
for   $r\not=0.$
This  discussion, combined with
 the fact that $E$ is a  finite set, implies that
 for $r_N>0$  small enough,  both properties (i) and (ii) are satisfied.
Finally, property (iii) follows from \eqref{e:Cf_r,N}
 and the fact that  two algebraic curves $\Cf'_{r,N}$ and $\Cf'_{0,N}$ of the  same degree $N$
 are cohomologous in $\P^2.$
\end{proof} 
\end{proposition}

 Let $
\rho_a:=  \big( \limsup_{j\to\infty} |a_j|^{1/j}\big)^{-1}\in (0,\infty].$
 So $\rho_a$ is the radius of convergence of the analytic function 
$z_N(Z,W)$ defined in \eqref{e:z_N_expansion}.
 Clearly,  $\rho_a\not=\infty,$ otherwise  the non-constant holomorphic map
$\C\ni W\mapsto (\sum_{j=2}^\infty a_j W^j,W)\subset \{z=0\}$ 
 contradicts   our assumption.
 \begin{remark}\rm  
 Together  with Lemma \ref{lem_poincare}, this is the place where the Brody  hyperbolic  assumption has fully been used
 (see also Remark \ref{R:Brody_sufficiency}). 
 \end{remark}
 For the  sake of clarity, we   may assume without loss of generality that
 $\rho_a=1.$
 In the sequel  we fix a  sequence  $N_j\nearrow\infty$ such that
 \begin{equation}\label{e:choice_N}
 \lim_{j\to\infty}  |a_{N_j}|^{1/N_j}= \limsup_{j\to\infty} |a_j|^{1/j}=\rho^{-1}_a = 1,
\end{equation}
 and we
always  choose $N=N_j$  for some $j$ large enough.

For  $r>0,$ recall  that $\B_r$ denotes the ball  centered at $0$ with radius $r$ in $\D^2\hookrightarrow X.$ 
The  basic  estimates  which are the main ingredients for the proof of     Theorem  \ref{T:main_estimate}
are  stated in the following four propositions. Their proofs will be  established in the  subsequent four sections.
 \begin{proposition}\label{P:comparison}
For every  $0<\delta<1,$ there is
  $c_\delta>1$ such that for every  harmonic   current $T,$  
$$
 c_\delta^{-1} G(x_0,r)\leq  \| T\wedge [\Cf_r]\|_{    \B_{r^\delta}}\leq c_\delta G(x_0,r) ,\qquad  0<r<1/2.
$$
Here  $G(x_0,r)$ is defined  by \eqref{e:G}. 
\end{proposition}
 
 \begin{proposition} \label{P:comparison_bis}
 For every $N$ large  enough   in the  sequence $(N_j)_{j=1}^\infty$ given in \eqref{e:choice_N}, there is
  a constant $c=c_N>1$ such that for every  harmonic   current $T,$  
$$
  \| T\wedge [\Cf_{0,N}]\|_{    \B_{r}}\leq c G(x_0,r) ,\qquad  0<r<1/2.
$$
 \end{proposition}

  \begin{proposition}\label{P:separatrices}
For every $N$ large  enough   in the  sequence $(N_j)_{j=1}^\infty$ given in \eqref{e:choice_N},  there are constants $c=c_N>1$ and $0<r_N<1/2$  such that
for every $0<r<r_N,$
 \begin{equation*} 
\Big | \|  T\wedge [\Cf_{0,N}] \|_{    \B_{r^{1/N}|\log r|^{-3/N}}}-\| T\wedge [\Cf_r]  \|_{\B_{r^{1/N}|\log r|^{-3/N}}  }\Big|  \leq c   | \log (-\log r)|  |\log r|^{-1}.
\end{equation*}
 \end{proposition}

\begin{proposition}\label{P:key_interpretations}
Let $N=N_j$ be large  enough in the  sequence $(N_j)_{j=1}^\infty$ given in \eqref{e:choice_N}.
 Then  the geometric intersection   $T\wedge [\Cf_r]$ (resp.  $T\wedge [\Cf_{0,N}]$) on  $    \B_{r^{1/N}|\log r|^{-3/N}}$ admits a coherent interpretation $(K^\alpha)_{\alpha\in\T}$  of the  form
$K^\alpha:=K_{-\log r,N}$  (resp. a  coherent interpretation   $({K^*}^\alpha)_{\alpha\in\T}$  of the  form ${K^*}^\alpha:=K^*_{-\log r,N}$).  
Here $$\R\ni y\mapsto K_{-\log r,N}(y)\qquad\text{ and}\qquad  \R\ni y\mapsto K^*_{-\log r,N}(y)$$  are functions such that
 there  are   constants $c,\kappa>1$ independent of $N$
  and a constant $c_N >1$  with  the following properties:
\begin{enumerate}
\item[(i)]
for $ (1+|y|)^{1/\gamma}\leq  \kappa^{-1} s,$ we have 
$ K^*_{s,N}(y) \leq  c     s^{1-\gamma}$ and   
$$c^{-1}\leq {K_{s,N}(y)\over  N^{\gamma-1} s^{1-\gamma}}\leq  c;$$  
\item[(ii)]
for $(1+|y|)^{1/\gamma}\geq \kappa s,$ we have 
$$c^{-1}\leq {K^*_{s,N}(y)\over   N(1+ | y|)^{1/\gamma -1}}\leq c \quad\text{and}\quad 
c^{-1}\leq {K_{s,N}(y) \over  (1+ | y|)^{1/\gamma -1}}\leq c  ;$$
\item[(iii)]
for $\kappa^{-1}s \leq (1+|y|)^{1/\gamma}\leq \kappa s,$ we have 
$$c^{-1}_N\leq {K^*_{s,N}(y)\over  (1+ | y|)^{1/\gamma -1}}\leq  c_N \quad\text{and}\quad  c^{-1}_N\leq {K_{s,N}(y) \over  (1+ | y|)^{1/\gamma -1}}\leq c_N  .$$  
\end{enumerate}
 \end{proposition}

 Now we  are in the position  to reduce  the proof of    Theorem \ref{T:main_estimate} to those  of 
Propositions
 \ref{P:separatrices} and  \ref{P:key_interpretations}   modulo Propositions \ref{P:comparison}, \ref{P:comparison_bis}. This is  the  second reduction.

\noindent{\bf  End of the proof of  Theorem \ref{T:main_estimate}.} Let $N\geq 1$ be large  enough.
By Proposition  \ref{P:separatrices},   
  there are constants $c_N$ and $r_N$   such that
for every $0<r<r_N,$
 \begin{equation*} 
 \|  T\wedge [\Cf_{0,N}] \|_{    B_{r^{1/N}(-\log r)^{-1/N}}}-\| T\wedge [\Cf_r]  \|_{B_{r^{1/N}(-\log r)^{-1/N}}  } 
  \leq c|\log(-\log{r})| |\log r|^{- 1}.
\end{equation*}
  By Proposition \ref{P:key_interpretations},
the  geometric intersections   $T\wedge [\Cf_r]$  
and   $T\wedge [\Cf_{0,N}]$ on  $    \B_{r^{1/N}|\log r|^{-3/N}}$ admit  coherent interpretations
$K_{-\log r,N}$ and $K^*_{-\log r,N}$ respectively.
Consequently, there  are two  functions $\vartheta,\ \vartheta^*:\ \R\to  [c'^{-1},c']$ for some
constant $c'>1$ such that 
  the above  inequality 
 can be  rewritten    as follows:
  \begin{equation*}
\int_{ \alpha \in\T}\Big(\int_{-\infty}^{\infty} 
(\vartheta^*(y)K^*_{-\log r,N}(y) -\vartheta(y)K_{-\log r,N}(y)) \tilde H_\alpha( y)
d y \Big)d\nu(\alpha)\leq  c|\log(-\log{r})| |\log r|^{- 1} .
 \end{equation*}
This  implies that
\begin{equation}\label{e:I_II}
I_2\leq  I_1+ I_3+ c|\log(-\log{r})| |\log r|^{- 1},
\end{equation}
where
\begin{equation*}
I_k:=  \Big |\int_{ \alpha \in\T}\Big(\int_{D_k} 
(\vartheta^*(y)K^*_{-\log r,N}(y) -\vartheta(y)K_{-\log r,N}(y)) \tilde H_\alpha( y)
d y \Big)d\nu(\alpha)\Big |,
\end{equation*}
with $D_1:=\{y\in\R:\ \kappa(1+|y|)^{1/\gamma}\leq -\log r\},$ and $D_2:= \{y\in\R:\ (1+|y|)^{1/\gamma}\geq -\kappa\log r\},$
  and  $D_3:=\{y\in\R:\ \kappa^{-1}(1+|y|)^{1/\gamma}\leq -\log r\leq
\kappa(1+|y|)^{1/\gamma}\}.$

Now we apply Proposition \ref{P:key_interpretations} (i)-(ii)-(iii)  to $I_1,$ $I_2$ and $I_3$
respectively.  Let  $N$ be large enough  in the  sequence \eqref{e:choice_N}  which  also satisfies
$N^{\min\{1,\gamma-1\}}\geq 2c^2.$ So we have that
\begin{eqnarray*}
I_1&\leq&(c^{-1}N^{\gamma-1}-c)\int_{\alpha\in\T}\Big (\int_{D_1} (-\log r)^{1-\gamma} \tilde H_\alpha(y)dy 
\Big)d\nu(\alpha),\\
I_2&\geq&(c^{-1}N-c)\int_{\alpha\in\T}\Big (\int_{D_2} ( 1+|y|)^{1/\gamma-1} \tilde H_\alpha(y)dy 
\Big)d\nu(\alpha),\\
I_3&\leq&(c'c_N-c'^{-1}c_N^{-1})\int_{\alpha\in\T}\Big (\int_{D_3} ( 1+|y|)^{1/\gamma-1} \tilde H_\alpha(y)dy 
\Big)d\nu(\alpha).
\end{eqnarray*}
 This, combined   with  \eqref{e:I_II} and  \eqref{eq_K_s}, implies 
that
\begin{multline*}
\int_{ \alpha \in\T}\Big(\int_{D_2} 
K_{-\log r}(y)\tilde H_\alpha( y)
d y \Big)d\nu(\alpha)
\lesssim |\log(-\log{r})| |\log r|^{- 1} \\
+\int_{ \alpha \in\T}\Big(\int_{D_1}
K_{-\log r}(y)\tilde H_\alpha( y)
d y \Big)d\nu(\alpha)
+\int_{ \alpha \in\T}\Big(\int_{D_3} 
K_{-\log r}(y)\tilde H_\alpha( y)
d y \Big)d\nu(\alpha).
\end{multline*}
Hence,
\begin{multline*}
\int_{ \alpha \in\T}\Big(\int_{D_1\cup D_2\cup D_3} 
K_{-\log r}(y)\tilde H_\alpha( y)
d y \Big)d\nu(\alpha)
\lesssim |\log(-\log{r})| |\log r|^{- 1}\\
 +\int_{ \alpha \in\T}\Big(\int_{D_1\cup D_3}
K_{-\log r}(y)\tilde H_\alpha( y)
d y \Big)d\nu(\alpha)
.
\end{multline*}
Since  we know  by Lemma \ref{lem_estimate_G} that  the left-hand side of the last line is equivalent to $G(r),$
   the  desired conclusion of 
 the  theorem follows  when  the constant $c_0$ is  large  enough.
\hfill $\square$

\section{Cohomological relation and third reduction}\label{S:coho}

We first state    several  basic estimates. Next, using  these estimate we establish a cohomological
invariance result (see Proposition \ref{P:coho_mass}). Finally, we deduce from this result Proposition
 \ref{P:separatrices}.
 Consequently,  modulo Propositions
 \ref{P:comparison}, \ref{P:comparison_bis},  \ref{P:mass_1}, \ref{P:mass_2}, the proof of    Theorem \ref{T:main_estimate} is  reduced to that  of 
Proposition \ref{P:key_interpretations}.
 This is  the  last reduction.

\begin{proposition}\label{P:mass_1}
For every $N$ large  enough in the  sequence $(N_j)_{j=1}^\infty$ given in \eqref{e:choice_N},
  there exist constants $ c=c_N>1,\delta=\delta_N>0$  and  a  constant $r_N$
satisfying  the conclusion of Proposition  \ref{P:cohomology}  with the  following properties.  
For every  $0<r<r_N $    
and for every   harmonic   current $T$ tangent to $\Fc$ of mass $1,$  
the following   mass estimates hold:
\begin{align}
\Big | \|  T\wedge [\Cf_{r,N}]\|_{X\setminus  \D^2} -  \|  T\wedge [\Cf_{0,N}]\|_{X\setminus  \D^2}\Big| &\leq cr^\delta,\label{e:mass_estimates_1}\\
\Big| \| T\wedge [\Cf_{r,N}]   \|_{  \D^2\setminus  \B_{r^{1/N}|\log r|^{3/N}}}-\| T\wedge [ \Cf_{0,N}] \|_{  \D^2\setminus  \B_{r^{1/N}|\log r|^{3/N}}}\Big |&\leq c  |\log r|^{-1}, \label{e:mass_estimates_2}\\
\Big | \|  T\wedge [\Cf_{r,N}] \|_{    \B_{r^{1/N}|\log r|^{-3/N}}}-\| T\wedge [\Cf_r]  \|_{    \B_{r^{1/N}|\log r|^{-3/N}}}\Big|  &\leq c |\log r|^{-1}.\label{e:mass_estimates_3}
\end{align}
\end{proposition}
We postpone  the proof of  Proposition \ref{P:mass_1} to Section  \ref{S:outside_corona}.

For $0<r<1/2$ and $N\geq 2,$  consider the corona
$$
\A_{r,N}:=  \B_{r^{1/N}|\log r|^{3/N}}\setminus   \B_{r^{1/N}|\log r|^{-3/N}}.
$$
So we obtain the following  partition of $X:$
\begin{equation}\label{e:partition}
X=(X\setminus\D^2)\coprod (\D^2\setminus   \B_{r^{1/N}|\log r|^{3/N} })\coprod \A_{r,N}\coprod  \B_{r^{1/N}|\log r|^{-3/N}}.
\end{equation}

\begin{proposition}\label{P:mass_2} For every $N$ large  enough in the  sequence  $(N_j)_{j=1}^\infty$ given in  \eqref{e:choice_N},
   there are constants $0<r_N\ll 1$ and $c=c_N>1$ such that  for every $0<r<r_N,$ 
  \begin{align}
 \| T\wedge [\Cf_{r}]\|_{\A_{r,N}}  &\leq c |\log (-\log r) | |\log r|^{-1},\label{e:mass_estimates_4}
  \\
 \| T\wedge [\Cf_{0,N}]\|_{\A_{r,N}}  & \leq c       |\log (-\log r)|  |\log r|^{-1},\label{e:mass_estimates_5}\\
  \| T\wedge [\Cf_{r,N}]\|_{\A_{r,N}}  &\leq c    | \log (-\log r)|  |\log r|^{-1}.\label{e:mass_estimates_6}
\end{align}
\end{proposition}
The proof  of Propostion  \ref{P:mass_2}  will  occupy  Section \ref{S:on_corona}.

 Taking for granted  these estimates, we want to prove  the following 
cohomological
invariance result.
\begin{proposition}\label{P:coho_mass}
   Let $r_N$ be given by Proposition  \ref{P:cohomology} for every $N\in \N$ with $N>2.$ Then  
 for every  $0<r<r_N $ and for every   harmonic   current $T,$   we have that
$ \|[\Cf_{r,N}]\wedge T\|_X= \|[\Cf_{0,N}]\wedge T\|_X.$  
  \end{proposition}
This result does  not follow from Proposition \ref{P:coh_inv}
 since $\Cf_{0,N}\cap E\not=\varnothing.$ 
 We need  the following   auxiliary result.
\begin{lemma}\label{L:mass_cont} We have $
\lim_{r\to 0+}\|  T\wedge [\Cf_{r,N}] \|_X=\|  T\wedge [\Cf_{0,N}] \|_X.$
\end{lemma}
\begin{proof}
Combining   estimates \eqref{e:mass_estimates_1}--\eqref{e:mass_estimates_2}, we get that
 \begin{multline*}
\Big| \| T\wedge [\Cf_{r,N}]   \|_X-\| T\wedge [ \Cf_{0,N}] \|_X\Big |\leq c  |\log r|^{-1}\\
+ \| T\wedge [\Cf_{r,N}]   \|_{  \B_{r^{1/N}|\log r|^{3/N}}}+\| T\wedge [ \Cf_{0,N}] \|_{  \B_{r^{1/N}|\log r|^{3/N}}}.
\end{multline*}
 In the  remainder of the proof we will  that the two terms in the last line tend to $0$ as $r\to 0+.$ 
This  will imply the lemma.

 Applying \eqref{e:mass_estimates_6} yields that
$$
\| T\wedge [\Cf_{r,N}]   \|_{  \B_{r^{1/N}|\log r|^{3/N}}}\leq   c |\log(-\log r)| |\log r|^{-1}+
\| T\wedge [\Cf_r]  \|_{    \B_{r^{1/N}|\log r|^{-3/N}}} .   
$$
Consequently, we infer that
\begin{eqnarray*}\lim_{r\to 0+}   \| T\wedge [\Cf_{r,N}]   \|_{\B_{r^{1/N}|\log r|^{3/N}}}
&\leq&  \lim_{r\to 0+} \| T\wedge [\Cf_r]  \|_{    \B_{r^{1/N}|\log r|^{-3/N}}}\\
&\leq&   \lim_{r\to 0+} \| T\wedge [\Cf_r]  \|_{   \B_{r^{1/2N}}}=0,
\end{eqnarray*}
where the last limit holds by Proposition \ref{P:comparison} applied to $\delta=1/(2N).$
Hence, $\lim_{r\to 0+}   \| T\wedge [\Cf_{r,N}]   \|_{\B_{r^{1/N}|\log r|^{3/N}}}=0.$

On the other hand, applying \eqref{e:mass_estimates_5} yields that
$$
\| T\wedge [ \Cf_{0,N}] \|_{  \B_{r^{1/N}|\log r|^{3/N}}}\leq     c |\log(-\log r)| |\log r|^{-1}+
\| T\wedge [ \Cf_{0,N}] \|_{  \B_{r^{1/N}|\log r|^{-3/N}}}.$$
Therefore, we deduce that
\begin{eqnarray*}
\lim_{r\to 0+} \| T\wedge [ \Cf_{0,N}] \|_{  \B_{r^{1/N}|\log r|^{3/N}}}&\leq&\lim_{r\to 0+} \| T\wedge [ \Cf_{0,N}] \|_{  \B_{r^{1/N}|\log r|^{-3/N}}}\\
&\leq  & \lim_{r\to 0+} \| T\wedge [ \Cf_{0,N}] \|_{  \B_{r^{1/N}|}}=0,
\end{eqnarray*}
where the last limit holds by Proposition \ref{P:comparison_bis} applied to $\delta=1/N.$
Hence, $\lim_{r\to 0+} \| T\wedge [ \Cf_{0,N}] \|_{  \B_{r^{1/N}|\log r|^{3/N}}}=0.$
\end{proof}

\noindent{\bf End of the  proof of Proposition \ref{P:coho_mass}.}
By Lemma \ref{L:mass_cont}, we have that
$$\|  T\wedge [\Cf_{0,N}] \|_X=\lim_{s\to 0+}\|  T\wedge [\Cf_{s,N}] \|_X.$$
On the other hand, by  Proposition \ref{P:cohomology} (ii)--(iii),
  $\Cf_{r,N}\cap E=\varnothing ,$ $\Cf_{s,N}\cap E=\varnothing, $ and    both $[\Cf_{r,N}]$  and
 $[\Cf_{s,N}]$ are  cohomologous.
Consequently, by Proposition  \ref{P:coho_mass},  the  right hand side of the last limit is  equal to $\|  T\wedge [\Cf_{r,N}] \|_X.$ 
 \hfill $\square$
 
\noindent{\bf End of the proof of Proposition \ref{P:separatrices}.}
Fix $N\in \N$ with $N>2$  and  let $r_N$ be given by Proposition \ref{P:mass_1} and Proposition  \ref{P:cohomology}.
So by Proposition \ref{P:coho_mass}, we have, for  $0<r<r_N,$
that
$$
 T\wedge [\Cf_{0,N}]= T\wedge [\Cf_{r,N}].
$$ 
 This, combined  with   estimates  \eqref{e:mass_estimates_1}--\eqref{e:mass_estimates_2}
and the partition \eqref{e:partition}, implies that
 \begin{equation*} 
\Big| \| T\wedge [\Cf_{r,N}]   \|_{  \B_{r^{1/N}|\log r|^{3/N}}}-\| T\wedge [ \Cf_{0,N}] \|_{ \B_{r^{1/N}|\log r|^{3/N}}}\Big |\leq c   | \log (-\log r)|  |\log r|^{-1}.
\end{equation*}
Putting this  together  with estimates \eqref{e:mass_estimates_5} and \eqref{e:mass_estimates_6} yields that
 \begin{equation*} 
\Big |  \|  T\wedge [\Cf_{r,N}] \|_{ \B_{r^{1/N}|\log r|^{-3/N} }  }-\| T\wedge [\Cf_{0,N}]  \|_{ 
\B_{r^{1/N}|\log r|^{-3/N} }   }\Big|  \leq c | \log (-\log r)|  |\log r|^{-1}.
\end{equation*}
This, coupled with  \eqref{e:mass_estimates_3}, gives that
 \begin{equation*} 
\Big | \|  T\wedge [\Cf_{0,N}] \|_{    \B_{r^{1/N}|\log r|^{-3/N}}}-\| T\wedge [\Cf_r]  \|_{\B_{r^{1/N}|\log r|^{-3/N}}  }\Big|  \leq c | \log (-\log r)|  |\log r|^{-1}.
\end{equation*}
 This  completes the proof.
 \hfill $\square$

\section{Intersection  of test curves with a leaf}\label{S:intersection_with_leaves}
In Section  \ref{S:test_curves}  we introduce the analytic curves
$\Cf_r,$ $\Cf^d_N$ which are defined on a neighborhood of a singular point of the foliation,  and the  algebraic curves $\Cf_{r,N}$ which are defined on the whole $X.$  
The main purpose of this  section is  to study the  distributions of these test curves with  the leaves of the foliations near  singularities.
 Therefore, in what follows, we restrict ourselves  to the  local model  of    Section \ref{S:test_curves}
and  Section \ref{S:parametrization}, and  we  keep the notation introduced therein.   
More specifically,   we may assume  without loss of generality that $x_0\equiv 0\in\D^2$ and that
there are holomorphic  coordinates $(z,w)$ defined on $\D^2$ such that $(z(0),w(0))=(0,0)=0$ and that the
the  foliation
$(\Psi|_{\V})_*\Fc$    is
associated with the vector field $F(z,w) = z {\partial\over \partial z}
+ \lambda w{\partial\over \partial w}$  with some complex number $\lambda = a + ib,$ $b \not= 0.$
Note that two  analytic curves $\{z=0\}$ and $\{w=0\}$  describe
 two separatrices $(\Psi|_{\V})_*\Fc$  at $0.$ Recall that 
  $\T\simeq \{\alpha\in\C:\  e^{-2\pi b}\leq |\alpha|\leq 1\}.$  
 
The  distribution of the intersection points of   $\mathfrak{C}_r$ with  a leaf  in the bidisc $\D^2$
is  quite simple  as the following result shows.
  \begin{lemma}\label{L:intersection_C_r} For each $0<r<1$ and each  $\alpha\in\T,$  
   the intersection  of $\frak{C}_r$ with the  Riemann surface
 $\widehat{L}_\alpha$ can be
parametrized, via \eqref{eq_parametrization_DS}, by 
\begin{equation}\label{e:para-inter_C_r}
\xi_{r,\alpha,k}=\psi_\alpha(\tau_{r,\alpha,k}),\quad\text{where}\quad 
\tau_{r,\alpha,k}=u_{\alpha,k}+iv_r:= 2k\pi- (\log {|\alpha|})/b+i(-\log r),\quad k\in \Z.
\end{equation}
\end{lemma}
\begin{proof}
Let $(z,w)=\psi_\alpha(\tau)$ with $\tau=u+iv$ be an intersection point of $\Cf_r$  with  the  Riemann surface $\widehat{L}_\alpha.$
Then $\tau$ is a solution of the  equation $r=z=e^{i(\tau+(\log  {|\alpha|})/b)}.$
Solving  this equation gives all the solutions \eqref{e:para-inter_C_r}.
\end{proof}

\begin{lemma}\label{L:para-inter_z=w^N}
Let $N\in\N\setminus\{0\}$  and $d\in\C\setminus\{0\}.$
 For each  $\alpha\in\T,$    
   the intersection  of $\Cf^d_N=\{z=dw^N\}$ with the Riemann surface
 $\widehat{L}_\alpha$ can be
parametrized, via \eqref{eq_parametrization_DS}, by 
\begin{equation}\label{e:para-inter_z=w^N}
\xi_{N,\alpha,k}=\psi_\alpha(\tau_{N,\alpha,k}),\quad\text{where}\quad 
\tau_{N,\alpha,k}=u_{N,\alpha,k}+iv_{N,\alpha,k} ,\quad k\in \N,
\end{equation}
and  $(u_{N,\alpha,k},v_{N,\alpha,k})$  is  the  unique solution of the following system
of linear equations:
\begin{equation}\label{e:system_u_and_v}
\begin{cases}
-(Na-1)u +Nbv  &= 2\pi k+  \arg d+ N\arg \alpha + b^{-1}(Na-1) \log|\alpha|\\
Nbu +(Na-1) v &=  \log|d|.
\end{cases}
\end{equation}
Moreover, let
\begin{equation}\label{e:t_N_alpha_k}
t_{N,\alpha,k}:= bu_{N,\alpha,k}+av_{N,\alpha,k},\qquad k\in\N.
\end{equation}
Then there  are  constants $v_{N}, t_{N}$  such that
\begin{equation}\label{e:difference_u_v}
v_{N,\alpha,k+1} -v_{N,\alpha,k}= v_N\quad\text{and}\quad t_{N,\alpha,k+1} -t_{N,\alpha,k}= t_N
\quad\text{for}\quad k\in\ \N\end{equation}
and that
\begin{equation}\label{e:approx_pace_u_v}
v_{N}\approx N^{-1}\quad\text{and}\quad t_{N}\approx N^{-2}.
\end{equation}
\end{lemma}
\begin{proof}
Let $(z,w)=\psi_\alpha(\tau)$ with $\tau=u+iv$ be an intersection point of $\Cf^c_N$  with  the  Riemann surface $\widehat{L}_\alpha.$
Then we deduce from   $
z=dw^N$ that $\tau$ is a solution of the  equation $$
  e^{i(\tau+(\log  {|\alpha|})/b)}=d\alpha^N e^{iN\lambda (\tau+(\log  {|\alpha|})/b)}.
$$
So there is  $k\in\Z$ such that
$$
i(\tau+(\log  {|\alpha|})/b)= 2i\pi k+\log |d|+i\arg d +N\log |\alpha| +iN \arg \alpha +iN(a+ib) (\tau+(\log  {|\alpha|})/b).
$$
Equating the  imaginary and the real parts of both sides, we obtain system \eqref{e:system_u_and_v}.

Writing $u_N:=u_{N,\alpha,k+1} -u_{N,\alpha,k},$ we infer from   \eqref{e:difference_u_v}
and  system \eqref{e:system_u_and_v}  that $(u_N,v_N)$ is a solution of the following system
\begin{equation*} 
\begin{cases}
-(Na-1)u +Nbv  &= 2\pi  \\
Nbu +(Na-1) v &=  0 .
\end{cases}
\end{equation*}
So we get that
$$
u_N={-2\pi(Na-1)\over (Na-1)^2+(Nb)^2},\quad v_N={2\pi Nb\over (Na-1)^2+(Nb)^2},\quad t_N ={2\pi b\over (Na-1)^2+(Nb)^2}.
$$
This  proves \eqref{e:difference_u_v}
 and \eqref{e:approx_pace_u_v}.
\end{proof}

 The  following result plays a vital role in this section. It allows us to approximate the function
 $z_N$ defined \eqref{e:z_N_expansion} efficiently. Consequently, we infer from this  result  a  good picture
 of the distributions of the intersection of $\Cf_{r,N}$ with a general leaf near singularities.
 \begin{proposition}\label{P:z_N_expression}
Let  $\N\ni N\mapsto M_N\in\N$ be  a sequence such that $\lim_{N\to\infty}M_N=\infty.$ Then,  for every $N\in\N$ large enough in the
sequence $(N_j)_{j=1}^\infty$ given in  \eqref{e:choice_N}, 
 there is a constant    $0<s=r_{N,M}<1$ 
such that 
 the  analytic functions 
$$z_\infty(z,w):=z_\infty(Z(z,w),W(z,w))\qquad \text{and}\qquad z_N(z,w):=z_N(Z(z,w),W(z,w)) ,$$
$z_\infty(Z,W)$ (resp. $z_N(Z,W)$) being the analytic function given in \eqref{e:z_expansion} (resp. \eqref{e:z_N_expansion}), are well-defined on  $\D^2$ (resp. $\D_s^2$) and  that
 the following two properties hold:
\begin{enumerate}
\item[(i)] For every   $w\in \D_s$ and every $0\leq r\leq  |w|/2,$ the equation
$z_N(z,w)=r$ with $|z|\leq |w|$   admits a unique  solution.

\item[(ii)] For every point  $(z,w)\in \D_s^2$  with  $z_N(z,w)=r$ for some $0\leq r<1,$    at least one of the following  two  items holds:
 \begin{enumerate}
 \item[(ii-a)]  $|z-r|\leq 4r^2$ and $|w|\leq r$ 
   
    \item[(ii-b)]  $
 |z_N(z,w)  -(z_\infty(z,w)+a_Nw^N)|\leq  M^{-1}_N|a_N||w|^N . $  
 \end{enumerate}
 \end{enumerate}
 \end{proposition}
 \begin{proof} 
 As  in Section  \ref{S:test_curves}, we may suppose  without loss of generality that
the complex line  $\{Z=0\} $  (resp.   $\{W=0\} $) in $\P^2$
is tangent  to  the separatrice   $\{z=0\}$ (resp. $\{w=0\}$) at $0.$
Therefore, arguing as in the proof of  \eqref{e:z_expansion}, we obtain the following  equation for $W:$  
 \begin{equation}\label{e:W_expansion}
 W=W(z,w)=\vartheta(z,w) \big(w+\sum_{j=2}^\infty b_j z^j\big),\qquad   b_j\in\C,
 \end{equation}
 where $\vartheta$ is  a holomorphic  function  on $\D^2$ such that
 \begin{equation}\label{e:W_expansion_bis}
1/2< |\vartheta(z,w)|<2\quad\text{on}\quad \D^2\quad\text{and}\quad \vartheta(0,0)=1.
 \end{equation}
 Let $\rho\in\R^+$ be such that
$ \limsup_{j\to\infty} |b_j|^{1/j} <\rho.$
This together with \eqref{e:choice_N}  gives an integer $N$ with $N>N_0$ and a constant $c>1$  such that
\begin{equation}\label{e:control_a_j_and_b_j}
|a_N|>2^{-N},\quad  |a_j|<2^j\ \text{for}\ j\geq N,\quad |b_k|<c\rho^k\ \text{for}\ k\geq N.
\end{equation}
Inserting \eqref{e:z_expansion} and \eqref{e:W_expansion} into \eqref{e:z_N_expansion}, we get that
\begin{equation}\label{eq_z_N_full_expansion} 
z_N(z,w)= z_\infty(z,w)+ \sum_{j=N}^\infty a_j W^j =  z_\infty(z,w)+ 
 \sum_{j=N}^\infty a_j \vartheta^j(z,w)( w+\sum_{k=2}^\infty b_k z^k   )^j.
\end{equation}
Now we prove  assertion (i).  We infer from \eqref{eq_z_N_full_expansion} and \eqref{e:control_a_j_and_b_j} and \eqref{e:W_expansion_bis} that  for $s$ small enough and $|z|\leq |w|\leq  s,$ 
\begin{eqnarray*}
|z_N(z,w)-z_\infty(z,w)|\leq  \sum_{j=N}^\infty |a_j| |\vartheta(z,w)|^j( |w|+\sum_{k=2}^\infty |b_k||w|^k   )^j = O(|w|^2)\ll |w|.
\end{eqnarray*}
This, combined with $r\leq |w|/2,$ implies that  
\begin{equation}\label{e:z_N_vs_z_infty}
|z_N(z,w)-z_\infty(z,w)|<|z_\infty(z,w)-r|\qquad\text{for}\qquad z\in\partial \D_{|w|}.
\end{equation}
Now  let  $0<s<1$ be  small  enough, and  fix $w\in \D_s,$ and fix $0\leq r <|w|/2.$
 Using \eqref{e:z_expansion}--\eqref{e:z_expansion_bis} and applying Rouch\'e's theorem to the functions  $z\mapsto z-r\theta \big(Z(z,w), W(z,w)\big)$  and $z\mapsto z$ on $\D_{|w|},$ we see easily that  the  function $z\mapsto z_\infty(z,w)-r$  admits a unique solution on $\D_{|w|}.$
Next, using  \eqref{e:z_N_vs_z_infty} we apply Rouch\'e's theorem to the functions $z\mapsto z_N(z,w)-r$ and  $z\mapsto z_\infty(z,w)-r$  on $\D_{|w|}.$ Consequently,
assertion  (i) follows.

In the remainder of the proof, we write $M$ instead of $M_N$ for the sake of simplicity.
To prove  assertion (ii) we take for granted  the  following 
 \\
 \noindent{\bf Fact. }{\it  When $s>0$ is small enough and $(z,w) \in (\D_s)^2$ with $z_N(z,w)=r$ does not satisfies  property  (ii-a), we have that $|w|\geq 8MN |\sum_{k=2}^\infty b_k z^k|.$}

  Using  \eqref{e:W_expansion} and then the above fact, we  see that
 \begin{equation}\label{e:a_NW^N-a_Nw^N}
 \begin{split}
 |\vartheta^{-N}(z,w)W^N-w^N|&\leq\sum_{p=1}^N {N\choose p} |w|^{N-p}  |\sum_{k=2}^\infty b_k z^k|^p\\
 &\leq \big( \sum_{p=1}^N {N\choose p}   (8MN)^{-p}\big ) |w|^{N}\\
  &\leq   \big((1+8^{-1}M^{-1}N^{-1})^N-1\big)|w^N|\\
  &\leq (e^{8^{-1}K^{-1}}-1) |w^N|\leq   6^{-1}M^{-1}|w^N|.
  \end{split}
 \end{equation}
 Moreover, for $s=r_{N,M}>0$ small  enough, we infer from \eqref{e:W_expansion}--\eqref{e:W_expansion_bis} and the  continuity of $\vartheta$   that for $(z,w) \in \D_s^2,$
\begin{equation}\label{e:vartheta_WN}
  |\vartheta^{-N}(z,w)W^N-W^N| < 12^{-1} M^{-1}|W^N| <   4^{-1} M^{-1}|w^N|,
\end{equation} 
where the last  estimate  follows from \eqref{e:a_NW^N-a_Nw^N}.
 On the other hand, using the second inequality in \eqref{e:control_a_j_and_b_j}, \eqref{e:W_expansion},
 \eqref{e:W_expansion_bis} and then the above fact, we  see that 
 \begin{eqnarray*}
 \big|\sum_{j=N+1}^\infty a_j W^j\big|&\leq &\sum_{j=N+1}^\infty 4^j(|w|+ |\sum_{k=2}^\infty b_k z^k|)^j\\ 
 &\leq&\sum_{j=N+1}^\infty 4^j(1+ 8^{-1}M^{-1}N^{-1})^j|w|^j\\
 &\leq&  2^{-N-1}M^{-1}|w|^N\\
&\leq & 2^{-1}|a_Nw^N|, 
 \end{eqnarray*}
 where  the third inequality holds when $s=r_{N,M}>0$ is small enough, and the last one  follows
 from the  first inequality in \eqref{e:control_a_j_and_b_j}.
 This, combined  with \eqref{e:a_NW^N-a_Nw^N} and \eqref{e:vartheta_WN}, yields that
 \begin{eqnarray*}
 && |(z_\infty(z,w)+ \sum_{j=N}^\infty a_j W^j) -(z_\infty(z,w)+a_Nw^N)|\\
 &\leq&  
|a_N \vartheta(z,w)^{-N}W^N-a_Nw^N|+|a_N \vartheta(z,w)^{-N}W^N-a_N W^N|
+|\sum_{j=N+1}^\infty a_j W^j|\\
& \leq& M^{-1}|a_Nw^N|.
 \end{eqnarray*}
 Since  we know  by \eqref{e:z_expansion} and  \eqref{e:z_N_expansion} that the left hand side of the last line  is equal to $
 |z_N(z,w)  -(z_\infty(z,w)+a_Nw^N)|,$ assertion (ii-b) and hence  the  proposition  follow modulo the above fact.
 
 Now we turn to the proof of this fact. Suppose in order to reach a contradiction that
 \begin{equation}\label{e:contra}
|w|\leq 8MN |\sum_{k=2}^\infty b_k z^k|.
\end{equation}
This, coupled with \eqref{eq_z_N_full_expansion} and \eqref{e:W_expansion}, 
 implies that
\begin{eqnarray*}
|z_\infty(z,w)-z_N(z,w)|&\leq &\sum_{j=N}^\infty |a_j||\vartheta^j(z,w)|( |w|+|\sum_{k=2}^\infty b_k z^k|   )^j\\
&\leq &\sum_{j=N}^\infty 4^j (1+8MN)^j \big|\sum_{k=2}^\infty b_k z^k\big |^j   \\
&\leq & \sum_{j=N}^\infty 4^j (1+8MN)^j c^j|z|^{2j} \big(\sum_{k=2}^\infty \rho^k |z|^{k-2}\big )^j,
\end{eqnarray*}
 where the second  inequality holds by the second inequality in \eqref{e:control_a_j_and_b_j}, \eqref{e:W_expansion_bis} and
\eqref{e:contra}, the last one  by the third inequality in \eqref{e:control_a_j_and_b_j}.
 Hence, we infer that for $0<s<1$ small enough, 
\begin{equation}\label{e:z_vs_z_N}
|z_\infty(z,w)-z_N(z,w)|\ll |z|^2.
\end{equation}
 Suppose now that the point $(z,w)\in (\D_s)^2$  satisfies the assumption  of  assertion (ii). 
We infer from   \eqref{e:z_vs_z_N}  that
$r=|z_N(z,w)|\geq  |z_\infty(z,w)|-|z|^2.$
Since $s$ is  small  enough, we infer from  \eqref{e:z_expansion}--\eqref{e:z_expansion_bis} that $z_\infty(z,w)/z$ is close to $1.$ So $|z|\leq  2r.$ Hence,
 \eqref{e:z_vs_z_N} implies that  $|z_\infty(z,w)-r|\leq  4r^2.$
 Moreover, \eqref{e:contra}, combined with $|z|\leq  2r,$ implies that $|w|\lesssim |z|^2\ll r.$
 Hence,  we  obtain property (ii-a) which is the desired contradiction.
The proof of assertion (ii) is thereby completed.
\end{proof}

In what follows  by shrinking $\D^2$ if necessary, we may  assume without loss of generality that
the vector field $F$  at the beginning of the section is  defined on the 
bidisc $(e\D)\times(e^{|\lambda|}\D).$  

 \begin{definition}\label{D:points_compatible}\rm  
 Two points  $x_1=(z_1,w_1)$ and  $x_2=(z_2,w_2)\in(\D\setminus\{0\})^2$ are said to be  {\it quasi-compatible} if
 there is $t\in \C$   such that
 $z_2=z_1e^t$ and  $w_2=w_1e^{\lambda t}.$ Clearly, $x_1$ and $x_2$ are on the  same  leaf.
 If, moreover, we can choose $t$ with $|t|<1,$ then  we say that  $x_1$ and $x_2$ are   {\it compatible}.
 
 Given two quasi-compatible points  $x_1=(z_1,w_1)$ and  $x_2=(z_2,w_2)\in\D^2,$
 the {\it compatible pseudo-distance} between them, denoted by $\dist_C(x_1,x_2),$ is  defined by
 $$
\dist_C(x_1,x_2):= \max\left\lbrace { |z_1-z_2|\over |z_1|},{ |z_1-z_2|\over |z_2|},{ |w_1-w_2|\over |w_1|}, { |w_1-w_2|\over |w_2|}\right\rbrace.
 $$
\end{definition}

 \begin{lemma} \label{L:property_compatible} Let $x,x'\in(\D\setminus\{0\})^2$ be two compatible points.
 Let  $t\in\C$  such that  $z_2=z_1e^t$ and  $w_2=w_1e^{\lambda t}$ with $|t|$   smallest  possible.
 Then
 \begin{enumerate}
 \item[(i)]  $|z|\approx |z'|,$ $|w|\approx |w'|,$ $\|x\|\approx \|x'\|,$ and 
 $$\dist_C(x_1,x_2)\approx{|z-z'|\over |z|}\approx {|w-w'|\over |w|} \approx  {\|x-x'\|\over \|x\|}\approx  |t|;$$ 
 \item[(ii)] there is a constant $c>1$ such that
 \begin{equation*}
  c^{-1}{\|x-x'\|\over -\|x\| \log^* \|x\|}\leq  \dist_P(x,x')\leq   c{\|x-x'\|\over -\|x\|\log^* \|x\|}.
 \end{equation*}
 \end{enumerate}
 \end{lemma}
 \begin{proof}
 Assertion (i) is an immediate consequence of Definition \ref{D:points_compatible}.
 
 To prove   assertion (ii),  let $\omega\in\Omega$ be a path  such that 
   there is a differentiable  function $s\ni[0,1]\mapsto \zeta(s)\in \D$ satisfying 
   $$\omega(s):= (z e^{\zeta(s) },we^{\lambda\zeta(s) })\in (e^1\D)\times(e^{|\lambda|}\D)  \quad\text{ for} \quad s\in[0,1].
  $$
   and $\zeta(0)=0$ and  $\zeta(1) = t.$
Hence,  $\omega(0)=x$ and $\omega(1)=x'.$
   By   Lemma \ref{L:property_compatible} (i), we get $|\log^* \| \omega(s)\||\approx  |\log^*\|x\||$ for $s\in[0,1].$ 
Therefore,  by integrating  along  the  path $[0,1]\ni s\mapsto \zeta(s)$ 
and  
applying Part 2) of Lemma  \ref{lem_poincare}, we get that
\begin{eqnarray*}
\dist_P(\omega:0,1)&=&\int_{\omega[0,1]} \sqrt{g_P(z)}\\
& =& \int_0^1\zeta^*(\psi^*_x  (\sqrt{g_P}))  
 = \int_{\zeta[0,1]}  (\log^*\|x\|)^{-1} ds\\
&\gtrsim& {|t|\over -\log^* \|x| }\approx {\|x-x'\|\over -\|x\|\log^* \|x\|}. 
\end{eqnarray*}
 When $\zeta(s):=st$ for $s\in[0,1],$   $\gtrsim$ above becomes $\approx.$  
This    implies  assertion (ii).   
\end{proof}

 In  the remainder of this section we  consider the function
 \begin{equation}\label{e:choice_K_N}  M_N:=8^N,\qquad  N\in\N. 
\end{equation}
  By the  first   inequality in \eqref{e:control_a_j_and_b_j},  this choice   ensures that
  $M_N^{-1}\ll |a_N| .$
Moreover,   we  take $N$ so large  in the
sequence $(N_j)_{j=1}^\infty$ given in  \eqref{e:choice_N} that
   $N$ and the constant $M=M_N$  satisfy the conclusion 
of Proposition  \ref{P:z_N_expression}.  
 Let  $0<r_N:=r_{N,M}<1$ be given by this proposition.

\begin{lemma}\label{L:near_by_z=w^N}
Let $N\in\N $ be  as above, let $d:=-a_N,$  where $a_N$ is  introduced in
\eqref{e:z_N_expansion}
 and $\alpha\in\T.$ 
 Let $\xi_{N,\alpha,k}$ $(k\in\N)$ be
   the intersection  of $\Cf^d_N=\{z=dw^N\}$ with the  Riemann  surface
 $\widehat{L}_\alpha$ described  by \eqref{e:para-inter_z=w^N}.
   Then the intersection of the curve $\frak{C}_{0,N}$ with  
 $\widehat{L}_\alpha$ can be enumerated as $\xi_{0,N,\alpha,k}$  $(k\in\N)$
such that $\xi_{N,\alpha,k}$ and $\xi_{0,N,\alpha,k}$ are compatible  and that
\begin{equation*}
  \dist_C(\xi_{N,\alpha,k},\xi_{0,N,\alpha,k})\leq c N^{-1}\quad \text{for}\quad k\in \N.
\end{equation*}
Here $c>1$ is a constant independent of $N,$ $\alpha$ and $k.$ 
\end{lemma}
\begin{proof} We need to prove that  for  every
point $\xi_1\in\Cf^d_N\cap\widehat{L}_\alpha$  (resp. $\xi_1\in
\Cf_{0,N}\cap\widehat{L}_\alpha$),  there is  exactly one point $\xi_2\in
\Cf_{0,N}\cap\widehat{L}_\alpha$ (resp. $\xi_2\in\Cf^d_N\cap\widehat{L}_\alpha$) such that $\xi_1$ and $\xi_2$ are  compatible and that
\begin{equation}\label{e:xi_vs_xi'}
  \dist_C(\xi_1 ,\xi_2)\lesssim N^{-1}.
\end{equation}
We will only  show that for  every
point $\xi_1\in\Cf^d_N\cap\widehat{L}_\alpha,$ there is  exactly one point $\xi_2\in
\Cf_{0,N}\cap\widehat{L}_\alpha$ satisfying \eqref{e:xi_vs_xi'}
since the other  assertion can be proved similarly. Let $s_0:=r_{N}.$   

Write $\xi_1=(z_1,w_1).$ So $z_1=-a_Nw_1^N.$
We need to  find $\xi_2=(z_2,w_2)\in\Cf_{0,N}$  which is  compatible with $\xi_1$
in the sens of Definition \ref{D:points_compatible}.
By \eqref{e:Cf'_r,N} and \eqref{e:Cf_r,N},
the  membership $\xi_2=(z_2,w_2)\in\Cf_{0,N}$ is equivalent to $z_N(z,w)=0.$
Therefore, applying  Proposition \ref{P:z_N_expression} (i) to $r=0,$ we may find a unique $z=f(w)$
such that $|z|\leq|w|$ and  that $z_N(z,w)=0.$ 
Clearly, $\D_s\ni w\mapsto f(w)$ is a holomorphic  function.
  Using the function $\theta$ given in  \eqref{e:z_expansion}, we introduce the following holomorphic function 
  $$  \theta_N(w):=\theta\big(Z(f(w),w),  W(f(w),w)\big), \qquad\text{for}\qquad  w\in\D_{s_0}.
$$
By \eqref{e:z_expansion_bis}, we get that
\begin{equation}\label{theta_N}
1/2<|\theta_N(w)|<2\quad\text{and}\quad \lim_{w\to 0} \theta_N(w)=\theta_N(0)=1,
\end{equation}
the limit being  uniform in $N.$
 By Proposition \ref{P:z_N_expression} (ii) with $r=0$  and  \eqref{e:z_expansion}, we may write
\begin{equation}\label{e:f_theta}
f(w)= \theta_N(w)(-a_Nw^N+g(w))\qquad\text{for}\qquad  w\in \D_{s_0},
\end{equation}
 where $g$ is a  holomorphic function on $\D_{s_0}$ which  satisfies $ |g(w)|\leq M^{-1}|a_N w^N|,$ $ w\in\D_{s_0}.$

In order to find $\xi_2=(z_2,w_2)\in\Cf_{0,N}$  which is  compatible with $\xi_1,$
 we write $z_2=e^tz_1,$ $w_2=e^{\lambda t}w_1$ for some $0<|t|\ll 1.$
 We deduce from this  and  from  $(z_2,w_2)\in\Cf_{0,N}$    that
 $f( e^{\lambda t}w_1)=e^tz_1.   $
 Since $z_1=-a_Nw_1^N,$ it follows that $t$ is  a root of  the  following holomorphic function
 on  the disc $\D_s,$   $s\in (0,s_0)$ being a number whose  
exact value will be determined later on:
 \begin{equation}\label{e:F(t)}
 F(t):=-a_Ne^{\lambda Nt}w_1^N+g(e^{\lambda t}w_1)+a_Ne^tw_1^N \theta^{-1}_N(e^{\lambda t} w_1),\qquad  t\in  \D_s.
 \end{equation}
 Consider  the holomorphic function 
\begin{equation*}
H(t):= -a_Ne^{\lambda Nt}w_1^N +a_Ne^tw_1^N,\qquad t\in  \D_s.
\end{equation*}
Observe that  $H(t)=0$ if and only if $t={2i\pi k\over \lambda N-1}$ for $k\in\Z.$
So we  choose   the  constant $s$ as  follows:
$$s=c' {\pi \over |\lambda N-1|}\qquad \mbox{for $ c'>0$ a constant independent of $N,r$}.$$
 Hence $H$ has the unique root $t=0$ on $\D_s.$ 
 On the  other hand, observe  that
\begin{equation*}
H(t)=a_Nw_1^N( (-\lambda N+1)t +O(t^2)),\qquad\text{where}\ O(t^2)\ \text{depends on}\ N.
\end{equation*}
Consequently,  when  the constant $ c' $ (being independent of $N,r$) is  small enough,
$$
|H(t)|\approx |a_Nw_1^N|\quad\text{and}\quad   |a_Nw_1^N|<|H(t)| \quad \text{for}\quad  t\in \partial \D_s.
$$
Using this, 
we can show that  for $N$ large enough and   $t\in\partial \D_s,$
\begin{eqnarray*}
|F(t)-H(t)|&\leq& |g(e^{\lambda t}w_1)|+|a_Ne^tw_1^N| | \theta^{-1}_N(e^{\lambda t} w_1)-1|\\
&\ll &|a_Nw_1^N|< |H(t)|,
\end{eqnarray*}
where the first inequality holds by  
   the  uniform limit (with respect to  $N$) in \eqref{theta_N} and 
the  estimate $|g(e^{\lambda t}w_1)|\leq M^{-1}|a_Ne^{\lambda Nt}w_1^N|.$

So $|G(t)-H(t)|<H(t)$ on $\partial \D_s,$ and hence 
  by Rouch\'e's theorem,   $G$   has  a unique  root on $\D_s.$ 
Consequently, there is a unique $t\in \D_s$ such that $F(t)=0,$ i.e., there is a unique  
$\xi_2=(e^tz_1,e^{\lambda t}w_1)\in\Cf_{0,N}$  with $|t|\leq  s.$ Since $s \approx  N^{-1},$
 \eqref{e:xi_vs_xi'} follows from Lemma \ref{L:property_compatible}.
\end{proof}
\begin{lemma}\label{L:near_C_0,N}
 Let $N\in\N $ be  as above 
 and $\alpha\in\T.$ 
Let $\xi_{0,N,\alpha,k}$ $(k\in\N)$ be
   the intersection points of $\frak{C}_{0,N}$ with the  Riemann surface
 $\widehat{L}_\alpha$ described  by  Lemma \ref{L:near_by_z=w^N}.
   Then there is a constant $c_N>1$   independent of $\alpha$  
   satisfying  the following properties for every $0<r<r_N:$
 \begin{enumerate}
\item[(i)] the intersection of the curve $\frak{C}_{r,N}$ with   the  Riemann surface
 $\widehat{L}_\alpha$  inside  $ (r_N\D)^2\setminus \B_{r^{1/N} |\log r|^{3/N}}$ can be enumerated as $\xi_{r,N,\alpha,k}$  
such that $\xi_{r,N,\alpha,k}$ and $\xi_{0,N,\alpha,k}$ are compatible,
where $ k\in \N$  such that $\xi_{0,N,\alpha,k}\in(r_N\D)^2\setminus \B_{r^{1/N} |\log r|^{3/N}}; $
\item[(ii)] for every $ k\in \N$  with $\xi_{0,N,\alpha,k}\in (r_N\D)^2\setminus\B_{r^{1/N} |\log r|^{3/N}}  , $
\begin{equation*}
\dist_C(\xi_{r,N,\alpha,k},\xi_{0,N,\alpha,k })\leq c_N |\log r|^{-3}.
\end{equation*}
\end{enumerate}
\end{lemma}
\begin{proof}
We need to prove that  for  every
point $$\xi_1\in\big(\Cf_{0,N}\cap\widehat{L}_\alpha\big)\cap  \big((r_N\D)^2\setminus \B_{r^{1/N} |\log r|^{3/N}}\big)$$
$$  \Big(\text{resp.} \qquad\xi_1\in
\big(\Cf_{r,N}\cap\widehat{L}_\alpha\big)\cap  \big((r_N\D)^2\setminus  \B_{r^{1/N} |\log r|^{3/N}}\big)\qquad\Big)  ,$$  
there is  exactly one point 
$$\xi_2\in
\big(\Cf_{r,N}\cap\widehat{L}_\alpha\big)   \cap  \big((r_N\D)^2\setminus    \B_{r^{1/N} |\log r|^{3/N}}\big) $$ 
$$\Big(\text{resp.}\qquad \xi_2\in\big(\Cf_{0,N}\cap\widehat{L}_\alpha\big)
  \cap  \big((r_N\D)^2\setminus  \B_{r^{1/N} |\log r|^{3/N}}\big)\qquad\Big)$$
 such that $\xi_1$ and $\xi_2$ are  compatible and that
\begin{equation}\label{e:xi_vs_xi'_bis}
  \dist_C(\xi_1 ,\xi_2)\lesssim |\log r|^{-3}.
\end{equation}
We will only  show that for  every
point 
$$\xi_1\in\big(\Cf_{0,N}\cap\widehat{L}_\alpha\big) \cap  \big((r_N\D)^2\setminus \B_{r^{1/N} |\log r|^{3/N}}\big),$$ there is  exactly one point 
$$\xi_2\in
\big(\Cf_{r,N}\cap\widehat{L}_\alpha\big) \cap  \big((r_N\D)^2  \setminus \B_{r^{1/N} |\log r|^{3/N}}\big)$$ satisfying \eqref{e:xi_vs_xi'_bis}
since the other  assertion can be proved similarly. Let $s_0:=r_{N}.$  

Let $f$ and $g$   be the holomorphic functions on $\D_{s_0}$ introduced in the proof of Lemma \ref{L:near_by_z=w^N}  (see \eqref{e:f_theta}).  Since $g$  satisfies $ |g(w)|\leq M^{-1}|a_N w^N|,$ $ w\in\D_{s_0},$ $f$ admits  the following Taylor expansion:
\begin{equation}\label{e:f_theta_bis}
f(w)= -\tilde{a}_Nw^N+h(w)\qquad\text{for}\qquad  w\in \D_{s_0},
\end{equation}
 where  $h(w)=O(w^{N+1}).$ By \eqref{e:choice_K_N}  we get that 
\begin{equation}\label{e:quotient}
|\tilde{a}_N/a_N -  1|< 2^{-N}\qquad\mbox{for $N$ large enough.}
\end{equation}
Write $\xi_1=(z_1,w_1).$ Since $\xi_1\in \Cf_{0,N},$ the previous lemma  implies that $z_1=f(w_1).$
Recall  from  Section \ref{S:test_curves} the  coordinates $(Z,W).$ Under the  coordinates $(Z,W),$
we infer from \eqref{e:Cf'_r,N}, \eqref{e:z_N_expansion}
 and \eqref{e:Cf_r,N} the following simple correspondence between  $\Cf_{0,N}$ and $ \Cf_{r,N}:$
 \begin{equation}\label{e:correspondence}
   (Z,W)\in  \Cf_{0,N} \qquad\Longleftrightarrow \qquad  (r+Z,W)\in \Cf_{r,N} .
 \end{equation}
 In order to exploit this  nice correspondence  under the coordinates $(z,w),$ we introduce
 the   holomorphic  function $R_r$  given by the following relation
 \begin{equation}\label{e:Z_r}
 Z(R_r(w)+f(w),w)- Z(f(w),w)=r,\qquad  w\in \D_{s_0}.
 \end{equation}
  Recall  from \eqref{e:Jac} that the Jacobian matrix of $(Z,W)$ over $(z,w)$ at $(0,0)$ is the identity matrix.
  Consequently, using the Taylor expansion of $Z(z,w)$ and substituting $f(w)$ (resp. $R_r(w)+f(w)$) for $z,$     we infer from \eqref {e:Z_r} 
 that
  \begin{equation}\label{e:R_r}
  R_r(w) + O(R^2_r(w)) +O (R_r(w)f(w))= r,\quad\text{for}\quad  w\in \D_{s_0}.
  \end{equation}
  
 We need to  find $\xi_2=(z_2,w_2)\in\Cf_{r,N}$  which is  compatible with $\xi_1.$
Write $z_2=e^tz_1,$ $w_2=e^{\lambda t}w_1$ for some $0<|t|\ll 1.$
 We deduce from this  and  from  $(z_2,w_2)\in\Cf_{r,N}$  and \eqref{e:Z_r}   that
 \begin{equation}\label{e:xi_vs_xi2}
R_r(e^{\lambda t}w_1)+f( e^{\lambda t}w_1)=e^tz_1.   
\end{equation}
In the sequel, $s\in (0,s_0)$ is a number whose  
exact value will be determined later on. Since $z_1=f(w_1),$ it follows from the last line and  \eqref{e:f_theta_bis} that $t$ is  a root of  the  following holomorphic function on $\D_s$    defined   by 
  \begin{equation}\label{e:F(t)_bis}
 F(t):=R_r(e^{\lambda t}w_1)  -\tilde{a}_Ne^{\lambda Nt}w_1^N+h(e^{\lambda t}w_1)+ \tilde{a}_Ne^tw_1^N -e^th(w_1),
\qquad   t\in \D_s.
 \end{equation}

On the other hand, since $(z_1,w_1)\in \Cf_{0,N},$ we get by Proposition \ref{P:z_N_expression} with $r=0$  that
$2|a_Nw^N_1|\geq |z_1|$. This together  with  the second inequality in \eqref{e:control_a_j_and_b_j} 
imply $ |z_1|\leq 2^{N+1} |w_1|.$
Since  $(z_1,w_1)\not\in  \B_{r^{1/N} |\log r|^{3/N}},$
it follows that  $|w_1|\geq  2^{-N-1} r^{1/N} |\log r|^{3/N}.$
This together  with  the first inequality in \eqref{e:control_a_j_and_b_j} yield that
$$|a_Nw_1^N| \geq 2^{-N} |w_1|^N\geq  2^{-N(N+1)}r|\log r|^3.$$
Hence, there is $c_N>1$  such that 
\begin{equation}\label{e:r(logr)-3}
r<  c_N|\log r|^{-3}|a_Nw_1^N|.
\end{equation}
Now  we choose $M$ large enough ($M$ depending on $N$), 
 and  $0<s< s_0 $  such that
\begin{equation}\label{e:s}
s:=c'  |\log r|^{-3} \qquad \mbox{for $ c'=c'_N>0$ a large constant independent of $r$}.
\end{equation}
Then  we deduce  from \eqref{e:r(logr)-3}, \eqref{e:s}
and \eqref{e:R_r}, \eqref{theta_N}, \eqref{e:f_theta} and   \eqref{e:quotient}  that for $r>0$ small enough,
\begin{equation}\label{e:R_r_bis}
R_r(w)=r+o(r) \qquad\text{and}\qquad  r\ll s|\tilde{a}_Nw_1^N|.
\end{equation}
Consider  the holomorphic function 
\begin{equation*}
H(t):= -\tilde{a}_Ne^{\lambda Nt}w_1^N +\tilde{a}_Ne^tw_1^N,\qquad t\in  \D_s.
\end{equation*}
Observe that  $H$ has the unique root $t=0$ on $\D_s.$ 
 Moreover, when the constant $c'$  is large  enough, we  have  that
\begin{equation}\label{e:esti_G_zeros}
|H(t)|\approx s|\tilde{a}_Nw_1^N| \quad \text{and}\quad  |H(t)| > s|\tilde{a}_Nw_1^N|,\qquad\text{for}\qquad t\in\partial \D_s.
\end{equation}
We also  infer from \eqref{e:f_theta_bis} that
\begin{equation*}
 h(e^{\lambda t}w) -e^th(w)= O(tw^{N+1})     \qquad\text{for}\qquad  w\in \D_{s_0},
\end{equation*}
 where  $O(\cdot)$ depends on $N.$ 

Putting this together   with  the definition of $F$ and $H$ and  \eqref{e:esti_G_zeros} and  \eqref{e:R_r_bis},      a straightforward computation shows that 
   for $t\in\partial \D_s,$
\begin{eqnarray*}
|F(t)-H(t)|&\leq&
|R_r(e^{\lambda t}w_1)|  + | h(e^{\lambda t}w_1)) -e^th(w_1)|\\
&\leq & s|\tilde a_Nw_1^N|
<|H(t)|.
\end{eqnarray*} 
Using  this, we can   apply Rouch\'e's theorem to $F$ and $H.$ Consequently,  $F$  has  a unique  root on $\D_s.$
Therefore, there is a unique $t\in \D_s$ such that $F(t)=0,$ i.e., there is  a unique  
$\xi_2=(e^tz_1,e^{\lambda t}w_1)\in\Cf_{r,N}$  with  $|t|\leq s.$  Since $ s \approx  |\log r|^{-3},$
 \eqref{e:xi_vs_xi'_bis} follows from Lemma \ref{L:property_compatible}.
\end{proof}

In order to prove  the last part of  Proposition  \ref{P:mass_1}, the following lemma gives us the  discrepancy  between  the intersection points  of  a leaf with the algebraic  
curve $\frak{C}_{r,N}$  and with the analytic  curve $\mathfrak{C}_r$
inside the ball  $\B_{r^{1/N}|\log r|^{-3/N}}.$
\begin{lemma}\label{L:discrepancy_C_r,N_and_C_r}
Let $N\in\N$ be  as above   and $\alpha\in\T.$   
 Let $\xi_{r,\alpha,k} $ $(k\in\N)$ be
   the intersection  of  the  analytic curve $\frak{C}_r$ with the Riemann surface
 $\widehat{L}_\alpha$ described  by  Lemma \ref {L:intersection_C_r}.  Then there is a constant $c_N>1$  large  enough independent of $\alpha$  
 satisfying  the following properties for every $0<r<r_N:$
\begin{enumerate}
\item[(i)] the intersection of the curve $\frak{C}_{r,N}$ with the  Riemann surface
 $\widehat{L}_\alpha$ inside the ball $\B_{r^{1/N}|\log r|^{-3/N}}$ can be enumerated as $\xi_{r,N,\alpha,k} $
such that   $\xi_{r,N,\alpha,k}$ and $\xi_{r,\alpha,k}$ are compatible, where 
 $  k\in\N$ such that $ \xi_{r,\alpha,k}\in \B_{r^{1/N}|\log r|^{-3/N}} ;$
 \item[(ii)]  for every $k\in \N$  with $ \xi_{r,\alpha,k}\in \B_{r^{1/N}|\log r|^{-3/N} },$ we  have that
\begin{equation}\label{e:nearby_bis}
\dist_C(\xi_{r,N,\alpha,k},\xi_{r,\alpha,k})\leq c_N   |\log r|^{-3}.
\end{equation}
\end{enumerate}
 \end{lemma}
 \begin{proof}
 We need to prove that  for  every
point $\xi_1\in(\Cf_r\cap\widehat{L}_\alpha) \cap \B_{r^{1/N} |\log r|^{-3/N}}$  (resp. $\xi_1\in
(\Cf_{r,N}\cap\widehat{L}_\alpha)  \cap \B_{r^{1/N} |\log r|^{-3/N}}  $),  there is  exactly one point $\xi_2\in
(\Cf_{r,N}\cap\widehat{L}_\alpha) \cap \B_{r^{1/N} |\log r|^{-3/N}} $ 
(resp. $\xi_2\in(\Cf_r\cap\widehat{L}_\alpha) \cap \B_{r^{1/N} |\log r|^{-3/N}}$) such that $\xi_1$ and $\xi_2$ are  compatible and that
\begin{equation}\label{e:xi_vs_xi'_bis_bis}
  \dist_C(\xi_1 ,\xi_2)\lesssim |\log r|^{-3}.
\end{equation}
We will only  show that for  every
point $\xi_1\in(\Cf_r\cap\widehat{L}_\alpha) \cap \B_{r^{1/N} |\log r|^{-3/N}},$ there is  exactly one point $\xi_2\in
(\Cf_{r,N}\cap\widehat{L}_\alpha) \cap\B_{r^{1/N} |\log r|^{-3/N}}$ satisfying \eqref{e:xi_vs_xi'_bis_bis}
since the other  assertion can be proved similarly.

 Write $\xi_1=(z_1,w_1).$ So $z_1=r.$ We need to  find $\xi_2=(z_2,w_2)\in\Cf_{r,N}$  which is  compatible with $\xi_1.$
 Let $s_0:=r_{N}.$ 
Consider two cases.
 
 
 \noindent {\bf Case 1:}  $|w_1|\geq 2r.$
 
 In this case  we  fix  a  number $s\in (0,s_0)$ as follows
$$s:=c'  |\log r|^{-3} \qquad \mbox{for $ c'>0$ a large constant independent of $r$}.$$
Let $f,$ $h$ and $R_r$ be the holomorphic functions on $\D_{s_0}$ introduced in the proof of Lemma \ref{L:near_by_z=w^N} and Lemma \ref{L:near_C_0,N}   (see \eqref{e:f_theta_bis} and  \eqref{e:Z_r}).
 On  the  other hand,  we   deduce  from  the  membership $(z_2,w_2)\in \Cf_{r,N}$
and  \eqref{e:correspondence}
    and \eqref{e:Z_r} that  $z_2=R_r(w_2)+f(w_2).$ 
 Write $z_2=e^tz_1,$ $w_2=e^{\lambda t}w_1$ for some $0<|t|\ll 1$ since $\xi_2 $   is  compatible with $\xi_1.$  Consequently, we infer that $t$ is a  solution of the  following equation
  $$R_r(  e^{\lambda t}w_1   )+f( e^{\lambda t}w_1)=e^tz_1.   $$
 Using \eqref{e:f_theta_bis} and  the  equality   $z_1=r,$ we  deduce  that $t$ is  a root of  the  following holomorphic function
 on  the disc $\D_{s}$  
 $$
 F(t):=R_r(  e^{\lambda t}w_1   )-\tilde{a}_Ne^{\lambda Nt}w_1^N+h(e^{\lambda t}w_1) -e^t r,\qquad  t\in  \D_{s}.
 $$
 Consider another  holomorphic  function on $\D_s:$
 $$
 H(t):= r-re^t,\qquad t\in\D_s.
 $$
 Observe  that $H$ admits a unique  root $t=0$ on  $\D_s.$
 Moreover,  $|H(t)|\approx  sr$ for $t\in\partial\D_s.$  
 
 Since  $(z_1,w_1)\in  \B_{r^{1/N} |\log r|^{-3/N}},$
it follows that  $|w_1|\leq   r^{1/N} |\log r|^{-3/N}.$
This together  with  the second inequality in \eqref{e:control_a_j_and_b_j} yield that
$$|a_Nw_1^N| \leq 2^N |w_1|^N\leq  2^N r|\log r|^{-3}.$$
This, combined  with \eqref{e:f_theta} and \eqref{e:f_theta_bis}, implies that when  the constant $c'$ is  large  enough,    
\begin{equation}\label{e:r(logr)-3_bis}
|\tilde{a}_Nw_1^N|<  sr/2.
\end{equation}
Then  we deduce  from \eqref{e:r(logr)-3_bis}
and \eqref{e:R_r}, \eqref{theta_N}, \eqref{e:f_theta} that for $r>0$ small enough,
\begin{equation}\label{e:R_r_bis_bis}
R_r(w)=r+O(r^2), \qquad w\in\D_{s_0}.
\end{equation}
 On the other hand, we infer from   \eqref{e:f_theta_bis}   that 
 $|h(e^{\lambda t}w_1)|\ll |\tilde{a}_Ne^{\lambda Nt}w_1^N|.$
 Putting  this together with  \eqref{e:r(logr)-3_bis} and \eqref{e:R_r_bis_bis}, we  deduce for $t\in\partial \D_s$ that
 \begin{eqnarray*}
 |F(t)-H(t)|&\leq& |R_r(  e^{\lambda t}w_1   )-r|+ |h(e^{\lambda t}w_1)|\\
 &\leq &O(r^2)+|\tilde{a}_Nw_1^N|<sr \approx |H(t)|.
 \end{eqnarray*}
 Consequently, by Rouch\'e's theorem applied to $F$ and $H,$  there is a unique $t\in \D_s$ such that $F(t)=0,$ i.e., there is  a unique  
$\xi_2=(e^tz_1,e^{\lambda t}w_1)\in\Cf_{r,N}$  with  $|t|\leq s.$  Since $ s \lesssim  |\log r|^{-3},$
 \eqref{e:xi_vs_xi'_bis_bis} follows from Lemma \ref{L:property_compatible}.
 
  \noindent {\bf Case 2:}  $|w_1|\leq 2r.$
  
  Here the  difficulty lies in the fact that  we cannot apply  Proposition   \ref{P:z_N_expression} (i)
 and  that   the functions $f,g,h$ etc  are therefore not available  any more. 
In this  case  we choose  $s:=c'r^2$  for  $c'>0$ a large  constant.
Using the  assumption $|w_1|\leq 2r$ and   the expansion  \eqref{eq_z_N_full_expansion}, we get
  that
\begin{equation*}
|z_\infty(r,w_1)-z_N(r,w_1)|\leq \sum_{j=N}^\infty |a_j||\vartheta^j(r,w_1)|\big( |w_1|+|\sum_{k=2}^\infty b_k r^k|   \big)^j=O(r^2).
\end{equation*}
Moreover, using \eqref{e:z_expansion} and \eqref{e:Jac} we infer that
\begin{equation*}
|r-z_\infty(r,w_1)|=O(r^2).
\end{equation*}
On the  other hand,   using  \eqref{eq_z_N_full_expansion} and  \eqref{e:z_expansion} and \eqref{e:Jac}, we  
obtain that
\begin{eqnarray*}
  z_N(e^t r,e^{\lambda t}w_1)-z_N(r,w_1)&=&  z_\infty(e^t r,e^{\lambda t}w_1)-z_\infty(r,w_1) +O(r^2)+O(t^2)\\
&=& e^tr-r + O(r^2)+O(t^2)= O(r^2)+O(t^2).
\end{eqnarray*}
Putting together these three estimates, we have  
for $t\in \D_s$ that
 \begin{eqnarray*}
 |z_N(e^t r,e^{\lambda t}w_1)-r|&\leq& \big | z_N(e^t r,e^{\lambda t}w_1)-z_N(r,w_1)\big| +        |z_N(r,w_1)- z_\infty(r,w_1)|\\
&+& |r-z_\infty(r,w_1)|\\
&=& t+O(t^2)+O(r^2) .
 \end{eqnarray*}
Consequently, when $c'$ is  large enough, we know by Rouch\'e's theorem applied to  the identity function
$\D_s\ni t\mapsto t$
 and   the
holomorphic function  $ t\mapsto z_N(e^t r,e^{\lambda t}w_1)-r$ that the latter    admits a unique root on $\D_{s}.$
 Hence,   there is  a unique  
$\xi_2=(e^tz_1,e^{\lambda t}w_1)\in\Cf_{r,N}$  with  $|t|\leq s.$  Since $ s \lesssim  |\log r|^{-3},$
 \eqref{e:xi_vs_xi'_bis_bis} follows from Lemma \ref{L:property_compatible}.
 \end{proof}

\section{Mass of   $T\wedge[\mathfrak{C}_r]$ on balls}
\label{S:balls}

 
  The main purpose of this section is to prove  Proposition \ref{P:comparison}   and one half  of Proposition \ref{P:key_interpretations}.
  Recall that $\T:= \{\alpha\in\C:\ e^{-2\pi b}\leq |\alpha|\leq 1\}.$
In parallel with the integral operator $K_s$ given in  \eqref{eq_K_s}, we also consider, for each $s>0,$ the  domain  
$   D_s:=\{ t\in\R:\  t\geq  s \},$
 and the  function $\widetilde{K}_s:\   \R\to\R^+$ given  by
\begin{equation}\label{e:widetilde_K_s}
\widetilde{K}_s(y):=                   \int_{ D_s }{V\over V^2+( y-U)^2}  dt,\qquad y\in \R.
\end{equation}
Here $U,$ $V$  are functions of the variable $t$ and the parameter $s$ which satisfy the following  system
of equations (see \eqref{eq_u,v_vs_U,V},   \eqref{e:t} and \eqref{eq_u,v_vs_z,w}):
  $$U+iV=(u+is)^\gamma\qquad\text{and}\qquad  t=bu+as.$$  

The following result is the main technical point  in the proof of Proposition \ref{P:comparison}.
\begin{lemma}\label{L:comparison} There is a constant $c>1$ such that 
for all $y\in \R$  and $s>0,$  
\begin{equation*}
        c^{-1}\leq  \widetilde{K}_s(y)/  K_s(y)\leq c . 
\end{equation*}
\end{lemma}
Taking  Lemma \ref{L:comparison} for granted, we arrive at the

\noindent{\bf  End of the proof of Proposition \ref{P:comparison}.} In what  follows we use the notation
introduced in  Lemma \ref{L:intersection_C_r}.
By \eqref{e:para-inter_C_r},  we have that $u_{\alpha,k}=  2k\pi- (\log {|\alpha|})/b $ and  $v_r=-\log r.$ This, coupled
with  
\eqref{eq_parametrization_DS}, \eqref{e:t} and \eqref{eq_u,v_vs_z,w}, implies that  $\xi_{r,\alpha,k}\in \B_{r^\delta}$
if and only if   $bu_{\alpha,k}+av_r\geq -\delta\log r,$ which is, in turn, equivalent to 
$$k\geq{1\over 2\pi b}\Big( (\delta-a)(-\log r) + \log {|\alpha|}\Big).$$ 
Let $\Z_{\delta,r, \alpha}$  be the set of all  integers $k$ satisfying  the last inequality.
Observe that
$$
 \| T\wedge [\Cf_r]\|_{    \B_{r^\delta}}=\int_{\T}\Big(\sum_{k\in \Z_{\delta, r, \alpha}}h_\alpha(\xi_{r,\alpha,k})\Big) d\nu(\alpha).
$$  
 For $\alpha\in\T$ let $D^\alpha:= D_{-\delta \log r},$
  $t_{r,\alpha,k}:=bu_{\alpha,k}+av_r,$ $k\in \Z.$ Then  $k\in \Z_{\delta, r,\alpha}$ if and only if
$t_{r,\alpha,k}\in D^\alpha.$  Moreover, consider the function $\chi^\alpha:\ D^\alpha\to\C$ given by
$$
\chi^\alpha (t):=u+iv_r,\qquad\text{where}\quad v_r=-\log r\quad\text{and}\quad t=bu+av_r.
$$ 
Let $m=1$ and choose $\rho>1$  large  enough.
Using  Harnack's inequality, we see that the assumption of Part 1) of  Proposition \ref{P:interpretation} 
is  fulfilled. Hence, by Definition \ref{D:interpretation} we obtain an interpretation $(K^\alpha)_{\alpha\in\T}$
of the geometric intersection $T\wedge [\Cf_r]$ on $    \B_{r^\delta}$ with mesh $m=1.$
Moreover, we infer from the above  discussion and  formula \eqref{e:widetilde_K_s} that
$K^\alpha(y)=\widetilde K_{-\delta \log r}(y).$
Consequently, by Part 1) of Proposition \ref{P:interpretation}   we get that  
\begin{equation*} 
 \| T\wedge [\Cf_r]\|_{    \B_{r^\delta}}\approx 
\int_{\alpha\in\T}\Big( \int_{-\infty}^\infty  \widetilde{K}_{ -\delta \log r}(y)\tilde{H}_\alpha(y)dy
\Big) d\nu(\alpha). 
\end{equation*}
On the  other hand, by Lemma \ref{L:comparison} we know that
  $ \widetilde{K}_{ -\delta \log r}(y)\approx  K_{ -\delta \log r}(y).$
By  the definition   of $K_s$ in \eqref{eq_K_s}, we may find  a constant $c_\delta>1$ such that 
 $c_\delta^{-1}<K_{ - \delta \log r}(y)/  K_{ -\delta \log r}(y)<c_\delta.$
 So   $\widetilde{K}_{ -\delta \log r}(y)\approx K_{ - \log r}(y).$
 This, combined with  the last   estimate  for  $\| T\wedge [\Cf_r]\|_{    \B_{r^\delta}},$ implies that
 $$
  \| T\wedge [\Cf_r]\|_{    \B_{r^\delta}}\approx 
\int_{ \alpha\in\T}\Big( \int_{-\infty}^\infty  K_{ - \log r}(y)\tilde{H}_\alpha(y)dy
\Big) d\nu(\alpha).
 $$
 Comparing  this  with   Lemma \ref{lem_estimate_G}, the proof is  completed. 
\hfill $\square$

\noindent{\bf  End of the proof of Lemma \ref{L:comparison}.}  Let  $c_2,c_3$ be the constants with $c_3>c_2>1$ given by Lemma \ref{lem_Poisson_kernel}.
 We consider three cases.

\noindent {\bf Case 1:} $  s\geq c_2   (1+ | y|)^{1/\gamma} .$ 

By  Part 2) of Lemma \ref{lem_Poisson_kernel} and by formula \eqref{e:widetilde_K_s},
we have that
$$
\widetilde{K}_s(y)\approx    \int_{t=s}^\infty  {sdt\over t^{\gamma+1}}    
\approx  s^{1-\gamma}.$$
This, compared with formula   \eqref{eq_K_s}, completes 
 the proof of   Case 1.

\noindent {\bf Case 2:} $  c_3^{-1} \leq  {s\over   (1+ | y|)^{1/\gamma}}\leq c_2.$ 

Write $D_s= D^1_s\cup D^2_s,$ where 
\begin{eqnarray*}
  D^1_s&:=&\left\lbrace t\in D_s:\ t\leq  c_2 (1+ | y|)^{1/\gamma}\right\rbrace,\\
  D^2_s&:=&\left\lbrace    t\in D_s :\  t\geq   c_2 (1+ | y|)^{1/\gamma}  \right\rbrace.
\end{eqnarray*}
Consequently, formula \eqref{e:widetilde_K_s}  gives that
\begin{equation}\label{eq_I_II}
\widetilde{K}_s(y)=    \Big(\int_{D^1_s}  +     \int_{D^2_s }\Big){V\over V^2+( y-U)^2} dt=: I+II.
\end{equation}
To estimate $(I),$  we  apply  Part 4) of Lemma \ref{lem_Poisson_kernel} and obtain that
$$
I \approx  \int_{D^1_s}{ dt\over (1+|y|)}= \int_{c_3^{-1}(1+ | y|)^{1/\gamma}}^{c_2(1+ | y|)^{1/\gamma}}  {dt\over (1+|y|)}  \approx (1+|y|)^{1/\gamma-1}.
$$
To estimate $(II),$  we  apply  Part 2) of Lemma \ref{lem_Poisson_kernel} and obtain that
$$
II\approx  \int_{D^2_s}{ \min\{ t,v\}dt\over  (  \max\{ t,v\} )^{\gamma+1} } 
\leq
 \int_{c_3^{-1}(1+ | y|)^{1/\gamma}}^\infty  {s dt\over  t^{\gamma+1}}   
 \approx (1+|y|)^{1/\gamma-1}.$$
 Inserting the above estimates for $(I)$ and $(II)$ into  (\ref{eq_I_II}), we obtain  that
  $\widetilde{K}_s(y)\approx (1+|y|)^{1/\gamma-1}$ in the  second case.  
 Comparing this  with  formula   \eqref{eq_K_s},
  the proof of   Case 2 is  complete.

\noindent {\bf Case 3:} $   s  \leq  c_3^{-1}     (1+ | y|)^{1/\gamma}.$ 

 Write  
 $D_s= D^{1}_s\cup D^{2}_s\cup D^{3}_s ,$ where
\begin{eqnarray*}
D^{1}_s&:=&\left\lbrace t:\ s\leq t\leq  c^{-1}_2 (1+ | y|)^{1/\gamma}\right\rbrace,\\
D^{2}_s&:=&\left\lbrace t:\  t\geq  c_2 (1+ | y|)^{1/\gamma}  \right\rbrace,\\
D^{3}_s&:=&\left\lbrace t:\   c^{-1}_2 (1+ | y|)^{1/\gamma} \leq   t \leq  c_2 (1+ | y|)^{1/\gamma}  \right\rbrace.
\end{eqnarray*}
 Consequently, we get,  similarly as in (\ref{eq_I_II}), that 
$$ \widetilde{K}_s(y)=   \Big( \int_{D^{1}_s}+   \int_{D^{2}_s}+\int_{D^{3}_s} \Big) {V\over V^2+( y-U)^2} dt=: I+II+III.
$$
 To estimate $(I)$   we apply  Part 1) and Part 3) of Lemma \ref{lem_Poisson_kernel}. Consequently,
we  obtain that
$$
I \approx \int_{D^{1}_s}   {t^{\gamma-1}s    dt\over (1+|y|)^2}
\approx 
s\Big( \int_{t=s}^{c^{-1}_2(1+ | y|)^{1/\gamma}}  {t^{\gamma-1}dt\over (1+|y|)^2}    \Big) \lesssim  s(1+|y|)^{-1 }.$$
To estimate $(II),$   we  apply  Part 2) of Lemma \ref{lem_Poisson_kernel} and obtain that
$$
II\approx \int_{D^{2}_s}  { sdt\over  t^{\gamma+1} } \approx 
s\Big( \int_{c_2(1+ | y|)^{1/\gamma}}^\infty  {dt\over t^{\gamma+1}}    \Big)  \approx  s(1+|y|)^{ -1}.$$
To  estimate $(III),$   we  apply  Part 5) of Lemma \ref{lem_Poisson_kernel} and obtain that
$$
III\approx    \int_{c^{-1}_2(1+ | y|)^{1/\gamma}}^{c_2(1+ | y|)^{1/\gamma}}
  { (1+|y|)^{1/\gamma-1} sdt \over s^2+  (t-\rho(y,s) )^2 }, 
$$
where $\rho(y,s)$ satisfies  $c^{-1}_2(1+ | y|)^{1/\gamma} \leq \rho(y,s)\leq c_2(1+ | y|)^{1/\gamma}.$ 
 Write $III:=III_1+III_2,$ where
$$ III_1=\int_{ |t- \rho(y,s)|\leq  s}
        { (1+|y|)^{1/\gamma-1}s dt \over s^2+  (t-\rho(y,s) )^2 }\approx
\int_{ |t- \rho(y,s)|\leq  s}
        { (1+|y|)^{1/\gamma-1} dt \over s  }\approx
(1+|y|)^{1/\gamma-1},
$$
 and
 $$
III_2\approx  \int { (1+|y|)^{1/\gamma-1} sdt \over s^2+  (t-\rho(y,s) )^2 }\leq 
 \int { (1+|y|)^{1/\gamma-1} sdt \over   (t-\rho(y,s) )^2 }\lesssim
c(1+|y|)^{1/\gamma-1},
$$
 the integrals in the last line  being taken over the region
$$\left\lbrace t\in\R:\ c^{-1}_2(1+ | y|)^{1/\gamma}\leq t\leq c_2(1+ | y|)^{1/\gamma} \ \text{and}\  |t- \rho(y,s)|\geq  s \right\rbrace.$$
   Thus, $III\approx (1+|y|)^{1/\gamma -1}.$

Combining the obtained  estimates for  $(I),$  $(II)$ and   $(III),$  and using the  assumption
$   s  \leq  c_3^{-1}     (1+ | y|)^{1/\gamma},$ we  infer that 
$$\widetilde{K}_s(y)=I+II+III\approx  s(1+|y|)^{ -1}+ (1+|y|)^{1/\gamma-1}\approx   (1+|y|)^{1/\gamma-1}.$$
 This, compared  with  formula   \eqref{eq_K_s},
   allows us to 
 conclude
 the proof of   the last case.  
  \hfill $\square$
  
  As an application of Lemma \ref{L:comparison},  we  are able to  establish one half of the proof of Proposition \ref{P:key_interpretations}.
  
\noindent  
   {\bf Proof of Proposition \ref{P:key_interpretations} for
the  geometric intersection $T\wedge [\Cf_r].$}
In parallel with the integral operator $\widetilde{K}_s$ given in  \eqref{e:widetilde_K_s}, we   consider, for each $s>0,$ the  domain  
$   D_{s,N}:=\{ t\in\R:\  t\geq  s/N+3(\log s)/N \},$
 and the  function $K_{s,N}:\   \R\to\R^+$ given  by
\begin{equation}
K_{s,N}(y):=                   \int_{ D_{s,N} }{V\over V^2+( y-U)^2}  dt,\qquad y\in \R,
\end{equation}
Here $U,$ $V$  are functions of the variable $t$ and the parameter $s$ which satisfy the following  system
of equations (see \eqref{eq_u,v_vs_U,V},   \eqref{e:t} and \eqref{eq_u,v_vs_z,w}):
  \begin{equation}\label{e:coherence_K}
U+iV=(u+is)^\gamma\qquad\text{and}\qquad  t=bu+as.
\end{equation}
We  argue as in the proof of  Lemma \ref{L:comparison} for $\delta:=1/N$ making the obviously necessary changes. The factor $\log s$ in the definition of $D_{s,N}$  can be overlooked without changing the final result.
 For $\alpha\in\T,$  set $D^\alpha:= D_{-\log r,N}$ and   $K^\alpha:=K_{-\log r,N}$ 
 and $\chi^\alpha(t):=u(t)-i\log r,   $ $t\in D^\alpha,$
 where $u$ is  a  function of $t$ satisfying equation \eqref{e:coherence_K} with $s:=-\log r.$
Consequently, using Lemma \ref{L:intersection_C_r} we can show that $(K^\alpha)_{\alpha\in\T}$
is   an  interpretation of the  geometric intersection $T\wedge [\Cf_r]$ on $\B_{r^{1/N}|\log r|^{-3/N}}$
with mesh $1.$

It remains us to show that the above  interpretation is  coherent. We can check this  using 
Lemma \ref{lem_Poisson_kernel}  and the fact that the mesh of the interpretation is  $1.$
\hfill $\square$  
 
\section{Mass of  $T\wedge [\frak{C}_{r,N}]$ outside the corona   $\A_{r,N}$}
\label{S:outside_corona}
   
  The objective of this  section is  to establish Proposition \ref{P:mass_1}. We start with the following simple lemma.
   
 \begin{lemma}\label{L:disc_r}
 Let $0<s<r<\infty.$  
 Let $h$ be  a  positive harmonic function   on the disc $\D_r.$
Then there is a constant $c>0$  depending only on the quotient $s/r$ such that for $x_1,x_2\in \overline\D_s,$ 
\begin{equation*}
|h(x_1)-h(x_2)|\leq  cr^{-3}|x_1-x_2|\int_{\D_r}  h(z)d\Leb (z),
\end{equation*}
where $d\Leb$ is the Lebesgue measure  in $\C.$ 
 \end{lemma}
\begin{proof} Using a continuity argument  we may  assume that $h$ is  continuous  on $\overline\D_r.$ By Poisson integral  formula, we have
\begin{equation*}
h(x)= {1\over 2\pi r}\int_{\partial  \D_r}    {r^2-|x|^2\over  |x-y|^2}h(y) d\sigma(y)
\quad\text{for}\quad  x\in \D_r,\ y\in \partial \D_r,
\end{equation*}
where $d\sigma(y)$ is  the Lebesgue measure  on $\partial  \D_r.$  
We  infer   from this  formula  that  
\begin{eqnarray*}
|h(x_1)-h(x_2)|\leq  c  {|x_1-x_2|\over r^2}\int_{\partial \D_r } h(y) d\sigma (y)&=& {2\pi c |x_1-x_2| h(0)\over  r}\\
&=& {2c|x_1-x_2| \over r^3}\int_{\D_r}  h(z)d\Leb (z),
\end{eqnarray*}
where the  equalities  hold by the  mean-property. Hence, the lemma follows.
 \end{proof}
 
  \smallskip
  
  Now we are in the position to prove the first part of  Proposition  \ref{P:mass_1}.
  
\noindent{\bf End of  the proof of estimate \eqref{e:mass_estimates_1} in Proposition  \ref{P:mass_1}.}
Let $N\in \N$ and  let $r_N$ be  given by Proposition   \ref{P:cohomology}.
 Consider the compact set
$$
Y:= \Cf_{r,N}\cap  (X\setminus \D^2).
$$ By  Proposition 
\ref{P:cohomology} (i) and (ii), we have that 
$Y\cap E=\varnothing.$
We  use the finite cover $\Uc$  of $X$  by singular  flow  boxes  $(\U_e)_{e\in E}$ and  regular 
flow boxes $(\U_p)_{p\in P}$  introduced in Subsection \ref{SS:local_model}.
We may assume  without loss of generality that 
\begin{equation}\label{e:Y}
 Y\subset \bigcup_{p\in P} \U_p .
 \end{equation}
  Putting \eqref{e:Cf'_r,N} and \eqref{e:Cf_r,N} and \eqref{e:Y} together,
  we use an argument of local  complex  geometry to  express    the intersection 
points of the algebraic curves  $\Cf_{r,N}$ and    $\Cf_{0,N}$
with a plaque of $\U_p$ as the roots of some holomorphic functions defined on some open subset of $\U_p.$
Consequently, by shrinking $r_N$ if necessary, 
 we may  find a constant $0<\delta=\delta_N<1$ 
 such that
for every $x\in \Cf_{0,N}\cap(X\setminus \D^2)$ and every $0< r<r_N,$  there is exactly  one point $\tau(x)\in\Cf_{r,N}$  
such that $x$ and $\tau(x)$ are on the same plaque $\V_x$  of at least  one of the regular flow boxes $(\U_p)_{p\in P}$
and that  
\begin{equation}\label{e:x_vs_tau_x}
\dist(x,\tau(x))\leq c r^\delta \quad \text{for some constant}\quad c>1\quad\text{independent of}\quad x.
\end{equation}
In fact, $\delta$ is the reciprocal of the multiplicity of the intersection of $\Cf_{0,N}$
and $\V_x$ at $x.$

By  shrinking  the union $\cup_{x\in \Cf_{0,N}} \V_x, $ we may find an open neighborhood
  $\V$  of  $(X\setminus \D^2)\cap \Cf,$ where $\Cf$ is  the closure  of $\bigcup_{0\leq r<r_N}\Cf_{r,N}.$
Therefore, by  Proposition \ref{prop_current_local}  we have the following integral  representation of $T$ in $\V:$
\begin{equation*}
T=\int h_x[\V_x] d\nu(x),
\end{equation*}
where, for each  $x\in \Cf_{0,N} \cap (X\setminus \D^2),$     $h_x$ denotes the positive harmonic function associated to the current $T$ on the plaque  $\V_x, $
and $\nu$ is a  positive Radon  measure on   $\Cf_{r,N} \cap (X\setminus \D^2).$ 

On the other hand,  by Part 1) of Lemma  \ref{lem_poincare}, we have that
\begin{equation*}
c^{-1}\leq \eta(x)\leq c,\quad x\in \bigcup_{p\in P} \U_p\quad\text{for some constant}\quad c>1. 
\end{equation*}
This, combined with \eqref{e:x_vs_tau_x} and  \eqref{e:Y}, implies    that
 $\dist_P(x,\tau(x))\leq  cr^\delta$ and that the diameter of the plaque $\V_x$ with respect to the Poincar\'e metric $g_P$ is  $\approx 1.$
Applying  Lemma \ref{L:disc_r}
  to the disc $\D_r$ for $r\approx 1,$
  we  get   a constant $c>1$    such that   
\begin{equation*}
|h_x(x)-h_x(\tau(x))|\lesssim  \dist_P(x,\tau(x)) \int_{\V_x}  h_x(y) g_P(y)\lesssim  r^\delta\int_{\V_x}  h_x(y)g_P(y).
\end{equation*}
Integrating both sides
with respect  to $ d\nu(x),$
we obtain  that
\begin{equation*}
\int_{x\in \Cf_{0,N} \cap (X\setminus \D^2)} |h_x(x)-h_x(\tau(x))|  d\nu(x) \lesssim    
r^\delta \| T\wedge g_P\|_{\V}\lesssim r^\delta ,
\end{equation*}
where the last inequality holds because  $\| T\wedge g_P\|_X$ is finite  by
 \cite[Proposition 4.2]{DinhNguyenSibony1}  (this corresponds to  \eqref{e:known_estimate} in the case $\delta=1$).
Since  we  know  by using \eqref{e:Y} and  Proposition  \ref{P:local_inter} that the  left hand side is bigger than
 $$\Big | \|  T\wedge [\Cf_{r,N}]\|_{X\setminus  \D^2} -  \|  T\wedge [\Cf_{0,N}]\|_{X\setminus  \D^2}\Big|,$$   the  desired estimate follows.
 \hfill
 $\square$

\begin{remark}\label{R:sD_red} Estimate \eqref{e:mass_estimates_1} in Proposition  \ref{P:mass_1}
 still holds if we  replace $X$ and  $\D^2$ by the bidiscs $\D^2$ and $(s\D)^2$ respectively for any $0<s<1,$ i.e.,
 there are  constants  $0<\delta=\delta_N<1$ and  $c=c_{s,N}>1$ such  that 
 $$\Big | \|  T\wedge [\Cf_{r,N}]\|_{\D^2\setminus  (s\D)^2} -  \|  T\wedge [\Cf_{0,N}]\|_{\D^2\setminus  (s\D)^2}\Big|\leq  c r^\delta\quad\text{for}\quad 0<r<\min\{s,r_N\}.$$ 
 \end{remark}

The   last two parts of Proposition   \ref{P:mass_1} concern    balls  in a singular  flow  box
around a  singular point $\bar x\in E$. In what follows,  we may  assume without loss of generality that
$\Fc$ is the foliation  on $\D^2$   associated to the vector field $F$  introduced in Section
\ref{S:parametrization} and that $\bar x=0\in\D^2.$  Moreover,
 let  $c_0$ be the constant $c>1$  given  by
 Lemma \ref{L:property_compatible} (ii).  
\begin{definition}\rm \label{D:cell}
Given  a point $x_0=(z_0,w_0)\in (\D\setminus\{0\})^2$ and
a number $0<\rho<1/2,$ 
  a {\it  cell} with center $x_0$ and radius $\rho$  is  the  set $\Cell (x_0,\rho)$
given by 
\begin{equation*}
\left\lbrace x=(z,w)\in(\D\setminus \{0\})^2:\  \max\{ |1-z/z_0|,|1-z_0/z|,|1-w/w_0|,|1-w_0/w|  \}<\rho \right\rbrace.
\end{equation*}
Note that for $ x\in\Cell(x_0,\rho),$ 
$$
(1+\rho)^{-1}\|x_0\|\leq \|x\|\leq (1+\rho)\|x_0\|
.
$$
We fix  a   number $0<\rho_0<1/2$ so   that the following  two conditions (i)--(ii) are satisfied:

(i) For  every $0<\rho\leq \rho_0$ and  every point $x_0 =(z_0,w_0)\in (\D\setminus\{0\})^2,$ 
  $\Cell(x_0,\rho)$ is a  flow box with the transversals
 $$\T_{x_0}:= \{ (z_0,w)\in \Cell(x_0,\rho):\ w\in \C\}\quad\text{and} \quad \T'_{x_0}:= \{ (z,w_0)\in \Cell(x_0,\rho):\ z\in\C\}.$$
  We often identify $\T_{x_0}$ and  $ \T'_{x_0}$ with its projection on second first and  its first 
 components respectively, that is, with the set $\{w\in\D: |1-w/w_0|<\rho\ \text{and}\ |1-w_0/w|<\rho\}$
 and $\{z\in\D: |1-z/z_0|<\rho\ \text{and}\ |1-z_0/z|<\rho\}$ respectively.
 The plaque  of $\V:=\Cell(x_0,\rho)$ passing through $\alpha=(z_0,w)\in \T_{x_0}$ is  denoted by $\V_\alpha.$
 
 (ii) All points in $\V_\alpha$ are compatible  with each other in the sense of Definition \ref{D:points_compatible}, in particular with $\alpha,$
for all $\alpha  \in \T_{x_0}.$

Finally, set $\rho_1:= 1/4 c_0^{-2}\rho_0.$ 
\end{definition}

The following result illustrates the usefulness of the constants $\rho_0$ and $\rho_1$ given in Definition   \ref{D:cell}.

\begin{proposition}\label{P:two_curves_in_cell}
Let $\U$  (resp. $\V$)  be the  cell with center $x_0\in  (\D\setminus\{0\})^2$ and radius $\rho_0$ (resp.
 radius  $\rho_1$). Let $\T_{x_0}$ be  a transversal of  $\V$  as in Definition   \ref{D:cell}.
 Let  $\Df_1,\Df_2$ be   two analytic curves in $\U$
 such that for  every $\alpha\in \T_{x_0}$ and $j=1,2,$  $\Df_j$ intersects the plaque $\V_\alpha$ at 
a unique point $\alpha_j.$
Then there is a constant $c>0$ independent of $x_0$  such that for  every
 harmonic current $T$ tangent to the foliation, we have  that
 $$
 \| [\Df_1]\wedge T-[\Df_2]\wedge T\|_\V \leq  c {(\log^*{|x_0|})^2\| T\wedge g_P\|_{\U} \over  \|x_0\| }   \sup_{\alpha\in \T_{x_0}}   \|\alpha_1- \alpha_2\|.
 $$
\end{proposition}
\begin{proof}
 By  Proposition \ref{prop_current_local} we have the following integral  representation of $T$ in $\V:$
\begin{equation*}
T=\int h_\alpha[\V_\alpha] d\nu(\alpha),
\end{equation*}
where, for each  $\alpha\in \T_{x_0},$     $h_\alpha$ denotes the positive harmonic function associated to the current $T$ on the plaque  $\V_\alpha. $  

Let $\alpha\in \T_{x_0}.$ Since we know  by Definition  \ref{D:cell} that $\|\alpha\|\approx \|x_0\|$
and that  the points in $\V_\alpha$ are compatible with each other,  it follows from Lemma \ref{L:property_compatible} (ii) and  the choice  of $\rho_0,$  $\rho_1$ that there are two constants $0<c'<c''$  (independent of    $\alpha$ and $x_0$) and  constants  $c_\alpha, c'_\alpha\in [c',c'']$    such that
\begin{equation*}
 \phi_\alpha(\D_{c_\alpha (\log^*\|x_0\|)^{-1}})\subset\V_{\alpha}\subset  \phi_{\alpha}(\D_{c'_\alpha (\log^*\|x_0\|)^{-1}})\subset   \phi_{\alpha}(\D_{2c'_\alpha (\log^*\|x_0\|)^{-1}})\subset \U_\alpha.  
\end{equation*}
Therefore, applying  Lemma \ref{L:disc_r}
  to the disc $\D_r$ for $r:=2c'_\alpha(\log^*\|x_0\|)^{-1},$
  we  get   a constant $c>1$    such that  for $x_1,x_2\in\V_\alpha,$  
 we have that
\begin{equation*}
|h_\alpha(x_1)-h_\alpha(x_2)|\leq  c\|x_0\|^{-1} (\log^*\|x_0\|)^{2}|x_1-x_2|\int_{\U_\alpha}  h_\alpha g_P,
\end{equation*}
where $g_P$ is, as  usual, the leafwise Poincar\'e metric restricted to  $\U_\alpha\subset L_\alpha.$ 
Applying  the  last inequality  to  $x_1=\alpha_1$ and $x_2=\alpha_2$ and integrating both sides
with respect  to $ d\nu(\alpha)$ and  using the  above integral representation of $T,$
we obtain  that
\begin{equation*}
\int_{\alpha\in  \T_{x_0}} |h_\alpha(\alpha_1)-h_\alpha(\alpha_2)|d\nu(\alpha)\leq    
c{ (\log^*{|x_0|})^2\| T\wedge g_P\|_{\U} \over  \|x_0\|}   \sup_{\alpha\in \T_{x_0}}   \|\alpha_1- \alpha_2\|.
 \end{equation*}
 Applying  Proposition     \ref{P:local_inter},  we  see   that the  left hand side is bigger than
 $ \| [\Df_1]\wedge T-[\Df_2]\wedge T\|_\V,$ and hence  the proposition follows.
\end{proof}


\noindent{\bf End of  the proof of estimate \eqref{e:mass_estimates_2} in Proposition  \ref{P:mass_1}.}
  Fix $N\in\N$ large  enough  and let $0<r_N<1$  be the constant given  by Lemma \ref{L:near_C_0,N}.
  Set $s:=r_N.$ By Remark \ref{R:sD_red} we can reduce  estimate \eqref{e:mass_estimates_2} to the following one:
\begin{equation}\label{e:mass_estimates_2_red}
\Big| \| T\wedge [\Cf_{r,N}]   \|_{  (s\D)^2\setminus  \B_{r^{1/N}|\log r|^{3/N}}}-\| T\wedge [ \Cf_{0,N}] \|_{  (s\D)^2\setminus  \B_{r^{1/N}|\log r|^{3/N}}}\Big |\leq c  |\log r|^{-1}, 
\end{equation}
where  $c$ is a  constant  which depends only on $N.$
Fix $\rho_2$ with $0<\rho_2\ll\rho_0$ and $2\pi\rho_2^{-1}\in\N.$
Consider  the  countable set 
\begin{multline*}
\X:=\left\lbrace  x=(z,w)\in(s\D)^2:\ \  |z|,|w|\in \{  s(1-\rho_2)^p:\ p\in\N\}\right.\\
\left.\qquad \text{and}\qquad
\arg z,\arg w\in \{0,\rho_2, 2\rho_2\ldots,2\pi\}\right\rbrace.
\end{multline*}
Consider  the family $\Cc$ of cells $\Cell(x,\rho_0),$ where $x\in\X.$
We see easily that when $\rho_2$ is small enough and when a constant $0<\rho_3\ll \rho_1$ is   small  enough, 
the following property holds:

\noindent {\bf Property (i).}  {\it
  For every point $x\in (s\D\setminus\{ 0\})^2,$ there  exists at least one cell $C\in\Cc$
  such that    $\Cell(x,\rho_3)\subset C.$ 
 }

In particular,  Property (i) implies that
$$(s\D)^2\setminus(\{z=0\}\cup \{w=0\})=\bigcup_{x\in\X} \Cell(x,\rho_0).$$
Moreover, we can check that  the following  property  also  holds.
 
\noindent {\bf Property (ii).}  {\it There is $K\in\N$ such that each point in $(s\D\setminus\{0\})^2$ belongs to at most $K$
cells in the family $\Cc.$}    

Let $0<r<r_N$ be  arbitrary. 
 Combing Lemma \ref{L:near_by_z=w^N} and  Lemma \ref{L:near_C_0,N}, we see that  
 for every $\alpha\in\T$ and  $ k\in \N$   with $\xi_{0,N,\alpha,k} \in (s\D)^2\setminus \B_{r^{1/N}|\log r|^{3/N}}, $
 \begin{equation}\label{e:nearby}
\dist_C(\xi_{r,N,\alpha,k},\xi_{0,N,\alpha,k})\lesssim  |\log r|^{-3}\quad\text{and}\quad
\dist_C(\xi_{N,\alpha,k},\xi_{0,N,\alpha,k})\approx N^{-1}.
\end{equation}
Fix   $\alpha_0\in\T.$ 
 Suppose that there is   $k\in\N$ such that
$\xi_{0,N,\alpha_0,k}\in (s\D)^2\setminus \B_{r^{1/N}|\log r|^{3/N}}.$
By   Property (i), we may find 
 a cell $C=C_k\subset \Cc$   such that $C':=\Cell(\xi_{0,N,\alpha_0,k},\rho_3)\subset C.$
 Set $\rho_4:= 1/4 c_0^{-2}\rho_3,$ where the constant $c_0>1$ is introduced  just before Definition \ref{D:cell}. 
 Next, using \eqref{e:nearby} we can  check that there  are constants  $c'>1$ and $0<\rho_5\ll 1$  depending  only on $\rho_0,\rho_1,\rho_2,\rho_3$ 
(in particular, they are independent of $r$ and $N$) with the  following property: 

There is an open ball  $W_{\alpha_0}$ with center $\alpha_0\in\T$ and  radius $\rho_5$ in $\T,$ where $\T$ is  defined in  \eqref{e:cano_trans_T},  
 and  an interval  $S_{r,\alpha_0}\subset \N$  such that 

$\bullet$  $c'^{-1}N\leq \#S_{r,\alpha_0} \leq c' N,$ where $\#$ denotes the cardinality;

$\bullet$
   for every $k\in S_{r,\alpha_0}$ and $\alpha\in W_{\alpha_0},$
all  three points $\xi_{N,\alpha,k},\,\xi_{0,N,\alpha,k},\,\xi_{r,N,\alpha,k}  $ not only  belong to 
the bidisc $(s\D)^2,$  but also belong to the cell $C'':=\Cell(\xi_{0,N,\alpha_0,k},\rho_4).$

Let $\T_{C''}$ be a transversal of $C''$ in the sense of Definition \ref{D:cell}. For  every $W\subset\T,$ let
$\widehat{L}_W:=\bigcup_{\alpha\in W} \widehat{L}_\alpha.$
 Set  $$\T_{C'',\alpha_0}:= \T_{C''}\cap \widehat{L}_{W_{\alpha_0}} .$$ This is    a nonempty open subset of $\T_{C''}.$
This, combined  
 with \eqref{e:nearby},  allows us to apply Proposition \ref{P:two_curves_in_cell}
 to  two algebraic curves $\frak{C}_{0,N}$ and $\frak{C}_{r,N}$ in the cells $C'$ and $C''.$
 The  only change is  that we use  the constants $\rho_3,\rho_4$ instead of $\rho_0,\rho_1.$
 Consequently, we get  a constant $c'>0$ such  that 
\begin{multline*}
\int_{\alpha\in \T_{C'',\alpha_0} } |h_\alpha(\xi_{0,N,\alpha,k}  )-h_\alpha(\xi_{r,N,\alpha,k} )|d\nu(\alpha)\\
\lesssim    
 c' {|\log{| \xi_{0,N,\alpha_0,k_0}|}|^2\| T\wedge g_P\|_{C'} \over  \| \xi_{0,N,\alpha_0,k_0} \| }  \cdot \sup_{\alpha\in \T}   \|\xi_{0,N,\alpha,k}-\xi_{r,N,\alpha,k} \|.
 \end{multline*}
Using  the first estimate in \eqref{e:nearby} and the inequality
$ | \xi_{0,N,\alpha_0,k_0}|\geq  r^{1/N}|\log r|^{3/N},$ the  right hand side is bounded from above by a constant times
$|\log r|^{-1}\| T\wedge g_P\|_C.$ On the other hand, the  left hand side
is bounded from below by  $|[\frak{C}_{0,N}]\wedge T -[\frak{C}_{r,N}]\wedge T\|_{C \cap \widehat{L}_{W_{\alpha_0}}}.$
Hence, for $C=C_k$ we have that 
$$|[\frak{C}_{0,N}]\wedge T -[\frak{C}_{r,N}]\wedge T\|_{C''\cap  \widehat{L}_{W_{\alpha_0}}}\lesssim |\log r|^{-1}\| T\wedge g_P\|_C.$$
Summing up the last inequality over all $k\in S_{r,\alpha_0},$  
 and using Property (ii)  above, we  get that
 $$|[\frak{C}_{0,N}]\wedge T -[\frak{C}_{r,N}]\wedge T\|_{ \widehat{L}_{W_{\alpha_0}}}\lesssim |\log r|^{-1}\| T\wedge g_P\|_X.$$
 A  compactness argument shows that  we can cover $\overline\T$
 by a finite number of open sets $W_{\alpha_0}.$
Applying the above  estimate  to each element of this  cover and summing up the obtained estimates,
\eqref{e:mass_estimates_2_red} follows  and   we are done.
\hfill $\square$

\noindent{\bf End of  the proof of estimate \eqref{e:mass_estimates_3} in Proposition  \ref{P:mass_1}.}
We  argue  as in the above  proof of  estimate \eqref{e:mass_estimates_2} in Proposition  \ref{P:mass_1}
using 
Lemma \ref{L:discrepancy_C_r,N_and_C_r}
instead  of Lemma    \ref{L:near_C_0,N}. We only  point out here the  necessary  modification.
Fix   $\alpha_0\in\T.$ 
 Suppose that there is   $k\in\N$ such that
$\xi_{r,\alpha_0,k}\in \B_{r^{1/N}|\log r|^{-3/N}}.$
By   Property (i) in the previous proof, we may find 
 a cell $C=C_k\subset \Cc$   such that $\Cell(\xi_{r,\alpha_0,k},\rho_3)\subset C.$
 Thus, there is an open ball   $W_{\alpha_0}$ with center $\alpha_0$ and  radius $\rho_5$ in $\T$
 and  an interval $S_{r,\alpha_0}\subset \N$  such that  $c'^{-1}N\leq \#S_{r,\alpha_0} \leq c' N$ points,
and that 
   for every $k\in S_{r,\alpha_0}$ and $\alpha\in W_{\alpha_0},$
the two points $\xi_{r,\alpha,k}$ and $\xi_{r,N,\alpha,k}  $ belong to the cell $C'':=\Cell(\xi_{r,\alpha_0,k},\rho_4).$
 
 Next, using  Lemma \ref{L:discrepancy_C_r,N_and_C_r} instead of the first estimate in \eqref{e:nearby},
 we conclude the proof  as in  the  previous  one.
\hfill $\square$

\section{Mass of  $T\wedge [\frak{C}_r],$   $T\wedge [\frak{C}_{0,N}],$
 $T\wedge [\frak{C}_{r,N}]$ on the corona   $\A_{r,N}$}
 \label{S:on_corona}
   
  The objective of this  section is  to establish Proposition \ref{P:mass_2}. 
Recall   that for every $s>0,$  the function $K_s:\ \R\to\R^+$ is  given by
 \eqref{eq_K_s}.

  \subsection{Mass of  $T\wedge [\frak{C}_r]$ on $\A_{r,N}$}
  In order to prove the first  inequality of this  proposition,  we  consider, for each $s>1$ and $N\in\N\setminus\{0\},$ the following  domain in $\R:$
\begin{equation}\label{e:D_s_N}
  D_{s,N}:=\{ t\in\R^+:\   N^{-1}(s-3\log s)\leq  t\leq   N^{-1}( s+3\log s)\},
 \end{equation}
 and the  function $K^{(1)}_{s,N}:\   \R\to\R^+$ given  by
\begin{equation}\label{eq_K^1_s,N}
K^{(1)}_{s,N}(y):=     \int_{D_{s,N} }{V\over V^2+( y-U)^2}  dt,\qquad y\in \R,
\end{equation}
Here $U,$ $V$  are functions of the variable $t$ and the parameter $s$ which satisfy the following  system
of equations (see \eqref{eq_u,v_vs_U,V},  \eqref{e:t}  and \eqref{eq_u,v_vs_z,w}):
  $$U+iV=(u+is)^\gamma\qquad\text{and}\qquad  t=bu+as.$$   
It is  worthy  comparing  the  domain  $D_{s,N}$ (resp. the function
$ K^{(1)}_{s,N}$) with  the  domains $D_s$  (resp. the function
$ \widetilde {K}_s$)
 given in  \eqref{e:widetilde_K_s}.
\begin{lemma}\label{L:K^1_s,N} For every $0<r<1/2,$
\begin{equation*}
 \| T\wedge [\Cf_r]\|_{    \A_{r,N}}\approx 
\int_{\alpha\in\T}\Big( \int_{-\infty}^\infty  K^{(1)}_{ - \log r,N}(y)\tilde H_\alpha (y)dy
\Big) d\nu(\alpha). 
\end{equation*}
\end{lemma}
\begin{proof}
Using  \eqref{e:para-inter_C_r}, 
\eqref{eq_parametrization_DS}  and \eqref{eq_u,v_vs_z,w}, we  see  that  $\xi_{r,\alpha,k}\in \A_{r,N}$
if and only if   
\begin{multline*}
{1\over 2\pi b}\Big( (N^{-1}-a)(-\log r)  -3\log (-\log r)+ \log {|\alpha|}\Big)\leq  k\\
\leq{1\over 2\pi b}\Big( (N^{-1}-a)(-\log r)  +3\log (-\log r)+ \log {|\alpha|}\Big).
\end{multline*}
Let $\Z^1_{r,N,\alpha}$  be the set of all  integers $k$ satisfying  the last inequalities.
Observe that
$$
 \| T\wedge [\Cf_r]\|_{    \A_{r,N}}=\int_{\alpha\in\T}\Big(\sum_{k\in \Z^1_{r,N,\alpha}} h_\alpha(\xi _{r,\alpha,k})\Big) d\nu(\alpha).
$$  
 For $\alpha\in\T$ let $D^\alpha:= D_{- \log r,N},$
  $t_{r,\alpha,k}:=bu_{\alpha,k}+av_r,$ $k\in \Z.$ Then  $k\in \Z_{r,N, \alpha}$ if and only if
$t_{r,\alpha,k}\in D^\alpha.$  Moreover, consider the function $\chi^\alpha:\ D^\alpha\to\C$ given by
$$
\chi^\alpha (t):=u+iv_r,\qquad\text{where}\quad v_r=-\log r\quad\text{and}\quad t=bu+av_r.
$$ 
Let $m=1$ and choose $\rho>1$  large  enough.
Using  Harnack's inequality, we see that the assumption  of Part 1) of  Proposition \ref{P:interpretation} 
is also fulfilled. In other words, by  Definition \ref{D:interpretation} we obtain an interpretation $(K^\alpha)_{\alpha\in\T}$
of the geometric intersection $T\wedge [\Cf_r]$ on $ \A_{r,N}   $ with mesh $m=1.$
 Consequently, applying  Part 1) of Proposition \ref{P:interpretation}    the lemma follows.
\end{proof}

\begin{lemma}\label{L:comparison_bis} There is a constant $c>1$ such that 
for all $y\in \R,$   $s>1$ and $N\in\N\setminus\{0\},$  
\begin{equation*}
        K^{(1)}_{s,N}(y)\leq c  s^{-1}(\log s)  K_{s}(y) . 
\end{equation*}
\end{lemma}
\begin{proof} We follow the method of 
 proof of Lemma \ref{L:comparison}. Let  $c_2,c_3$ be the constants with $c_3>c_2>1$ given by Lemma \ref{lem_Poisson_kernel}.
 We consider three cases.

\noindent {\bf Case 1:} $  s\geq c_2   (1+ | y|)^{1/\gamma} .$ 

By  Part 2) of Lemma \ref{lem_Poisson_kernel} and by formula \eqref{eq_K^1_s,N},
we have that
$$
K^{(1)}_{s,N}(y)\approx    \int_{t=s/N-3(\log s)/N}^{t=s/N+3(\log s)/N}  {(s/N) dt\over s^{\gamma+1}}    
\approx  N^{-2} s^{-\gamma}\log s.$$
This, compared with formula   \eqref{eq_K_s}, completes 
 the proof of   Case 1.

\noindent {\bf Case 2:}   $  c_3^{-1} (1+ | y|)^{1/\gamma} \leq  s\leq c_2 (1+ | y|)^{1/\gamma}.$ 

Applying  Part 5) of Lemma \ref{lem_Poisson_kernel}, we get   that
$$
K^{(1)}_{s,N}(y)\approx    \int_{t=s/N-3(\log s)/N}^{t=s/N+3(\log s)/N}    
  { (1+|y|)^{1/\gamma-1} tdt \over t^2+  (s-\rho(y,t) )^2 }, 
$$
where $\rho(y,t)$ satisfies  $c^{-1}_2(1+ | y|)^{1/\gamma} \leq \rho(y,t)\leq c_2(1+ | y|)^{1/\gamma}.$ 
A straightforward computation  shows that the right hand side is $\approx N^{-1}s^{-1}\log s (1+ | y|)^{1/\gamma-1}.$  Comparing this with formula   \eqref{eq_K_s},   the proof of Case 2 is complete.
   
\noindent {\bf Case 3:} $   s  \leq  c_3^{-1}     (1+ | y|)^{1/\gamma}.$

Applying  Part 1) and Part 3) of Lemma \ref{lem_Poisson_kernel}, we get that  
$$
K^{(1)}_{s,N}(y) \approx \int_{t=s/N-3(\log s)/N}^{t=s/N+3(\log s)/N}    {t^{\gamma-1}s    dt\over (1+|y|)^2}
\approx 
  N^{-\gamma}s^{\gamma}\log s (1+|y|)^{-2 }.$$
  This, compared with formula  \eqref{eq_K_s},  allows us to conclude the proof of the last case.
   \end{proof}
 \noindent {\bf End of the proof of inequality \eqref{e:mass_estimates_4}  in Proposition \ref{P:mass_2}.}
 Applying  Lemma \ref{L:K^1_s,N} and then Lemma  \ref{L:comparison_bis}, we  see that  for every $0<r<1/2,$
\begin{eqnarray*}
 \| T\wedge [\Cf_r]\|_{    \A_{r,N}}&\approx& 
\int_{\alpha\in \T}\Big( \int_{-\infty}^\infty  K^{(1)}_{ - \log r,N}(y)\tilde H_\alpha(y)dy
\Big) d\nu(\alpha)\\
&\lesssim& N^{-1}
 (-\log r)^{-1}\log(-\log r)  \int_{ \alpha\in \T}\Big( \int_{-\infty}^\infty  K_{ - \log r}(y)\tilde H_\alpha(y)dy
\Big) d\nu(\alpha).
\end{eqnarray*}
By Lemma \ref{lem_estimate_G} 
and  identity \eqref{e:Lelong}, the integral in the last line is uniformly bounded in $r.$
 The proof is thereby  completed.
\hfill $\square$

\subsection{Mass of  $T\wedge [\frak{C}_{0,N}]$ on $\A_{r,N}$}

  The next part of this  section is  devoted to the  proof of the second  inequality of Proposition \ref{P:mass_2}.
  We  consider, for each $s>1$ and $N\in\N\setminus\{0\},$ the  function $K^{(2)}_{s,N}:\   \R\to\R^+$ given  by
\begin{equation}\label{eq_K_2_star_s,N}
K^{(2)}_{s,N}(y):=  N^2   \int_{ D_{s,N} }{V\over V^2+( y-U)^2}  dt,\qquad y\in \R.
\end{equation}
Here the domain  $D_{s,N}$  is  given in \eqref{e:D_s_N}, and $U,$ $V$  are functions of the variable $t$   which satisfy the following  system
of equations (see \eqref{eq_u,v_vs_U,V},  \eqref{e:t}  and \eqref{eq_u,v_vs_z,w}):
  $$U+iV=(u+iv)^\gamma\quad\text{and}\quad  t=bu+av \quad\text{and}\quad v=Nt+\log {|d|},$$ 
 where $d:=-a_N$ (see \eqref{e:z_N_expansion} and 
Lemma \ref{L:near_by_z=w^N} for $a_N$).

\begin{lemma} \label{L:K^star,1_s,N} For every $0<r<1/2,$
\begin{equation*}
 \| T\wedge [\Cf_{0,N}]\|_{    \A_{r,N}}\approx 
\int_{\alpha\in\T}\Big( \int_{-\infty}^\infty  K^{(2)}_{ - \log r,N}(y)dy
\Big) d\nu(\alpha). 
\end{equation*}
\end{lemma}
\begin{proof}
Using  \eqref{e:para-inter_z=w^N}, \eqref{e:system_u_and_v}, \eqref{e:t_N_alpha_k},
\eqref{eq_parametrization_DS}, \eqref{e:t}  and \eqref{eq_u,v_vs_z,w}, we  see  that  $ \xi_{N,\alpha,k}\in \A_{r,N}$
if and only if   
\begin{equation}\label{e:cond_t_N_alpha_k}
 (-\log r)  -3\log (-\log r) \leq  t_{N,\alpha,k}
\leq (-\log r)  +3\log (-\log r) .
\end{equation}
Let $\Z^1_{N,\alpha}$  be the set of all  integers $k$ satisfying  the last inequalities.
Observe that
\begin{eqnarray*}
 \| T\wedge [\Cf^d_{N}]\|_{    \A_{r,N}}&=&\int_{\alpha\in\T}\Big(\sum_{k\in \Z^1_{N,\alpha}} h_\alpha(\xi_{N,\alpha,k})\Big) d\nu(\alpha)\\
 \| T\wedge [\Cf_{0,N}]\|_{    \A_{r,N}}&=&\int_{\alpha\in\T}\Big(\sum_{k\in \Z^1_{N,\alpha}} h_\alpha(\xi_{0,N,\alpha,k})\Big) d\nu(\alpha).
\end{eqnarray*}  
On the other hand, by Lemma \ref{L:near_by_z=w^N}
we know that  $\xi_{N,\alpha,k}$ and  $\xi_{0,N,\alpha,k}$ are compatible for $k\in \Z^1_{N,\alpha}.$
Hence, by Harnack's inequality,  there is a constant $c>1$  such that
$$
 c^{-1}h_\alpha(\xi_{N,\alpha,k})\leq h_\alpha(\xi_{0,N,\alpha,k})\leq c h_\alpha(\xi_{N,\alpha,k}).
$$
This, combined with the above equalities, implies that
$$
\| T\wedge [\Cf_{0,N}]\|_{    \A_{r,N}}\approx \int_{\alpha\in\T}\Big(\sum_{k\in \Z^1_{N,\alpha}} h_\alpha(\xi_{N,\alpha,k})\Big) d\nu(\alpha).
$$
So we need to show that the right hand side in the last line  is  equivalent to the right hand side of the lemma.
  For $\alpha\in\T$ let $D^\alpha:= D_{-\log r,N},$
  $t_{N,\alpha,k}:=bu_{N,\alpha,k}+av_{N,\alpha,k},$ $k\in \Z.$ Then 
   $k\in \Z^1_{N,\alpha}$ if and only if
$t_{N,\alpha,k}\in D^\alpha.$
   Moreover, consider the function $\chi^\alpha:\ D^\alpha\to\C$ given by
$$
\chi^\alpha (t):=u_N(t)+iv_N(t),\quad\text{where}\quad v_N(t)=Nt+\log{|d|}\quad\text{and}\quad t=bu_N(t)+av_N(t).
$$ 
By \eqref{e:t_N_alpha_k}, \eqref{e:difference_u_v} and \eqref{e:approx_pace_u_v}, we  choose
 $m=N^{-2}.$ Moreover, take $\rho>1$  large  enough.
Using  Harnack's inequality, we see that the  assumption of  Part 1)  of  Proposition \ref{P:interpretation} 
is  fulfilled. Hence,   we obtain an interpretation $(K^\alpha)_{\alpha\in\T},$ in the sense of  Definition \ref{D:interpretation},
of the geometric intersection $T\wedge [\Cf_r]$ on $ \A_{r,N}   $ with mesh $m=N^{-2}.$
 Consequently, applying  Part 1) of Proposition \ref{P:interpretation}    the lemma follows. 
\end{proof}

\begin{lemma}\label{L:comparison_bis_bis} There is a constant $c>1$ such that 
for all $y\in \R$  and $s>1$ and $N\in\N\setminus\{0\},$  
\begin{equation*}
        K^{(2)}_{s,N}(y)\leq c N^2s^{-1}(\log s)  K_{s}(y) . 
\end{equation*}
\end{lemma}
\begin{proof} We follow the method of 
 proof of Lemma \ref{L:comparison}. Let  $c_2,c_3$ be the constants with $c_3>c_2>1$ given by Lemma \ref{lem_Poisson_kernel}.
 We consider three cases.

\noindent {\bf Case 1:} $  s\geq c_2   (1+ | y|)^{1/\gamma} .$ 

By  Part 2) of Lemma \ref{lem_Poisson_kernel} and by formula (\ref {eq_K_2_star_s,N}),
we have that
$$
K^{(2)}_{s,N}(y)\approx    N^2\int_{t=s/N-3(\log s)/N}^{t=s/N+3(\log s)/N}  {t dt\over (Nt)^{\gamma+1}}    
\approx   s^{-\gamma}\log s.$$
This, compared with formula   \eqref{eq_K_s}, completes 
 the proof of   Case 1.

\noindent {\bf Case 2:}   $  c_3^{-1}  (1+ | y|)^{1/\gamma}\leq  s\leq c_2 (1+ | y|)^{1/\gamma}.$ 

Applying  Part 5) of Lemma \ref{lem_Poisson_kernel}, we get   that
$$
K^{(2)}_{s,N}(y)\approx   N^2 \int_{t=s/N-3(\log s)/N}^{t=s/N+3(\log s)/N}    
  { (1+|y|)^{1/\gamma-1} tdt \over t^2+  (Nt-\rho(y,t) )^2 }, 
$$
where $\rho(y,t)$ satisfies  $c^{-1}_2(1+ | y|)^{1/\gamma} \leq \rho(y,t)\leq c_2(1+ | y|)^{1/\gamma}.$ 
A straightforward computation  shows that the right hand side is $\approx N s^{-1}\log s (1+ | y|)^{1/\gamma-1}.$ On the other hand, by  formula   \eqref{eq_K_s}  we have that 
   $K_s(y)\approx  (1+|y|)^{1/\gamma-1}.$  This  completes  the proof  of Case 2.
   
\noindent {\bf Case 3:} $   s  \leq  c_3^{-1}     (1+ | y|)^{1/\gamma}.$

Applying  Part 1) and Part 3) of Lemma \ref{lem_Poisson_kernel}, we get that  
$$
K^{(2)}_{s,N}(y) \approx N^2\int_{t=s/N-3(\log s)/N}^{t=s/N+3(\log s)/N}    {t^{\gamma-1} (Nt)    dt\over (1+|y|)^2}
\approx 
  N^{2-\gamma}s^{\gamma}\log s (1+|y|)^{-2 }.$$ 
 Since  we know  by  formula   \eqref{eq_K_s}  that 
   $K_s(y)\approx  (1+|y|)^{1/\gamma-1},$ 
 the proof of   the last case, and hence the lemma, is  thereby completed.
   \end{proof}
 \noindent {\bf End of the proof of inequality \eqref{e:mass_estimates_5}  in Proposition \ref{P:mass_2}.}
 Applying  Lemma \ref {L:K^star,1_s,N}  and then Lemma  \ref{L:comparison_bis_bis}, we  see that  for every $0<r<1/2,$
\begin{eqnarray*}
 \| T\wedge [\Cf_{0,N}]\|_{    \A_{r,N}}&\approx& 
\int_{ \alpha \in\T}\Big( \int_{-\infty}^\infty  K^{(2)}_{ - \log r,N}(y)\tilde H_\alpha(y)dy
\Big) d\nu(\alpha)\\
&\lesssim& N
 (-\log r)^{-1}\log(-\log r)  \int_{ \alpha \in\T}\Big( \int_{-\infty}^\infty  K_{ - \log r}(y)\tilde H_\alpha(y)dy
\Big) d\nu(\alpha).
\end{eqnarray*}
By Lemma \ref{lem_estimate_G} 
and  identity \eqref{e:Lelong}, the integral in the last line is uniformly bounded in $r.$
 The proof is thereby  completed.
\hfill $\square$

\subsection{Mass of  $T\wedge [\frak{C}_{r,N}]$ on $\A_{r,N}$}

  The remainder of the  section is  devoted to   the proof of inequality \eqref{e:mass_estimates_6}  in Proposition \ref{P:mass_2}.   

Let $0<\kappa=\kappa_N\ll 1$ be  a  very small  constant whose  exact value will be determined later on, and let $r_N>0$ be
the 
 constant satisfying both   Lemmas \ref{L:near_C_0,N} and  \ref{L:discrepancy_C_r,N_and_C_r}. 
Write 
\begin{equation}\label{e:mass_C_r,N_into_3_parts}
\| T\wedge [\Cf_{r,N}]\|_{    \A_{r,N}} =
\| T\wedge [\Cf_{r,N}]\|_{\A_{r,N}\cap \B_{\kappa r^{1/N}}} +
\| T\wedge [\Cf_{r,N}]\|_{    \A_{r,N}\setminus \B_{\kappa r^{1/N}}}=:I+II.
\end{equation}
 Arguing as in the proof of   Lemma  
 \ref{L:discrepancy_C_r,N_and_C_r}  and replacing the ball  $ \B_{r^{1/N} |\log r|^{-3/N}}$
 with  $ \B_{\kappa r^{1/N} }$ and choosing $0<\kappa<1$ small enough, we obtain  the following  weaker result for every $0<r<r_N$
 and $\alpha\in\T:$
 The intersection of the curve $\frak{C}_{r,N}$ with the Riemann surface
 $\widehat{L}_\alpha$ inside the ball $\B_{\kappa r^{1/N}}$ can be enumerated as $\xi_{r,N,\alpha,k} $
such that   $\xi_{r,N,\alpha,k}$ and $\xi_{r,\alpha,k}$ are compatible, where 
 $  k\in\N$ such that $ \xi_{r,\alpha,k}\in \B_{\kappa r^{1/N} } .$
  Consequently, we get that
$$
 I=\int_{ \alpha\in\T}\Big(\sum h_\alpha( \xi_{r,N,\alpha,k})\Big) d\nu(\alpha),
$$  
where  the sum is  taken  over all $k\in\N$ such that
 $\zeta_{r,\alpha,k}\in    \A_{r,N}\cap  \B_{\kappa r^{1/N}}. $
  Moreover, using that   $ \zeta_{r,N,\alpha,k}$ and $ \zeta_{r,\alpha,k}$ are compatible,   an application of  Harnack's inequality gives that
 $$h_\alpha( \xi_{r,N,\alpha,k})\leq c'h_\alpha(\xi_{r,\alpha,k} )\quad\text{for some constant}\quad c'>1.
$$
Therefore, we infer that
 $$
 I\leq c'\int_{ \alpha\in\T}\Big(\sum h_\alpha( \xi_{r,\alpha,k} )\Big) d\nu(\alpha)=
c'\| T\wedge [\Cf_r]\|_{    \A_{r,N}\cap \B_{\kappa r^{1/N}}}.
$$ 
The right hand side is  bounded from above by $c'\| T\wedge [\Cf_r]\|_{    \A_{r,N}}.$
This,  coupled with \eqref{e:mass_estimates_4}, implies that
\begin{equation}\label{e:est_I}
I\leq c(-\log r)^{-1}\log(-\log r) .
\end{equation}

Next, we turn to  $(II).$
Observe  that every point $(z,w)\not\in\B_{\kappa r^{1/N}}   $ with $z_N(z,w)=r$
satisfies the assumption of  
 Proposition \ref{P:z_N_expression} for  the sequence $M_N:=8^N$ as  in \eqref{e:choice_K_N}.
  Therefore,  we have for  such  a point that
 \begin{equation}\label{e:case_3_on_corona}
 |z_N(z,w)  -(z_\infty(z,w)+a_Nw^N)|\leq  8^{-N}|a_N||w|^N.
 \end{equation}
Consequently, we can argue as in the proof of   Lemma  
 \ref{L:near_C_0,N} while replacing the ball  $ \B_{r^{1/N} |\log r|^{3/N}}$
 with  $ \B_{\kappa r^{1/N} }.$  Thus we obtain  the following  weaker fact   than    Lemma  
 \ref{L:near_C_0,N}.

 \medskip
 
 \noindent  {\bf Claim.  }{\it For $N$  large  enough, there  exist two numbers $\Gamma_N$ and $\Lambda_N$  such that by reducing $r_N$ if necessary, 
  for every $0<r<r_N$
 and $\alpha\in\T,$    the following two properties hold:
 \begin{itemize}
  \item [(i)]
 for every point $\xi_1\in (\frak{C}_{0,N}\cap \widehat{L}_\alpha) \cap( \A_{r,N}\setminus \B_{\kappa r^{1/N} }),$
 there  exist at least  one point and  at most  $\Gamma_N$ points $\xi_2\in (\frak{C}_{r,N}\cap \widehat{L}_\alpha) \cap( \A_{r,N}\setminus \B_{\kappa r^{1/N} }) $   
such that  $\xi_1$ and $\xi_2$ are quasi-compatible in the sense of 
Definition \ref{D:points_compatible}
and $\dist_C(\xi_1,\xi_2)\leq \Lambda_N;$
   \item[(ii)]
   for every point $\xi_1\in (\frak{C}_{r,N}\cap \widehat{L}_\alpha) \cap( \A_{r,N}\setminus \B_{\kappa r^{1/N} }),$
 there  exist at least  one point and  at most  $\Gamma_N$ points $\xi_2\in (\frak{C}_{0,N}\cap \widehat{L}_\alpha) \cap( \A_{r,N}\setminus \B_{\kappa r^{1/N} }) $   
such that  $\xi_1$ and $\xi_2$ are quasi-compatible and $\dist_C(\xi_1,\xi_2)\leq \Lambda_N.$
   \end{itemize}

  }

 \medskip
 
 \noindent
{\bf Sketchy proof of the  claim.}  We only prove   assertion (i) since assertion (ii) can be done similarly.
Unlike  the  proof of  Lemma  \ref{L:near_C_0,N}, $s$ in this  claim  is a large positive number.  
Arguing as  in the  proof of  \eqref{e:r(logr)-3}, we may find a constant $c_N>1$  such that 
\begin{equation*}
r<  c_N\kappa^{-1}|a_Nw_1^N|.
\end{equation*}
Now  we choose $M$ large enough ($M$ depending on $N$). In fact, instead of \eqref{e:s}
   $s$ in this  claim    is  of the  form
\begin{equation}\label{e:s_bis}
s:=c'  \kappa^{-1} \qquad \mbox{for $ c'=c'_N>0$ a large constant independent of $r$}.
\end{equation}
As in the  proof of  Lemma  \ref{L:near_C_0,N} we want  to estimate the number of  roots $t$ of
the  following holomorphic function on $\D_s$    defined   by \eqref{e:F(t)_bis}:
  \begin{equation*}
 F(t):=R_r(e^{\lambda t}w_1)  -\tilde{a}_Ne^{\lambda Nt}w_1^N+h(e^{\lambda t}w_1)+ \tilde{a}_Ne^tw_1^N -e^th(w_1),
\qquad   t\in \D_s.
 \end{equation*}
Consider again the  function 
\begin{equation*}
H(t):=-\tilde{a}_Ne^{\lambda Nt}w_1^N +\tilde{a}_Ne^tw_1^N,\qquad t\in  \overline{\D}_s.
\end{equation*}
Since  $H(t)=0$ if and only if  $t={2i\pi k\over \lambda N-1}$ for $k\in \Z,$
 we may choose    $s$ and $c'$  large enough (depending only on $N$ and $\lambda$) such that 
$$
|H(t)|\geq   c'^{-1} |\tilde{a}_Nw_1^N|\qquad\text{for}\qquad  t\in \partial\D_s,
$$
and that  
 $H$ admits a finite  number of roots, say  $\Gamma_N\geq 1$ roots  on $\D_s.$
Using this and   
 \eqref{theta_N}, \eqref{e:f_theta}, \eqref{e:quotient},   \eqref{e:R_r}, \eqref{e:r(logr)-3}, \eqref{e:s_bis} and  \eqref{e:R_r_bis},
 we can  show that
 $$
 |F(t)-H(t)|<H(t)\qquad\mbox{for $t\in\partial\D_s.$}
 $$
 Consequently, 
applying  Rouch\'e's theorem again to $F$ and $H,$
the claim follows with $\Lambda_N:=s.$ 
\hfill $\square$

\medskip

Using the  claim  we may find a constant $c''=c''_N>1$ such  that
$$
 II\leq c''\int_{ \alpha\in\T}\Big(\sum h_\alpha(\xi_{0,N,\alpha,k})\Big) d\nu(\alpha)=
c''\| T\wedge [\Cf_{0,N}]\|_{    \A_{r,N}\setminus \B_{\kappa r^{1/N}}},
$$ 
where  the sum is taken over all $k\in\N$
such that $\xi_{0,N,\alpha,k}\in \A_{r,N}\setminus \B_{\kappa r^{1/N}}.$
The right hand side is  bounded from above by $c''\| T\wedge [\Cf_{0,N}]\|_{    \A_{r,N}}.$
%
 %
%
This,  coupled with \eqref{e:mass_estimates_5}, implies that
\begin{equation}\label{e:est_II}
II\leq c(-\log r)^{-1}\log(-\log r) .
\end{equation}

 \noindent {\bf End of the proof of inequality \eqref{e:mass_estimates_6}  in Proposition \ref{P:mass_2}.}
  Putting  \eqref{e:mass_C_r,N_into_3_parts}, \eqref{e:est_I} and \eqref{e:est_II} altogether,  
 \eqref{e:mass_estimates_6} follows.
\hfill $\square$

\section{Completion of the  reductions}\label{S:completion}

 In the first part of this   section  we complete the proof of    Proposition \ref{P:key_interpretations}.
More specifically, we will show  that 
the  geometric intersection $T\wedge [\Cf_{0,N}]$ on  $    \B_{r^{1/N}|\log r|^{-3/N}}$ admits a coherent interpretation
  $K^*_{-\log r,N}$  satisfying   the conclusion  of
this proposition. As a consequence, the  second part  is devoted  to the proof of   Proposition \ref{P:comparison_bis}.

\subsection{End of the proof of    Proposition \ref{P:key_interpretations}}
The proof is  divided into 
4 steps.

\noindent  {\bf Step 1:} {\it Construction of a coherent interpretation with mesh $N^{-2}.$}

Let $d:=-a_N$ (see \eqref{e:z_N_expansion} and 
Lemma \ref{L:near_by_z=w^N} for $a_N$).
We  consider, for each $s>1$ and $N\in\N\setminus\{0\},$
   the following  domain in $\R:$
\begin{equation}\label{e:D^star_s_N}
  D^*_{s,N}:=\{ t\in\R^+:\     t\geq   N^{-1}( s-3\log s)\},
 \end{equation}
and the  function $K^{*}_{s,N}:\   \R\to\R^+$ given  by
\begin{equation}\label{eq_widetilde_K_star_s,N}
K^{*}_{s,N}(y):=  N^2   \int_{ D^*_{s,N} }{V\over V^2+( y-U)^2}  dt,\qquad y\in \R.
\end{equation}
Here  $U,$ $V$  are functions of the variable $t$   which satisfy the following  system
of equations (see \eqref{eq_u,v_vs_U,V},  \eqref{e:t}  and \eqref{eq_u,v_vs_z,w}):
  $$U+iV=(u+iv)^\gamma\quad\text{and}\quad  t=bu+av \quad\text{and}\quad v=Nt+\log {|d|}.$$ 
For $\alpha\in\T,$ set $D^\alpha:= D^*_{-\log r,N},$ and  $K^\alpha:=K^{*}_{-\log r,N},$
and $\chi^\alpha(t)=u_\alpha(t)+iv_\alpha(t),$ $t\in D^\alpha.$ Here $u_\alpha$ and $ v_\alpha$ are affine functions in $t$
such that $u_\alpha(t_{N,\alpha,k})=u_{N,\alpha,k}$ and  $v_\alpha(t_{N,\alpha,k})=v_{N,\alpha,k}$ for $ k\in\N$  (see \eqref{e:system_u_and_v}--\eqref{e:t_N_alpha_k} in Lemma \ref{L:para-inter_z=w^N}).
We will show that $(K^\alpha)_{\alpha\in\T}$ is a coherent interpretation of the geometric intersection
 $T\wedge [\Cf_{0,N}]$ on  $    \B_{r^{1/N}|\log r|^{-3/N}}.$

By Lemma  \ref{L:near_by_z=w^N} and using  \eqref{eq_parametrization_DS} we know that  
 for every $\alpha\in\T,$ 
 each point $\xi=(z,w)=\psi_\alpha(u+iv)\in
   \Cf^d_N=\{z=dw^N\}\cap \widehat{L}_\alpha$   corresponds  to 
   a unique point $\xi=(z,w)=\psi_\alpha(u+iv)\in
    \frak{C}_{0,N}  \cap \widehat{L}_\alpha$  
such that $\xi $ and $\xi'$ are compatible  and 
$
  \dist_C(\xi_{N,\alpha,k},\xi_{0,N,\alpha,k})\lesssim N^{-1}.$ 
Using  Definition  \ref{D:points_compatible}   we infer from the last  inequality that
\begin{equation}\label{u_u'_v_v'}
|u-u'|\lesssim N^{-1}\qquad \text{and}\qquad |v-v'|\lesssim N^{-1}.
\end{equation}
On the other hand,  by \eqref{e:difference_u_v}--\eqref{e:approx_pace_u_v} in Lemma \ref{L:para-inter_z=w^N}, 
if $u:=u_{N,\alpha,k},$ $v:=v_{N,\alpha,k},$ $u':=u_{N,\alpha,k+1},$ $v':=v_{N,\alpha,k+1}$
for some $k\in\N,$ we also  get    inequality \eqref{u_u'_v_v'}.

Using  \eqref{e:t} and  \eqref{eq_u,v_vs_U,V}, we set
$$
t=bu+av,\quad t'=bu'+av',\quad U+iV=(u+iv)^\gamma,\quad U'+iV'=(u'+iv')^\gamma.
$$
This, combined with \eqref{u_u'_v_v'}, yields that $t\approx t'$ and $v\approx v'.$
Since $\xi\in \Cf^d_N=\{z=dw^N\}\cap \widehat{L}_\alpha\cap \B_{r^{1/N}|\log r|^{-3/N}},$
it follows from \eqref{eq_u,v_vs_z,w} that
$$t\geq  -N^{-1}\log{ r}+3N^{-1}\log{(-\log r)}\quad\text{and}\quad v=Nt+\log{|d|}.
$$
Note that the second estimate in \eqref{e:approx_pace_u_v} shows that the mesh of $(K^\alpha)_{\alpha\in\T}$
should be  $N^{-2}.$
Using the  above estimates for $t,v$ and $t',v'$ and  applying   Lemma  \ref{lem_Poisson_kernel}, we can show that there is a constant $c>1$ independent of the above points $\xi,\, \xi'$
such that
$$
c^{-1}{V\over V^2+(y-U)^2}\leq {V'\over V'^2+(y-U')^2}\leq c{V\over V^2+(y-U)^2}\quad\text{for}\quad y\in\R.
$$ 
Therefore, by Definition \ref{D:interpretation}, $(K^\alpha)_{\alpha\in\T}$ is a coherent interpretation of the geometric intersection
 $T\wedge [\Cf_{0,N}]$ on  $    \B_{r^{1/N}|\log r|^{-3/N}}$ as  desired.
This completes Step 1.

\noindent  {\bf Step 2:} {\it 
There  are constant $c,\kappa>1$ independent of $N$
 such that 
for $ \kappa (1+|y|)^{1/\gamma}\leq  s,$ we have 
$ K^*_{s,N}(y) \leq  c     s^{1-\gamma}.$}

To start 
Step 2,
 let  $c_2,c_3$ be the constants with $c_3>c_2>1$ given by Lemma \ref{lem_Poisson_kernel}.
 Set  
$$\kappa:=\max\{c_2,c_3\}.$$
 By  Part 2) of Lemma \ref{lem_Poisson_kernel} and by formula  \eqref{eq_widetilde_K_star_s,N},
we have that
$$
K^{*}_{s,N}(y)\approx    \int_{t=s/N-3(\log s)/N}^{ \infty}  {N^2t dt\over (Nt)^{\gamma+1}}    
\approx   s^{1-\gamma} .$$
This  completes 
 the proof of  Step  2.

\noindent  {\bf Step 3:} {\it  There  are   constant $c,\kappa>1$ independent of $N$
 such that
for $ s\leq  \kappa^{-1}(1+|y|)^{1/\gamma}\geq  s,$ we have 
$c^{-1}\leq {K^*_{s,N}(y)\over   N(1+ | y|)^{1/\gamma -1}}\leq c  .$}

Let $\kappa$ be   given by  Step 2.
Applying  Part 5) of Lemma \ref{lem_Poisson_kernel} and using Step 2 above for $s =\kappa (1+|y|)^{1/\gamma},$ we get   that
\begin{equation}\label{e:widetilde_K_star_s,N_Case_3}
K^*_{s,N}(y)\approx (1+|y|)^{1/\gamma -1}+    \int_{t=s/N-3(\log s)/N}^{(1+|y|)^{1/\gamma}/N}    
  { (1+|y|)^{1/\gamma-1} N^2tdt \over t^2+ (Nt+\log {|d|}-  \rho(y,t))^2 }, 
\end{equation}
where $\rho(y,t)$  is defined as  follows: there exists  a    solution  $u:=u(y,t) $ of the following   equation
\begin{equation}\label{e:u(y,t)}
  U=y,\qquad U+iV=(u+it)^\gamma
  \end{equation}
 satisfying
   $ c_2^{-1}(1+|y|)^{1/\gamma}\leq u,\rho(y,t)\leq c_2(1+|y|)^{1/\gamma}$
  with $\rho(y,t):=bu+at.$ 
  
  For every $k=1,\ldots, N,$ let $t_k,u_k\in\R$ be  such that
\begin{equation}\label{e:t_k}
\rho(y,t_k)=k t_k\qquad\text{and}\qquad \rho(y,t_k):=bu_k+at_k.
\end{equation}
 Observe that
\begin{equation}\label{e:consequence_t_k}
(N-k-1)t+\log {|d|}\leq  Nt+\log {|d|}-\rho(y,t)\leq (N-k)t+\log {|d|}\quad\text{for}\quad t\in [t_{k+1},t_k].
\end{equation}
On the other hand,
we deduce from \eqref{e:t_k} and \eqref{e:u(y,t)}
that 
$$
t_k^\gamma \Re\big( (b^{-1}(k-a) +i)^\gamma \big)=\Re\big((b^{-1}(k-a) t_k+it_k)^\gamma\big)=\Re \big((u_k+it_k)^\gamma\big)= y.
$$
This, combined with the  estimate
$$
\Re\big ( (b^{-1}(k-a) +i)^\gamma\big )\approx  k^\gamma\quad\text{for large}\quad k,
$$
implies the following   estimates
$$
t_k\approx  k^{-1}|y|^{1/\gamma} ,
$$
and 
\begin{equation*}
\begin{split}
t_{k-1}-t_{k}&\approx  t_k^{\gamma+1}(t_k^\gamma -t_{k-1}^\gamma)\\
    &\approx  k^{-(\gamma+1)}|y|^{1+1/\gamma}{\Re\big( (b^{-1}(k-a) +i)^\gamma \big)-\Re\big( (b^{-1}(k-a) +i)^\gamma \big)\over  |y|}\\
    &\approx k^{-1}|y|^{1/\gamma}\left (\Re\big( (b^{-1}(1-ak^{-1}) +ik^{-1})^\gamma \big)-\Re\big( (b^{-1}(1-ak^{-1}) +ik^{-1})^\gamma \big)\right)\\
  &\approx k^{-2}|y|^{1/\gamma}.            
\end{split}
\end{equation*}
Inserting these inequalities into \eqref{e:consequence_t_k} and hence \eqref{e:widetilde_K_star_s,N_Case_3},
we get that
\begin{eqnarray*}
K^*_{s,N}(y)&\approx &(1+|y|)^{1/\gamma -1}+  N^2\sum_{k=1}^N { (t_{k-1}-t_k)(1+|y|)^{1/\gamma-1} t_k \over t_k^2+ (Nt_k+\log {|d|}-  \rho(y,t_k))^2 }\\
&\approx & 
(1+|y|)^{1/\gamma -1}+  N^2(1+|y|)^{1/\gamma-1} \sum_{k=1}^N  {(t_{k-1}-t_k) \over t_k (N-k)^2 }\\
&\approx & (1+|y|)^{1/\gamma -1}+  N^2(1+|y|)^{1/\gamma-1}\sum_{k=1}^{N-1}  {k^{-2}|y|^{1/\gamma} \over  k^{-1}|y|^{1/\gamma}(N-k)^2 }\\
&\approx & (1+|y|)^{1/\gamma -1}\big (1+ N^2\sum_{k=1}^{N-1}  { 1 \over  k(N-k)^2 }\big)\\
&\approx & N(1+|y|)^{1/\gamma -1},
\end{eqnarray*}
where in the second  $\approx$ we use  that  $-N\log2\leq \log |d|\leq  N\log 2,$
which follows, in turn, from the first two inequalities in \eqref{e:control_a_j_and_b_j}.

\noindent  {\bf Step 4:} {\it  There  are a   constant $\kappa>1$ independent of $N$
and a constant $c_N>1$
 such that 
for $\kappa^{-1}s \leq (1+|y|)^{1/\gamma}\leq \kappa s,$ we have 
$c^{-1}_N\leq {K^*_{s,N}(y)\over  (1+ | y|)^{1/\gamma -1}}\leq  c_N   .$}
  
  We use Lemma    \ref{lem_Poisson_kernel} in order to estimate $K^*_{s,N}(y).$
Since  this step is  much  easier than  Step 2 and Step 3, we leave it to the interested reader.

Putting   Step 1, 2, 3 and  4 altogether,  the proof of   Proposition \ref{P:key_interpretations}
is thereby completed.
\hfill $\square$

\subsection{End of the proof of
Proposition \ref{P:comparison_bis}.} We apply  what  has been done in this section 
to  $\B_{r^{1/N}}$  instead of   $    \B_{r^{1/N}|\log r|^{-3/N}}.$
Consequently, we obtain   quite similar estimates  as in Step 2, 3, 4 above.
This, combined with Lemma \ref{lem_estimate_G} for $r^{1/N}$ instead of $r$, yields a constant $c=c_N>0$ such  that 
$$
\| T\wedge [\Cf_{0,N}]\|_{\B_{r^{1/N}}}\leq c G(x_0,r^{1/N})\quad\text{for}\quad 0<r<1/2.
$$ 
Replacing $r^{1/N}$ by $r, $ the result follows. 
\hfill $\square$

\medskip

\noindent
Vi{\^e}t-Anh Nguy{\^e}n,  
Universit\'e de Lille 1, 
Laboratoire de math\'ematiques Paul Painlev\'e, 
CNRS U.M.R. 8524,  
59655 Villeneuve d'Ascq Cedex, 
France.\\
{\tt Viet-Anh.Nguyen@math.univ-lille1.fr},
{\tt http://math.univ-lille1.fr/$\sim$vnguyen}

\end{document}